\documentclass[reqno]{amsart}

\usepackage{graphics}

\usepackage{latexsym}
\usepackage[dvips]{epsfig}
\usepackage{color}
\usepackage{graphicx}
\usepackage{amsmath}
\usepackage{amssymb}
\usepackage{amsfonts}
\usepackage{eufrak}
\usepackage{amstext}
\usepackage{amsopn}
\usepackage{amsbsy}
\usepackage{amscd}
\usepackage{amsxtra}
\usepackage{amsthm}
\usepackage{bm}
\usepackage{CJK}

\newcommand{\R}{{\mathbb R}}

\newcommand{\beq}{\begin{equation}}
\newcommand{\eeq}{\end{equation}}
\newcommand{\ben}{\begin{eqnarray}}
\newcommand{\een}{\end{eqnarray}}
\newcommand{\beno}{\begin{eqnarray*}}
\newcommand{\eeno}{\end{eqnarray*}}

\newtheorem{thm}{Theorem}[section]

\newtheorem{lem}[thm]{Lemma}
\newtheorem{prop}[thm]{Proposition}
\newtheorem{coro}[thm]{Corollary}
\newtheorem{rmk}[thm]{Remark}

\allowdisplaybreaks

\begin{document}

\title[finite Morse index]{Finite Morse Index Implies Finite Ends}

\author[K. Wang]{ Kelei Wang}
 \address{\noindent K. Wang-
 School of Mathematics and Statistics, Wuhan University, Wuhan 430072, China.}
\email{wangkelei@whu.edu.cn}

\author[J. Wei]{Juncheng Wei}
\address{\noindent J. Wei -Department of Mathematics, University of British
Columbia, Vancouver, B.C., Canada, V6T 1Z2 }
\email{jcwei@math.ubc.ca}

\begin{abstract}
We prove that finite Morse index solutions to the Allen-Cahn equation in $\R^2$ have {\bf finitely many ends} and {\bf linear  energy growth}. The main tool is a {\bf curvature decay estimate} on level sets of these finite Morse index solutions, which in turn is reduced to a problem on the uniform second order regularity of clustering interfaces for the singularly perturbed Allen-Cahn equation in $\R^n$. Using an indirect blow-up technique, in the spirit of  the classical Colding-Minicozzi theory in minimal surfaces, we show    that  the {\bf  obstruction} to the uniform second order regularity  of clustering interfaces in $\R^n$ is associated to the existence of nontrivial entire solutions to a (finite or infinite)  {\bf Toda system} in $\R^{n-1}$.  For finite Morse index solutions in $\R^2$, we show that this obstruction does not exist by using information on stable solutions of the Toda system.
\end{abstract}



\maketitle

\date{}

\tableofcontents

\section{Introduction}
The intricate  connection between the Allen-Cahn equation and minimal surfaces is best  illustrated by the following famous De Giorgi's Conjecture \cite{DG}.

\medskip

\noindent
{\bf Conjecture.} {\em Let $ u \in C^2 (\R^n)$ be a solution to the Allen-Cahn equation
\begin{equation}
-\Delta u= u-u^3 \ \ \ \mbox{in} \ \R^n
\end{equation}
satisfying $ \partial_{x_n} u>0$. If $ n\leq 8$, all level sets $ \{ u=\lambda\}$ of $u$ must be hyperplanes.}

\medskip

In the last twenty years, great advances in De Giorgi's conjecture have been achieved,  having been fully
established in dimensions $n = 2$ by Ghoussoub and Gui \cite{GG} and for $n = 3$ by Ambrosio and
Cabre \cite{A-C}. A celebrated result by Savin \cite{Savin1} established its validity for $4 \leq n \leq 8$ under the following additional assumption
\begin{equation}
\label{extra}
 \lim_{x_n \to \pm \infty} u(x^\prime, x_n)=\pm 1.
 \end{equation}
On the other hand, Del Pino, Kowalczyk and Wei \cite{DKW} constructed a counterexample in dimensions $n\geq 9$.

\medskip

After the classification of monotone solutions, it is natural to consider {\bf stable  solutions}. Unfortunately this has been less successful. The arguments in \cite{A-C, Dancer, GG} imply that all stable solutions in $\R^2$ are one-dimensional.  On the other hand, Pacard and Wei \cite{PW} found a nontrivial stable solution in $\R^8$. (This is later shown to be also global minimizer \cite{LWW}.)

\medskip

In this paper we consider a more difficult  problem of classification of {\bf finite Morse index solutions} in $\R^2$.  Finite Morse index is a spectrum condition which is hard to use to obtain energy estimate. In the literature, another condition--{\bf finite-ended solutions}--is used. Roughly speaking a solution is called finite-ended if the number of components of the nodal set $ \{ u= 0\}$ is finite outside a ball. (In fact more restrictions are needed, see Gui \cite{Gui 1}.)    Analogous to the structure of minimal surfaces with finite Morse index (\cite{Fischer, Gulliver, Gulliver 2}),  a long standing conjecture is  that finite Morse index solutions to the Allen-Cahn equation in $\R^2$ have linear energy growth and hence finitely many ends (see \cite{Gui 1, DKW 2}). In this paper we will prove this conjecture by establishing a {\bf curvature decay estimate} on level sets of these finite Morse index solutions.

\medskip

This curvature estimate is similar to the one for stable minimal surfaces established by Schoen  in \cite{Schoen}. However, the key tool used in minimal surfaces is the so-called Simons type inequality \cite{Simons} which has no analogue  for semilinear elliptic equations. (The closest one may be the so-called Sternberg-Zumbrun inequality \cite{S-Z} for stable solutions.) Here an indirect blow up method will be employed in this paper. Our blow-up procedure  is inspired by the groundbreaking work of Colding and Minicozzi on the structure of limits of
sequences of embedded minimal surfaces of fixed genus in a ball in $\R^3$ (\cite{CM1, CM2, CM3, CM4, CM5, CM6}).

\medskip

We prove the curvature estimate by studying the uniform second order regularity of clustering interfaces in the singularly perturbed Allen-Cahn equation. It turns out the uniform second order regularity does not always hold true and the obstruction is associated to the existence of nontrivial entire solutions to the {\bf Toda system}
  \begin{equation}
  -\Delta f_\alpha = e^{ -\sqrt{2} (f_{\alpha+1}-f_\alpha)} - e^{ -\sqrt{2} (f_\alpha-f_{\alpha-1})}, \ \ \ \ \ \ \ \mbox{in} \ \R^{n-1}.
  \end{equation}

\medskip

  This connection between the Allen-Cahn equation and the Toda system was previously used in \cite{DKP, DKW 2, DKWY} to construct solutions to the Allen-Cahn equation with clustering interfaces. The analysis of clustering interfaces started in Hutchinson-Tonegawa \cite{H-T}. It is shown that the energy at the clustered interfaces is quantized. In \cite{Tonegawa, Tonegawa 2}, the convergence of clustering interfaces as well as regularity of their limit varifolds were studied. However, the uniform regularity of clustering interfaces (see \cite{Tonegawa 3}) and  precise behavior of the solutions near the interfaces and the connection to Toda system (except some special cases such as two end solutions in $\R^3$ studied in \cite{Gui-Liu-Wei}) are still missing.  In this paper we give precise second order estimates and show that when clustering interfaces appear, then a suitable rescaling of these interfaces converge to the graphs of a solution to the Toda system. It is through this blow up procedure we reduce the uniform second order regularity of interfaces to the non-existence of nontrivial entire stable solutions to the Toda system.  We also show that the stability condition is preserved in this blow up procedure. Then using results on stable solutions of the Toda system, we establish the uniform second order regularity of interfaces for stable solutions of the singularly perturbed Allen-Cahn equation, and then the curvature estimate for finite Morse index solutions in $\R^2$.

\medskip

For other related results on De Giorgi  conjecture for Allen-Cahn equation, we refer to \cite{AAC, Ca, FSV, FV1, FV2, FV3, GG2, JM, SSV, Wang} and the references therein.

\section{Main results}
\setcounter{equation}{0}

We consider general  Allen-Cahn equation
\begin{equation}\label{equation}
\Delta u=W^\prime(u), \quad |u|<1, \ \ \ \ \ \ \ \ \ \ \ \ \ \ \mbox{in} \ \R^n
\end{equation}

where $W(u)$ is a double well potential, that is, $W\in C^2([-1,1])$ satisfying
\begin{itemize}
\item $W>0$ in $(-1,1)$ and $W(\pm1)=0$;
\item  $W^\prime(\pm1)=0$ and $W^{\prime\prime}(-1)=W^{\prime\prime}(1)=2$;
\item there exists only one critical point of $W$ in $(-1,1)$, which is assumed to be $0$.
\end{itemize}
A typical model is given by $W(u)=(1-u^2)^2/4$.

\medskip

Under these assumptions on $W$, it is known that there exists a unique solution (up to a translation) to the following one dimensional problem
\begin{equation}\label{1d problem}
g^{\prime\prime}(t)=W^\prime(g(t)),  \ \ g(0)=0, \ \ \quad \lim_{t\to\pm\infty}g(t)=\pm 1.
\end{equation}

After a scaling $u_\varepsilon(x):=u(\varepsilon^{-1}x)$, we obtain the singularly perturbed version of the Allen-Cahn equation:
\begin{equation}\label{equation scaled}
\varepsilon\Delta u_\varepsilon=\frac{1}{\varepsilon}W^\prime(u_\varepsilon) \ \ \ \ \ \ \ \ \ \mbox{in} \ \R^n.
\end{equation}

\subsection{Finite Morse index solutions}

We say a solution $u\in C^2(\R^n)$ has finite Morse index if  there is a finite upper bound on its Morse index in any compact set. By \cite{dev}, this is equivalent to the condition that
$u$ is stable outside a compact set, that is, there is a compact set $K\subset\R^n$ such that
\[\mathcal{Q}(\varphi):=\int_{\R^n}|\nabla\varphi|^2+W^{\prime\prime}(u)\varphi^2\geq0, \quad \forall \varphi\in C_0^\infty(\R^n\setminus K).\]

Our first main result is
\begin{thm}\label{main result 1}
Suppose $u$ is a finite Morse index solution of \eqref{equation} in $\R^2$. Then there exist finitely many distinct rays $L_i:=\{x_i+te_i, t\in[0,+\infty)\}$, where $x_i, e_i\in \R^2$, $1\leq i\leq N$ and $e_i\neq e_j$ for $i\neq j$, such that outside a compact set
$\{u=0\}$ has exactly $N$ connected components which are exponentially close to $L_i$ respectively. Moreover, $u$ has linear energy growth, i.e., there exists a constant $C$ such that
\begin{equation}
\label{lineargrowth}
\int_{B_R(0)} \left[\frac{1}{2}|\nabla u|^2+W(u)\right] \leq CR, \quad \forall R\geq 1.
\end{equation}
\end{thm}

\
\medskip

A component of $\{u=0\}$ is called an {\bf end} of $u$. In $\R^2$ we know that stable solutions (Morse index $0$) are one dimensional, i.e. after rigid motions in $\R^2$, $u(x_1,x_2)\equiv g(x_2)$.
In the above terminology, this solution has $2$ ends. All two-ended solutions are one-dimensional and the number of ends must be even. Near each end the solution approaches to the one-dimensional profile exponentially,   see Del Pino-Kowalczyk-Pacard \cite{DKP}, Gui \cite{Gui} and  Kowalczyk-Liu-Pacard  \cite{KLP 1, KLP 2, KLP 3}.  The existence of multiple-ended solutions and infinite-ended solutions to Allen-Cahn equation in $\R^2$ have been constructed in \cite{AM, DKW 3, KLW, KLPW}. The structure and classification of four end solutions have been studied extensively in \cite{Gui 1, DKP, KLP 1, KLP 2, KLP 3}.  It is shown that the four-ended solutions have even symmetries and the moduli space of  four-ended solutions is one-dimensional.

\medskip

As a byproduct of our analysis, for solutions with Morse index $1$ we can show that
\begin{thm}\label{main result 2}
Any solution to \eqref{equation} in $\R^2$ with Morse index $1$ has four ends.
\end{thm}

\medskip

\noindent
{\bf Remark:} The linear growth condition (\ref{lineargrowth}) implies that the nodal set $ \{ u=0 \}$ has finite length at $\infty$.  In $ \R^n$ the analogue energy bound
\begin{equation}
\label{lineargrowth1}
\int_{B_R(0)} \left[\frac{1}{2}|\nabla u|^2+W(u)\right] \leq C \ R^{n-1}, \quad \forall R\geq 1
\end{equation}
 is a classical assumption in the setting
of semilinear elliptic equations (see e.g. Hutchinson-Tonegawa \cite{H-T}).  It is satisfied by minimizers or monotone solutions satisfying  \eqref{extra}. This is precisely the use of \eqref{extra} in Savin's proof of De Giorgi's conjecture (\cite{Savin1}). (See also Ambrosio-Cabre \cite{A-C} and Alberti-Ambrosio-Cabre \cite{AAC}.) In dimensions $4$ and $5$, condition (\ref{lineargrowth1}) is also  an essential  estimate in Ghoussoub-Gui \cite{GG2}.
A similar area bound for minimal hypersurfaces seems to be also crucial for the study of {\em Stable Bernstein Conjecture} when the dimension is larger than $3$. (Only three dimension case has been solved in \cite{Carmo-Peng, F-Schoen}. See also  \cite{Cao-S-Z, Li-Wang, Li-Wang 2}.)

\medskip

The main tool to prove Theorem \ref{main result 1} is the following curvature estimate on level sets of $u$ (see Theorem \ref{curvature decay} and Theorem \ref{quadratic curvature decay} below):

\medskip

\noindent
{\bf Key Curvature Estimates (Theorem \ref{curvature decay}):} For any solution of \eqref{equation} in $\R^2$ with finite Morse index and $ b \in (0,1)$, there exist a constant $C$ and $ R=R(b)$ such that
\begin{equation}
|B(u)(x) | \leq \frac{C}{|x|} \ \ \mbox{for} \  x \in \{ | u| \leq 1-b \} \cap (B_{R(b)} (0))^c, \ \ \ \ \mbox{where} \  B(u)(x)= \sqrt{ \frac{|\nabla^2u(x)|^2-|\nabla|\nabla u(x)||^2}{|\nabla u(x)|^2} }.
\end{equation}

\medskip

This curvature decay is similar to Schoen's curvature estimate for stable minimal surfaces \cite{Schoen}, however the proof is quite different. This is mainly due to the lack of a suitable Simons type inequality for semilinear elliptic equations. Hence an indirect method is employed, by introducing a blow up procedure and reducing the curvature decay estimate to a second order estimate on interfaces of solutions to \eqref{equation scaled}, see Theorem \ref{second order estimate} below.

\subsection{Second order estimates on interfaces}

It turns out that our analysis on the uniform   second order regularity of level sets of solutions to the singularly perturbed Allen-Cahn equation \eqref{equation scaled} works in a more general setting and any dimension $n\geq 2$.
In Part II of this paper we give precise analysis in the case of clustering interfaces. More precisely we assume that
 \begin{itemize}
\item [{\bf (H1)}]$u_\varepsilon$ is a sequence of solutions to \eqref{equation scaled} in $\mathcal{C}_2=B_2^{n-1}\times(-1,1)\subset\R^n$;
 \item  [{\bf (H2)}]there exists $Q\in\mathbb{N}$, $b\in(0,1)$ and $t_\varepsilon\in(-1+b,1-b)$ such that    $\{u_\varepsilon=t_\varepsilon\}$
consists of $Q$ connected components
\[\Gamma_{\alpha,\varepsilon}=\{x_n=f_{\alpha,\varepsilon}(x^\prime), \quad x^\prime:=(x_1,\cdots, x_n)\in B_2^{n-1}\}, \ \ \ \alpha=1,\cdots, Q,\]
where $-1/2<f_{1,\varepsilon}<f_{2,\varepsilon}<\cdots< f_{Q,\varepsilon}<1/2$;
\item  [{\bf (H3)}]for each $\alpha$, $f_{\alpha,\varepsilon}$ are uniformly bounded in $Lip(B_2^{n-1})$ and they converge to the same limit $f_\infty$ in $C_{loc}(B_2^{n-1})$.
\end{itemize}

\medskip

Here $Q$ is called the multiplicity of the interfaces. 
Analyzing clustering interfaces is one of main difficulties in the study of singularly perturbed Allen-Cahn equations.  See e.g.  \cite{H-T, Tonegawa, Tonegawa 2, Tonegawa 3}. In particular, it is not known if flatness implies uniform $C^{1,\theta}$ regularity when there are clustering interfaces (i.e. the Lipschitz regularity in the above hypothesis {\bf(H3)}).

\medskip

Under these assumptions, it can be shown that  $f_\infty$ satisfies the minimal surface equation (see \cite{H-T})
\begin{equation}
\mbox{div}\left(\frac{\nabla f_\infty}{\sqrt{1+|\nabla f_\infty|^2}}\right)=0 \ \ \ \ \mbox{in} \ \ \R^{n-1}.
\end{equation}
Because $f_\infty$ is Lipschitz, by standard elliptic estimates on the minimal surface equation \cite[Chapter 16]{GT}, $f_\infty\in C^\infty_{loc}(B_2^{n-1})$.

We want to study whether $f_{\alpha,\varepsilon}$ converges to $f_\infty$ in $C^2_{loc}(B_2^{n-1})$. It turns out this may not be true and the obstruction is related to a Toda system
\begin{equation}\label{Toda entire}
\Delta f_\alpha(x^\prime)=A_1e^{-\sqrt{2}\left(f_\alpha(x^\prime)-f_{\alpha-1}(x^\prime)\right)}
-A_2e^{-\sqrt{2}\left(f_{\alpha+1}(x^\prime)-f_\alpha(x^\prime)\right)}, \quad x^\prime\in\R^{n-1}, \  \ 1\leq \alpha\leq Q^\prime,
\end{equation}
where $Q^\prime\leq Q$, $A_1$ and $A_2$ are positive constants.

\medskip

More precisely, we show
\begin{thm}\label{main result 3}
If  $f_{\alpha,\varepsilon}$ does not converge  to $f_\infty$ in $C^2_{loc}(B_2^{n-1})$, then a suitable rescaling of them converge to a nontrivial entire solution to the Toda system \eqref{Toda entire}.
\end{thm}

\medskip

\noindent
{\bf Remark:} Conversely,  the existence of multiplicity two ($Q=2$) solutions using Toda system  has been carried out in \cite{ADW1, ADW2}.

\medskip

For the multiplicity one case $Q=1$,  we get the following uniform $C^{2,\theta}$ estimate.
\begin{thm}\label{main result 4}
If $\{u_\varepsilon=0\}=\{x_n=f_\varepsilon(x_1,\cdots, x_{n-1})\}$,  then for any $\theta\in(0,1)$, $f_\varepsilon$ are uniformly bounded in $C^{2,\theta}_{loc}(B_2^{n-1})$.
\end{thm}

\medskip

These results answer partly a question of Tonegawa \cite{Tonegawa 3} and improves the uniform $C^{1,\theta}$ estimate in Caffarelli-Cordoba \cite{CC} to the second order $C^{2,\theta}$  estimate.

\medskip

The main idea in the proof of these two theorems relies on the determination of the {\bf interaction} between $\Gamma_{\alpha,\varepsilon}$. To this aim, we introduce the Fermi coordinates with respect to $\Gamma_{\alpha,\varepsilon}$
and near each $\Gamma_{\alpha,\varepsilon}$ we find the optimal approximation of $u_\varepsilon$ along the normal direction using the one dimensional profile $g$ and the distance to $\Gamma_{\alpha,\varepsilon}$. More precisely, we use an approximate solution in the form
\[g\left(\frac{\mbox{dist}_{\Gamma_{\alpha,\varepsilon}}-h_{\alpha,\varepsilon}}{\varepsilon}\right).\]
Here $h_{\alpha,\varepsilon}$ is introduced to make sure that this is the optimal approximation along the normal direction with respect to $\Gamma_{\alpha,\varepsilon}$. With this construction, using the nondegeneracy of $g$, we can get a good estimate on the error between $u_\varepsilon$ and these approximate solutions, which in turn shows that the interaction between $\Gamma_{\alpha,\varepsilon}$ is exactly through the Toda system
\[\Delta f_{\alpha,\varepsilon}=\frac{A_1}{\varepsilon}e^{-\sqrt{2}\frac{f_{\alpha,\varepsilon}-f_{\alpha-1,\varepsilon}}{\varepsilon}}
-\frac{A_2}{\varepsilon}e^{-\sqrt{2}\frac{f_{\alpha+1,\varepsilon}-f_{\alpha,\varepsilon}}{\varepsilon}}+\mbox{high order terms}.\]
Using this representation, we show that the uniform second order regularity of $f_{\alpha,\varepsilon}$ does not hold only if the lower bound of intermediate distances between $\Gamma_{\alpha,\varepsilon}$ is of the order
\begin{equation}\label{lower bound}
\frac{\sqrt{2}}{2}\varepsilon|\log\varepsilon|+O(\varepsilon).
\end{equation}
(Here the constant $\sqrt{2}=\sqrt{W^{\prime\prime}(1)}$.) Moreover, if this is the case, the rescalings
\[\widetilde{f}_{\alpha,\varepsilon}(x^\prime):=\frac{1}{\varepsilon}f_{\alpha,\varepsilon}\left(\varepsilon^{\frac{1}{2}}x^\prime\right)-\frac{\sqrt{2}\alpha}{2}|\log\varepsilon|\]
converges to a solution of \eqref{Toda entire}.

\medskip

In other words, if intermediate distances between $\Gamma_{\alpha,\varepsilon}$ are large (compared with \eqref{lower bound}), the interaction between different interfaces is so weak enough that it does not affect the second order regularity of $f_{\alpha,\varepsilon}$. In particular, if there is only one component and hence no interaction between different components, we get Theorem \ref{main result 4}.

\medskip

In Theorem \ref{second order estimate}, the situation is a little different where more and more connected components of $\{u_\varepsilon=0\}$ could appear. However, the above discussion still applies. This is because, by using the stability condition, we can get an explicit lower bound on intermediate distance between different components of $\{u_\varepsilon=0\}$ which is just a little smaller than \eqref{lower bound}. To get a lower bound higher than \eqref{lower bound}, we use the stability of $f_{\alpha,\varepsilon}$ (as a solution to the approximate Toda system)  inherited from $u_\varepsilon$. By this stability and a classical estimate of Choi-Schoen \cite{Choi-Schoen}, we get a decay estimate of $e^{-\sqrt{2}\frac{f_{\alpha,\varepsilon}-f_{\alpha-1,\varepsilon}}{\varepsilon}}$ in the interior. In some sense $e^{-\sqrt{2}\frac{f_{\alpha,\varepsilon}-f_{\alpha-1,\varepsilon}}{\varepsilon}}$ replaces the role of the curvature in minimal surface theory.

\medskip

We also would like to call readers' attention to the resemblance of pictures here (especially when we consider $\R^3$ and not only $\R^2$) with the multi-valued graph construction in seminal Colding-Minicozzi theory \cite{CM1, CM2, CM3, CM4, Colding}. When the number of connected components of $\{u_\varepsilon=0\}$ goes to infinity and we do not assume the stability condition, the blow up procedure as in Theorem \ref{main result 3} produces a solution to the Toda lattice (i.e. in \eqref{Toda entire} the index $\alpha$ runs over integers $\mathbb{Z}$). The difference is that, different sheets of minimal surfaces do not interact (in other words, interact only when they touch) while different sheets of interfaces in the Allen-Cahn equation have an exponential interaction. It is this exponential interaction leading to the Toda system. We notice that in a recent paper \cite{JK}, Jerison and Kambrunov also performed a similar blow-up procedure for the one-phase free boundary problem in $\R^2$. The difference is again that different sheets of nodal sets of free boundary  do not interact.

\medskip

\noindent
{\bf Organization of the paper.} This paper is divided into three parts. Part I is devoted to the analysis of finite Morse index solutions, by assuming the curvature decay estimate. In Part II we study the second order regularity of interfaces and prove Theorem \ref{main result 3} and Theorem \ref{main result 4}. Techniques in Part II are modified in Part III to prove  the curvature decay estimate needed in Part I. Some technical calculations in Part II are collected in the Appendix.

\medskip

\noindent
\textit{Acknowledgement.} The research of J. Wei is partially supported by
NSERC of Canada and the Cheung-Kong Chair Professorship.  K. Wang is supported by NSFC no. 11631011 and \textquotedblleft
the Fundamental Research Funds for the Central Universities". We thank Professor Changfeng Gui for useful discussions. K. Wang is also grateful to Yong Liu for several enlightening discussions.

\part{Finite Morse index solutions}

In this part we study finite Morse index solutions of \eqref{equation} in $\R^2$ and prove Theorem \ref{main result 1} and Theorem \ref{main result 2}, by assuming the curvature decay estimates Theorem \ref{curvature decay} and Theorem \ref{quadratic curvature decay}, which is based on Theorem \ref{second order estimate} whose proof is given in Part III.
Throughout this section we always assume that $n=2$ and that $u$ is a finite Morse index solution to \eqref{equation}.

\section{Curvature decay}
\setcounter{equation}{0}

The following characterization of stable solutions is well known (see for example \cite{A-C, Dancer, GG}).
\begin{thm}\label{thm stable solution}
Let $u$ be a stable solution of \eqref{equation} in $\mathbb{R}^2$, then there exists a unit vector $\xi\in\R^2$ and $t\in\R$ such that
\[u(x)\equiv g(x\cdot\xi-t), \quad \forall x\in\R^2.\]
\end{thm}
Since $u$ has finite Morse index, $u$ is stable outside a compact set. As a consequence we then obtain
\begin{lem}\label{lem 1.2.1}
For any $b\in(0,1)$, there exist $c(b)>0$ and $R(b)>0$ such that,
 for $x\in\{|u|\leq 1-b\}\setminus B_{R(b)}(0)$,
 \[|\nabla u(x)|\geq c(b).\]
\end{lem}
\begin{proof}
If the claim were false, there would exist a sequence of $x_i\in\{|u|\leq 1-b\}$, $x_i\to\infty$, but
\begin{equation}\label{3.01}
|\nabla u(x_i)|\to 0.
\end{equation}
Let $u_i(x):=u(x_i+x)$. By standard elliptic estimates and the Arzela-Ascoli theorem, up to a subsequence, $u_i$ converges to a limit $u_\infty$
in $C^2_{loc}(\mathbb{R}^2)$. Because $u$ is stable outside a compact set, $u_\infty$ is stable in $\mathbb{R}^2$. Then by Theorem \ref{thm stable solution}, $u_\infty$ is one dimensional. In particular, $|\nabla u_\infty|\neq 0$ everywhere. However, by passing to the limit in \eqref{3.01}, we get
\[|\nabla u_\infty(0)|=\lim_{i\to+\infty}|\nabla u_i(0)|=\lim_{i\to+\infty}|\nabla u(x_i)|=0.\]
This is a contradiction.
\end{proof}
The proof also shows that $u$ is close to  one dimensional solutions at infinity.

\medskip

The following lemma shows that the nodal set $\{u=0\}$ cannot be contained in any bounded set.
\begin{lem}
For any solution of \eqref{equation} in $\mathbb{R}^2$ with finite Morse index, if $u$ is not constant, then $\{u=0\}$ is unbounded.
\end{lem}
\begin{proof}
Assume by the contradiction, $u>0$ outside a ball $B_R(0)$. By Lemma \ref{lem 1.2.1} and the fact that the only positive solution to the one dimensional Allen-Cahn equation is the constant function $1$, we see $u(x)\to1$ uniformly as $|x|\to+\infty$. For $u$ near $1$, $W^{\prime}(u)\leq -c(1-u)$ for some positive constant $c>0$. Thus by comparison principle  we get two constants $C$ and $R$ such that, for any $x\in B_{R}^c$,
\begin{equation}\label{exponential decay 1}
u(x)\geq1-Ce^{-\frac{|x|-R}{C}}.
\end{equation}
Then by standard elliptic estimates or Modica's estimates (\cite{Modica}),
\begin{equation}\label{exponential decay 2}
|\nabla u(x)|\leq  \sqrt{ 2 W(u)} \leq Ce^{-\frac{|x|-R}{C}}.
\end{equation}

Recall that the Pohozaev type equality on ball $B_r(0)$ is
\[\int_{B_r}2W(u)=r\int_{\partial B_r}\frac{1}{2}|\nabla u|^2+W(u)-\left(\frac{\partial u}{\partial r}\right)^2.\]
For $r>R$, substituting \eqref{exponential decay 1} and \eqref{exponential decay 2} into the right hand side,
we get
\[\int_{B_r}2W(u)\leq Cre^{-\frac{r-R}{C}}.\]
Letting $r\to+\infty$ leads to
\[\int_{\mathbb{R}^2}W(u)=0.\]
Hence either $u\equiv1$ or $u\equiv-1$.
\end{proof}

If $|\nabla u(x)|\neq 0$, denote
\begin{equation}
\nu(x):=\frac{\nabla u(x)}{|\nabla u(x)|}, \quad \mbox{and} \quad B(u)(x)=\nabla\nu(x).
\end{equation}

If $|\nabla u(x)|\neq 0$, locally $\{u=u(x)\}$ is a $C^2$ curve, thus its curvature $H$ is well defined. Then
\begin{equation}
|B(u)(x)|^2=\frac{|\nabla^2u(x)|^2-|\nabla^2u(x)\cdot\nu(x)|^2}{|\nabla u(x)|^2}=H(x)^2+|\nabla_T\log|\nabla u(x)||^2,
\end{equation}
where $\nabla_T$ is the tangential derivative along the level set of $u$, see \cite{S-Z, Tonegawa}.

\begin{coro}\label{coro 3.2}
For any $b\in(0,1)$, $|B(u)(x)|^2$ is bounded in $(\mathbb{R}^2\setminus B_{R(b)}(0))\cap\{|u|\leq 1-b\}$. Moreover, for $x\in(\mathbb{R}^2\setminus B_{R(b)}(0))\cap\{|u|\leq 1-b\}$, if
$x\to\infty$, $|B(u)(x)|^2\to 0$.
\end{coro}
\begin{proof}
The first claim follows from the fact that $|\nabla^2 u|^2$
is bounded in $\mathbb{R}^2$ and the lower bound on $|\nabla u|$ in Lemma \ref{lem 1.2.1}.

The second claim also follows from Lemma \ref{lem 1.2.1}, by noting that for one dimensional solutions $|B(g)|\equiv 0$.
\end{proof}

Now we give the following key estimate on the  decay rate on $|B(u)(x)|$ at infinity.
\begin{thm}\label{curvature decay}
For any solution of \eqref{equation} in $\R^2$ with finite Morse index and $b\in(0,1)$,
there exists a constant $C$ such that
\[|B(u)(x)|\leq \frac{C}{|x|}, \ \ \ \ \mbox{for}\ x\in \{|u|\leq 1-b\}\cap B_{R(b)}(0)^c.\]
\end{thm}

To prove this theorem, we argue by contradiction.
Take $\mathcal{X}$ to be the complete metric space $\{|u|\leq 1-b\}$
with the extrinsic distance and $\Gamma:=\mathcal{X}\cap
\overline{B_{R(b)}(0)}$.
 Assume there exists a sequence of
$X_k\in \mathcal{X}\setminus\Gamma$,
$|B(X_k)|\text{dist}(X_k,\Gamma)\geq 2k$. By the doubling lemma in \cite{Polacik-Q-S}, there exist $Y_k\in \mathcal{X}\setminus\Gamma$
such that
\[|B(Y_k)|\geq|B(X_k)|, \ \ \ \ |B(Y_k)|\text{dist}(Y_k,\Gamma)\geq 2k,\]
\[|B(Z)|\leq 2|B(Y_k)|\ \ \ \ \ \text{for}\ \ Z\in B_{k|B(Y_k)|^{-1}}(Y_k).\]
Let $\varepsilon_k:=|B(Y_k)|$ and define
\[u_k(x):=u(y_k+\varepsilon_k^{-1}x).\]
Note that
\begin{equation}\label{1.2.1}
\text{dist}(Y_k,\Gamma)\geq 2k|B(Y_k)|^{-1}.
\end{equation}
By Corollary \ref{coro 3.2}, $|Y_k|\to+\infty$ and $\varepsilon_k\to 0$.

In $B_k(0)$, $u_k$ is a solution of \eqref{equation scaled}
with the parameter $\varepsilon_k$. By \eqref{1.2.1}, $u_k$ is stable in $B_k(0)$.

 For $X\in B_k(0)\cap\{|u_k|<1-b\}$,
\begin{equation}\label{curvature rescaling bound}
|B(u_k)(X)|\leq 2.
\end{equation}
On the other hand, by the above construction we have
\begin{equation}\label{curvature rescaling normalization}
|B(u_k)(0)|=1.
\end{equation}
The bound on $|B_k|$ implies that, for any
$X\in \{|u_k|<1-b\}\cap B_k(0)$, $\{u_k=u_k(X)\}\cap B_{1/8}(X)$
can be represented by the graph of a function with a uniform
$C^{1,1}$ bound.

After a rotation, assume that
the connected component of $\{u_k=u_k(0)\}\cap B_{1/8}(0)$ passing through
$0$ (denoted by $\Sigma_k$) is represented by the graph
$\{x_2=f_k(x_1)\}$, where $f_k(0)=f_k^\prime(0)=0$. By the curvature
bound \eqref{curvature rescaling bound},
\begin{equation}\label{C1,1 bound}
\sup_{[-1/8,1/8]}|f_k^{\prime\prime}|\leq 32.
\end{equation}
By these bound, after passing to a subsequence, we can assume $f_k$ converges to $f_\infty$ in $C^1([-1/8,1/8])$.

There are two cases.
\begin{itemize}
\item {\bf Case 1.} $\{x_2=f_k(x_1)\}$ is an isolated component of $\{u_k=u_k(0)\}$. In other words, there exists an $h>0$ independent of $k$ such that $\{u_k=u_k(0)\}\cap B_h(0)=\{x_2=f_k(x_1)\}$.
\item {\bf Case 2.} There exists a sequence of points on a component of $\{u_k=u_k(0)\}\cap B_{1/8}(0)$ other than $\{x_2=f_k(x_1)\}$,  converging to a point on $\{x_2=f_\infty(x_1)\}$.
\end{itemize}

The following simple lemma can be proved by combining the curvature bound \eqref{curvature rescaling bound} with the fact that different connected components of $\{u_k=u_k(0)\}$ are disjoint. (This fact   has been used a lot in minimal surface theory, in particular, in Colding-Minicozzi   \cite{CM2}.)
\begin{lem}\label{lem graph construction}
There exist two universal constants $h$ and $C(h)$ such that if a connected component $\Gamma$ of $\{u_k=u_k(0)\}\cap B_{1/8}(0)$ intersects $B_h(0)$, then $\Gamma\cap B_{2h}(0)$ can be represented by the graph $\{x_2=f(x_1)\}$, where $\|f\|_{C^{1,1}([-2h,2h])}\leq C(h)$.
\end{lem}

Changing the notation slightly,  there exist a sequence of solutions $u_\varepsilon$ in $\mathcal{C}_1$ ($\mathcal{C}_1$ denotes the cylinder $(-1,1)\times(-1,1)$) with its nodal sets given by
\[\cup_\alpha\{x_2=f_{\alpha,\varepsilon}(x_1)\}.\]
Moreover, in $\{|u_\varepsilon|<1-b\}$, $|\nabla u_\varepsilon|>0$ and $|B(u_\varepsilon)|\leq 2$. In particular, $|f_{\alpha,\varepsilon}^{\prime\prime}(x_1)|\leq 64$ for every $\alpha$ and $x_1\in(-1,1)$. After a rigid motion we may assume $f_{0,\varepsilon}(0)=f_{0,\varepsilon}^\prime(0)=0$ and $|f_{0,\varepsilon}^{\prime\prime}(0)|=1$.
Here the cardinality of the index set $\alpha$ could go to infinite.

\medskip

The following theorem leads to a contradiction with \eqref{curvature rescaling normalization} and the proof of Theorem \ref{curvature decay} is thus finished.
\begin{thm}\label{second order estimate}
Suppose $u_\varepsilon$ is a sequence of stable solutions to \eqref{equation scaled} in $B_1(0)$ satisfying for some constant $b\in(0,1)$ and $C>0$ independent of $\varepsilon$,
\[|B(u_\varepsilon)|\leq C, \quad \mbox{in } \{|u_\varepsilon|<1-b\}\cap B_1(0).\]
  Then for all $\varepsilon$ small,
\[\sup_{\{|u_\varepsilon|<1-b\}\cap B_{1/2}(0)}|B(u_{\varepsilon})|\leq C\varepsilon^{1/7}.\]
\end{thm}
The proof will be postponed to Part III.

\medskip
In the above we do not use the full power of Theorem \ref{second order estimate}.
In fact, we can improve Theorem \ref{curvature decay} to a higher order decay rate.
\begin{thm}\label{quadratic curvature decay}
There exists a constant $C$  such that
\[|B(u)(x)|\leq \frac{C}{|x|^{8/7}}, \ \ \ \ \mbox{for}\ x\in \{|u|<1-b\}\cap B_{R(b)}(0)^c.\]
\end{thm}
\begin{proof}
Take an arbitrary sequence $X_k\in\{|u|<1-b\}\to\infty$.
Denote $\varepsilon_k:=|B(u)(X_k)|$, which converges to $0$ as $k\to\infty$. Let
\[u_k(x):=u(X_k+\varepsilon_k^{-1}x),\]
which is a solution of \eqref{equation scaled} with parameter $\varepsilon_k$.

By Theorem \ref{curvature decay}, $\varepsilon_k|X_k|\leq C$ and for any $x\in \{|u|<1-b\}\cap B_{|X_k|/2}(X_k)$,
\[|B(u)(x)|\leq\frac{C}{|x|}\leq\frac{2C}{|X_k|}\leq 2C\varepsilon_k.\]
Thus after a scaling, there exists a constant $\rho\in(0,1/2)$ independent of $k$ such that  in $B_\rho(0)$, $u_k$ satisfies the assumptions of Theorem \ref{second order estimate}.
Note that for all $k$ large, $u$ is stable in $B_{|X_k|/2}(X_k)$. Hence $u_k$ is stable in $B_\rho(0)$.
Applying Theorem \ref{second order estimate} gives  $|B(u_k)(0)|\leq C\varepsilon_k^{1/7}$. Rescaling back we get the desired bound on $|B(u)(X_k)|$.
\end{proof}

\section{Lipschitz regularity of nodal sets at infinity}
\setcounter{equation}{0}

First using Theorem \ref{quadratic curvature decay} and proceeding as in \cite{Wang 2}, we can show that there are at most finitely many connected components of $\{u=0\}$. This is achieved by choosing the smallest ball centered at the origin which contains a bounded connected component of $\{u=0\}$ and comparing their curvatures at the contact point.

\medskip

In the following we take a constant $R_0>R(1/2)$ so that $u$ is stable outside $B_{R_0}(0)$.
We first give a chord-arc bound on $\{u=0\}$.
\begin{lem}
Let $\Sigma$ be a unbounded connected component of $\{u=0\}\setminus B_{R_0}(0)$ and $X(t)$ be an arc length parametrization of $\Sigma$, where $t\in[0,+\infty)$. Then there exists a constant $c$ such that for any $t$ large,
\[|X(t)|\geq c|t|.\]
\end{lem}
\begin{proof}
Because $\Sigma$ is a smooth embedded curve diffeomorphic to $[0,+\infty)$, if $t\to+\infty$, $|X(t)|\to+\infty$.

By direct differentiation and applying Theorem \ref{quadratic curvature decay}, we obtain
\begin{eqnarray*}
\frac{d^2}{dt^2}\big|X(t)\big|^2=2\Big|\frac{dX}{dt}\Big|^2+2X(t)\cdot\frac{d^2 X}{dt^2}(t)
&\geq&2-\frac{C}{|X(t)|^{1/8}}\\
&\geq&1,
\end{eqnarray*}
for all $t$ large.
Integrating this differential inequality we finish the proof.
\end{proof}

Keeping assumptions as in this lemma, we can further show that
\begin{prop}\label{Lip regularity at infinity}
The limit
\[e_\infty:=\lim_{t\to+\infty}X^\prime(t)\]
exists. Moreover, for all $t$ large,
\[|X^\prime(t)-e_\infty|\leq\frac{C}{t^{1/7}}.\]
\end{prop}
\begin{proof}
Combining the previous lemma with  Theorem \ref{quadratic curvature decay} we obtain
\[|X^{\prime\prime}(t)|\leq\frac{C}{t^{8/7}}.\]
Integrating this in $t$ we finish the proof.
\end{proof}

The direction $e_\infty$ obtained in this proposition is called the limit direction of the connected component $\Sigma$.

\begin{rmk}\label{rmk 1.3.3}
If there exists a $\sigma>0$ such that
\[|X^\prime(t)-e_\infty|\leq\frac{C}{t^{1+\sigma}},\]
by using the fact that translations of $u$ at infinity is close to the one dimensional profile (see Lemma \ref{lem 1.2.1}), we can show that
different connected components of $\{u=0\}$ have different limit directions. However, this would require a higher order decay rate in Theorem \ref{quadratic curvature decay} as well as in Theorem \ref{second order estimate}, which we will not pursue in this paper.
\end{rmk}

\section{Energy growth bound: Proof of Theorem \ref{main result 1}}
\setcounter{equation}{0}

First using the stability of $u$  outside $B_{R_0}(0)$, we study the structure of nodal set of direction derivatives of $u$ at infinity. The following method can be compared with those in \cite{Dancer, Savin-V}.
\begin{prop}\label{finiteness of nodal domains II}
For any unit vector $e$, every connected component of $\{u_e:=e\cdot\nabla u\neq0\}$ intersects with $B_{R_0}(0)$.
\end{prop}
\begin{proof}
Assume by the contrary there exists a unit vector $e$ and a connected component $\Omega$ of $\{u_e\neq0\}$ contained in $B_{R_0}(0)^c$.
Let $\psi$ be the restriction of $|u_e|$ to $\Omega$, with zero
extension outside it. Hence $\psi$ is continuous, and in
$\Omega$ it satisfies the linearized equation
\begin{equation}\label{linearized equation}
\Delta\psi=W^{\prime\prime}(u)\psi.
\end{equation}

 For any $R>R_0$,
let \makeatletter
\let\@@@alph\@alph
\def\@alph#1{\ifcase#1\or \or $'$\or $''$\fi}\makeatother
\begin{equation*}
{\eta_R(x):=}
\begin{cases}
1, &x\in B_R(0), \\
2-\frac{\log |x|}{\log R}, &x\in B_{R^2}(0)\setminus B_R(0),\\
0, &x\in B_{R^2}(0)^c.
\end{cases}
\end{equation*}
\makeatletter\let\@alph\@@@alph\makeatother

Multiplying \eqref{linearized equation} by $\psi\eta_R^2$ and integrating by parts
leads to
\begin{equation}\label{1.4.2}
\int_{\R^2}|\nabla\left(\psi\eta_R\right)|^2+W^{\prime\prime}(u)
\left(\psi\eta_R\right)^2=\int_{\R^2}\psi^2|\nabla\eta_R|^2\leq\frac{C}{\log
R},
\end{equation}
where we have used the fact that $|\psi|\leq|\nabla u|\leq C$.

Take an $X\in\partial\Omega$ such that $\partial\Omega$ is smooth in a neighborhood of $X$. By a suitable compact modification of  $\psi$ in a small ball $B_h(X)$, we get a new function $\tilde{\psi}$ and a constant $\delta>0$ so that
\begin{equation}\label{1.4.3}
\int_{B_{h}(X)}\frac{1}{2}|\nabla\tilde{\psi}|^2+W^{\prime\prime}(u)\tilde{\psi}^2\leq\left[
\int_{B_{h}(X)}\frac{1}{2}|\nabla\psi_i|^2+W^{\prime\prime}(u)\psi^2\right]-\delta.
\end{equation}

Combining \eqref{1.4.2} and \eqref{1.4.3} we get an $R$ such that
\[
\int_{\R^2}|\nabla\left(\tilde{\psi}\eta_R\right)|^2+W^{\prime\prime}(u)
\left(\tilde{\psi}\eta_R\right)^2<0.
\]
This is a contradiction with the stability condition of $u$ outside $B_{R_0}(0)$.
\end{proof}

The following finiteness result on the ends of $u$ can be proved by the same method in \cite{Wang 2}, using Proposition \ref{finiteness of nodal domains II} and Proposition \ref{Lip regularity at infinity}.
\begin{prop}
By taking a large enough $R_1>0$, there are only finitely many connected components of $\{u=0\}\cap B_{R_1}(0)^c$.
\end{prop}
The main idea is as follows.
 \begin{itemize}
 \item[(i)] By choosing a generic direction $e$, using Proposition \ref{Lip regularity at infinity} we can show that for each connected component of $\{u=0\}\cap B_{R_1}(0)^c$, $u_e$ has fixed sign in an $O(1)$ neighborhood of it.
     \item[(ii)] If two connected components of  $\{u=0\}\cap B_{R_1}(0)^c$ are neighboring and the angle between their limit directions are small, $u_e$ has different sign near these two connected components.
     \item[(iii)] If there are too many connected  components of $\{u=0\}\cap B_{R_1}(0)^c$, we can construct as many connected components of $\{u_e\neq0\}\cap B_{R_1}(0)^c$ as we want. On the other hand, by Proposition \ref{finiteness of nodal domains II}, the number of connected components of $\{u_e\neq0\}\cap B_{R_1}(0)^c$ is controlled by the number of connected components of $\{u_e\neq0\}\cap \partial B_{R_1}(0)$. This leads to a contradiction.
\end{itemize}

With this proposition in hand, we can proceed as in \cite{Gui, Wang 2} to obtain the linear energy growth bound in Theorem \ref{main result 1}. The main idea is to divide $\R^2\setminus B_{R_0}(0)$ into a number of cones with their angles strictly smaller than $\pi$ and $\{u=0\}$ is strictly contained in the interior of these cones, and then apply the Hamiltonian identity of Gui \cite{Gui} in these cones separately.

\medskip

Once we have this linear energy growth bound, there are many ways to show that the solution has finitely many ends in the sense of \cite{DKP} and the refined asymptotic behavior of $u$ at infinity, see for example \cite{DKP, F0, Gui, Wang 3}.

\medskip

If the claim in Remark \ref{rmk 1.3.3} holds, the finiteness of ends can be obtained directly and then the energy bound follows as above.

\section{Morse index $1$ solutions: Proof of Theorem \ref{main result 2}}
\setcounter{equation}{0}

In this section we study solutions with Morse index $1$ in detail.
We use nodal set information to show that these solutions have only one critical point of saddle type.

\medskip

First we establish a general estimate on the number of nodal domains for direction derivatives of $u$, in terms of the Morse index bound.
\begin{prop}\label{number of nodal domains}
Suppose the Morse index of $u$ equals $N$. For any unit vector $e$, the number of connected components of
$\{u_e\neq 0\}$ is not larger than $2N$.
\end{prop}
\begin{proof}
First recall some basic facts about the nodal set $\{u_e=0\}$ (see \cite{Bers}). Because $ u_e$ satisfies the linearized equation \eqref{linearized equation}, it
can be decomposed into $\mbox{sing}(u_e)\cup \mbox{reg}(u_e)$, where
$\mbox{sing}(u_e)$ consists of isolated points and $\mbox{reg}(u_e)$
is a family of embedded smooth curves with their endpoints in
$\mbox{sing}(u_e)$ or at infinity.

Assume by the contrary, the number of connected components of $\{u_e\neq 0\}$ is
larger than $2N$. Without loss of generality, assume $\{u_e>0\}$ has at least $N+1$ connected components, $\Omega_i,
i=1,\cdots, N+1$. By the strong maximum principle, $u_e>0$ on the other side of regular parts of $\partial\Omega_i$.

Let $\psi_i$ be the restriction of $|u_e|$ to $\Omega_i$, with zero
extension outside $\Omega_i$. Hence $\psi_i$ is continuous, and in
$\{\psi_i>0\}$, it satisfies the linearized equation \eqref{linearized equation}.

 For any $R>R_0$, choose the cut-off function $\eta_R$ as in the previous section.
Multiplying \eqref{linearized equation} by $\psi_i\eta_R^2$ and integrating by parts
on $\Omega_i$ leads to
\begin{equation}\label{2.02}
\int_{\R^2}|\nabla\left(\psi_i\eta_R\right)|^2+W^{\prime\prime}(u)
\left(\psi_i\eta_R\right)^2=\int_{\R^2}\psi_i^2|\nabla\eta_R|^2\leq\frac{C}{\log
R}.
\end{equation}

Take an $x_i$ belonging to the regular part of $\partial\Omega_i$. There exists $h_i>0$ so
that $B_{h_i}(x_i)$ is disjoint from $\Omega_j$, for any $j\neq i$. (For example, $u_e<0$ in $B_{h_i}(x_i)\setminus\Omega_i$.)
Let $\tilde{\psi}_i$ equal $\psi_i$ outside $B_{h_i}(x_i)$, while in
$B_{h_i}(x_i)$ it solves \eqref{linearized equation}.
By this choice, we get a constant $\delta_i>0$ such that
\begin{equation}\label{2.03}
\int_{B_{h_i}(x_i)}\frac{1}{2}|\nabla\tilde{\psi}_i|^2+W^{\prime\prime}(u)\tilde{\psi}_i^2\leq
\left[\int_{B_{h_i}(x_i)}\frac{1}{2}|\nabla\psi_i|^2+W^{\prime\prime}(u)\psi_i^2\right]-\delta_i.
\end{equation}

Combining \eqref{2.02} and \eqref{2.03} we get an $R$ such that
\begin{equation}\label{2.04}
\int_{\R^2}|\nabla\left(\tilde{\psi}_i\eta_R\right)|^2+W^{\prime\prime}(u)
\left(\tilde{\psi}_i\eta_R\right)^2<0, \quad \forall i=1,\cdots, N+1.
\end{equation}

Note that
$\tilde{\psi}_i\eta_R\in H_0^1(B_R)$ are continuous functions
satisfying
\[\tilde{\psi}_i\eta_R \tilde{\psi}_j\eta_R\equiv 0, \quad \forall 1\leq i\neq j\leq N+1.\]
Hence they form an orthogonal basis of an $(N+1)$-dimensional
subspaces of $H_0^1(B_R)$. By \eqref{2.04}, $\mathcal{Q}$ is negative definite
on this subspace. This is a contradiction with the Morse index bound on $u$.
\end{proof}

\begin{rmk}
It seems  more interesting to establish a relation between the number of ends and the Morse index, as in minimal surfaces \cite{Choe, GNY, Ros}. By the method in \cite{Du-Gui-Wang}, we can show that the number of ends is at most $4N+4$.
\end{rmk}

As a corollary we get
\begin{coro}\label{coro 1}
Given a solution $u$ with Morse index $1$, for any direction $e$, the nodal set $\{u_e=0\}$ is a single smooth curve. In particular, $\nabla u_e\neq 0$ on $\{u_e=0\}$.
\end{coro}
\begin{proof}
First recall that $\mbox{reg}(u_e)$ are smooth embedded curves where $\nabla u_e\neq0$, and  $\mbox{sing}(u_e)=\{u_e=0,\nabla u_e=0\}$. Moreover, for any $X\in\{u_e=0,\nabla u_e=0\}$, in a neighborhood of $X$, $\{u_e=0\}$ consists of at least $4$ smooth curves emanating from $X$. See \cite{Bers}. Hence if there is a singular point on $\{u_e=0\}$, by Jordan curve theorem there exist at least three connected components of $\{u_e\neq 0\}$, a contradiction with Proposition \ref{number of nodal domains}.
Therefore there is no singular point on $\{u_e=0\}$ and they are smooth curves.

If there are two connected components of $\{u_e=0\}$, they are smooth, properly embedded curves. Hence they are either closed or unbounded. By Jordan curve theorem, there are at least three components of $\{u_e\neq0\}$, still a contradiction with Proposition \ref{number of nodal domains}.
\end{proof}

\begin{coro}
Given a solution $u$ with Morse index $1$, any critical point of $u$ is nondegenerate.
\end{coro}
\begin{proof}
Suppose $X$ is a critical point of $u$. For any direction $e$, we have $u_e(X)=\nabla u(X)\cdot e=0$, that is, $X\in\{u_e=0\}$. By the previous corollary, $\nabla^2u(X)\cdot e=\nabla u_e(X)\neq 0$. Since $e$ is arbitrary, this means $\nabla^2u(X)$ is invertible.
\end{proof}

Denote
\begin{equation}
P:=W(u)-\frac{1}{2}|\nabla u|^2.
\end{equation}
By the Modica's inequality \cite{Modica}, $P>0$ in $\R^2$. By the proof of Lemma \ref{lem 1.2.1}, we also have
\[\lim_{|X|\to+\infty}P(X)=0.\]

\begin{lem}\label{lem critical point}
$\nabla P=0$ if and only if $\nabla u=0$.
\end{lem}
\begin{proof}
Since
\[\nabla P=W^\prime(u)\nabla u-\nabla^2u\cdot\nabla u,\]
we see that $\nabla P=0$ if $\nabla u=0$.

On the other hand assume that $\nabla u(X)\neq 0$. Without loss of generality, take two orthonormal basis $\{e_1,e_2\}$ and assume $u_{e_2}(X)=|\nabla u(X)|$, $u_{e_1}(X)=0$.
Note that locally $\{u_{e_1}/u_{e_2}=0\}$ coincides with $\{u_{e_1}=0\}$, which is a smooth curve by Corollary \ref{coro 1}. Since both $u_{e_1}$ and $u_{e_2}$ satisfy the linearized equation
\eqref{linearized equation},
 we infer that
\[\mbox{div}\left(u_{e_2}^2\nabla\frac{u_{e_1}}{u_{e_2}}\right)=0,\]
which implies that $\nabla \frac{u_{e_1}}{u_{e_2}}(X)\neq 0$.

By a direct calculation we get
\[\nabla P=u_{e_2}^2J\nabla\frac{u_{e_1}}{u_{e_2}},\]
where $J$ is the $\pi/2$-rotation in the anti-clockwise direction. Therefore
$\nabla P(X)\neq 0$.
\end{proof}

At a critical point of $P$, since $\nabla u=0$, we have
\[\nabla^2P=W^\prime(u)\nabla^2u-\nabla^2u\cdot\nabla^2u=\Delta u\nabla^2u-\nabla^2u\cdot\nabla^2u,\]
where $\cdot$ denotes matrix multiplication. Since $\nabla^2u$ is invertible at this point, by a direct calculation we see both of the eigenvalues of $\nabla^2P$ equal $\mbox{det}\nabla^2u$. Thus every critical point of $P$ is either a strict  maximal or a strict minimal point.

\begin{prop}\label{prop one critical point}
There is only one critical point of $P$.
\end{prop}
\begin{proof}
Since $P>0$ and $P\to0$ at infinity, the maxima of $P$ is attained, which is a critical point of $P$. Denote this point by $X_1$.

Assume there exists a second critical point of $P$, $X_2$. By the previous analysis, $X_2$ is either a strict maximal or minimal point.
\begin{itemize}
\item If $X_2$ is a strict maximal point, take
\[\Upsilon:=\{\gamma\in H^1([0,1],\R^2): \gamma(0)=X_1, \gamma(1)=X_2\}.\]
Define
\[c_\ast:=\max_{\gamma\in\Upsilon}\min_{t\in[0,1]}P(\gamma(t)).\]
Clearly $c_\ast<\min\{u(X_1),u(X_2)\}$.
Since $P\to0$ at infinity, by constructing a competitor curve, we can show that $c_\ast>0$.
By the Mountain Pass Theorem, $c_\ast$ is a critical value of $P$. Moreover, there exists a curve $\gamma_\ast\in\Upsilon$ and $t_\ast\in(0,1)$ such that
\[P(\gamma_\ast(t_\ast))=\min_{t\in[0,1]}P(\gamma_\ast(t))=c_\ast\]
and $\nabla P(\gamma_\ast(t_\ast))=0$.
Therefore $\gamma_\ast(t_\ast)$ cannot be a strict local maxima. If it is a strict local minima, by deforming $\gamma_\ast$ in a small neighborhood of $\gamma_\ast(t_\ast)$, we get a contradiction with the definition of $c_\ast$. This contradiction implies that $X_2$ cannot be a strict maximal point of $P$.

\item If $X_2$ is a strict local minimal point, take
\[\Upsilon:=\{\gamma\in H^1([0,+\infty),\R^2): \gamma(0)=X_2, \lim_{t\to+\infty}\gamma(t)=+\infty\}.\]
Define
\[c_\ast:=\min_{\gamma\in\Upsilon}\max_{t\in[0,+\infty)}P(\gamma(t)).\]
As in the first case we get a critical point of $P$, which is of mountain pass type. This leads to the same contradiction as before.
\end{itemize}
These contradictions show that $X_1$ is the only critical point of $P$.
\end{proof}

By Lemma \ref{lem critical point}, $u$ has only one critical point, too. Denote this point by $X$. Since this point is the maximal point of $P$, $\mbox{det}\nabla^2u(X)<0$. Thus it is a nondegenerate saddle point of $u$.

\begin{rmk}
Let  $\Psi:=g^{-1}\circ u$ be the distance type function. The Modica inequality \cite{Modica} is equivalent to the condition that $|\nabla\Psi|<1$.
The above method can be further developed to show that $\nabla\Psi$ is a diffeomorphism from $\R^2$ to $B_1(0)$. In particular, for any $r>0$,
\[\mbox{deg}\left(\frac{\nabla u}{|\nabla u|},\partial B_r(X)\right)= 1 \quad \mbox{or } -1,\]
Compare this with \cite{CC1}.
\end{rmk}

\begin{lem}
$\{u=u(X)\}$ is composed by two smooth curves diffeomorphic to $\R$, intersecting exactly at $X$.
\end{lem}
\begin{proof}
Since $X$ is the only critical point of $u$, $\{u=u(X)\}$ is a smooth embedded curve outside $X$. Because $X$ is nondegenerate and of saddle type, in a small neighborhood of $X$ this level set consists of two smooth curves intersecting transversally at $X$.

Denote this connected component of $\{u=u(X)\}$ by $\Sigma$. $\Sigma$ does not enclose any bounded domain, because otherwise $u$ would have a local maximal or minimal point in this bounded domain, which is a contradiction with Proposition \ref{prop one critical point}. Hence we can write $\Sigma=\Sigma_1\cup\Sigma_2$, where $\Sigma_1$ and $\Sigma_2$ are smooth properly embedded curves diffeomorphic to $\R$. Moreover, $\Sigma_1$ and $\Sigma_2$ intersect at and only at $X$.

If there exists a second connected component of $\{u=u(X)\}$. Denote it by $\widetilde{\Sigma}$. Similar to the above discussion, $\widetilde{\Sigma}$ is a smooth embedded curve diffeomorphic to $\R$.
$\widetilde{\Sigma}$ and $\Sigma$ bound a domain $\Omega$. Without loss of generality, assume $u>u(X)$ in $\Omega$.

Let
\[\Upsilon:=\{\gamma\in H^1([0,1],\R^2): \gamma(0)\in\Sigma, \gamma(1)\in\widetilde{\Sigma}\},\]
and
\[c_\ast:=\min_{\gamma\in\Upsilon}\max_{t\in[0,1]}u(\gamma(t)).\]
By choosing a competitor curve, we see $c_\ast<1$.   Hence by Lemma \ref{lem 1.2.1}, $c_\ast$ is attained by a curve $\gamma_\ast\in\Upsilon$. Because $\Sigma$ and $\widetilde{\Sigma}$ are separated,
\[\max_{t\in[0,1]}u(\gamma_\ast(t))>u(X).\]

By the Mountain Pass Theorem, there exists $t_\ast\in(0,1)$ such that $u(\gamma_\ast(t_\ast))=c_\ast$ and $\gamma_\ast(t_\ast)$ is a critical point of $u$. This is a contradiction with Proposition \ref{prop one critical point}. Therefore $\{u=u(X)\}=\Sigma$.
\end{proof}

Combining this lemma with Theorem \ref{main result 1}, we see there are exactly four ends of $u$. This completes the proof of Theorem \ref{main result 2}.

\part{Second order estimate on interfaces}

In this part we study second order regularity of clustering interfaces and prove Theorem \ref{main result 3} and Theorem \ref{main result 4}. Recall that $u_\varepsilon$ is a sequence of solutions to \eqref{equation scaled}  satisfying  $({\bf H1})-({\bf H3})$ in Section 2.2.

\section{The case of unbounded curvatures}\label{sec 2}
\setcounter{equation}{0}

By standard elliptic regularity theory, $u_\varepsilon\in C^{2,\theta}_{loc}(B_2)$. Concerning the regularity of $f_{\alpha,\varepsilon}$, we first prove that different components are at least $O(\varepsilon)$ apart.

\begin{lem}\label{O(1) scale}
For any $\alpha\in\{1,\cdots,Q\}$ and $x_{\varepsilon}\in\Gamma_{\alpha,\varepsilon}\cap \mathcal{C}_{3/2}$, as $\varepsilon\to0$, $\tilde{u}_\varepsilon(x):=u_\varepsilon(x_{\varepsilon}+\varepsilon x)$ converges to a one dimensional solution in $C^2_{loc}(\R^n)$.
In particular, for any $\alpha\in\{1,\cdots,Q\}$,
\begin{equation}
\frac{f_{\alpha+1,\varepsilon}-f_{\alpha,\varepsilon}}{\varepsilon}\to+\infty \quad\mbox{ uniformly in } B_{3/2}^{n-1}.
\end{equation}
\end{lem}
\begin{proof}
In $B_{\varepsilon^{-1}/2}$, $\tilde{u}_\varepsilon(x)$ satisfies the Allen-Cahn equation \eqref{equation}.
By standard elliptic regularity theory, $\tilde{u}_\varepsilon(x)$ is uniformly bounded in $C^{2,\theta}_{loc}(\R^n)$. Using Arzela-Ascoli theorem, as $\varepsilon\to0$, it converges to a limit function $u_\infty$ in $C^2_{loc}(\R^n)$. For each $\beta\in\{1,\cdots,Q\}$, either $\left(f_{\beta,\varepsilon}(x_\ast^\prime+\varepsilon x^\prime)-f_{\alpha,\varepsilon}(x_\ast^\prime)\right)/\varepsilon$ converges to a limit function $f_{\beta,\infty}$ in $C_{loc}(\R^{n-1})$ or it converges to $\pm\infty$ uniformly on any compact set of $\R^{n-1}$.

Assume $t_\varepsilon\to t_\infty$. Then $\{u_\infty=t_\infty\}$ consists of $Q^\prime\leq Q$  connected components, $\Gamma_{\alpha,\infty}$, $1\leq \alpha\leq Q^\prime$. Each $\Gamma_{\alpha,\infty}$ is
represented by the graph $\{x_n:=f_{\alpha,\infty}(x^\prime)\}$.
In $\R^{n-1}$, $|\nabla f_{\alpha,\infty}|\leq C$ for a universal constant $C$ and
\[f_{1,\infty}\leq\cdots\leq f_{Q^\prime,\infty}.\]
By applying the sliding method in \cite{BCN}, $u_\infty(x)=g(x\cdot e)$ for some unit vector $e$. In particular, $Q^\prime=1$ and for any $\beta\neq\alpha$, $\left(f_{\beta,\varepsilon}(x_\ast^\prime+\varepsilon x^\prime)-f_{\alpha,\varepsilon}(x_\ast^\prime)\right)/\varepsilon$  goes to $\pm\infty$ uniformly on any compact set of $\R^{n-1}$.
\end{proof}

A consequence of this lemma is
\begin{coro}\label{coro 2.2}
Given a constant $b\in(0,1)$,
\begin{itemize}
\item[(i)] there exists a constant $c(b)>0$ depending only on $b$ such that
\[\frac{\partial u_\varepsilon}{\partial x_n}>\frac{c(b)}{\varepsilon}, \quad\mbox{in } \{|u_\varepsilon|<1-b\}\cap \mathcal{C}_{3/2};\]
\item[(ii)] for any $t\in[-1+b,1-b]$ and all $\varepsilon$ small, $\{u_\varepsilon=t\}$ is composed by $Q$ Lipschitz graphs
\[\{x_n=f_{\alpha,\varepsilon}^t(x^\prime)\}, \quad \alpha=1,\cdots, Q.\]
\end{itemize}
\end{coro}
By the implicit function theorem, $f_{\alpha,\varepsilon}$
belongs to $C^{2,\theta}_{loc}(B_2^{n-1})$, although we do not have any uniform bound on their $C^{2,\theta}$ norm but only a uniform Lipschitz bound.

\medskip

Now
\[\nu_\varepsilon(x):=\frac{\nabla u_\varepsilon(x)}{|\nabla u_\varepsilon(x)|}\]
is well defined and smooth in $\{|u_\varepsilon|\leq 1-b\}$.
Recall that $B(u_\varepsilon)(x)=\nabla \nu_\varepsilon(x)$. We have
\[|B(u_\varepsilon)(x)|^2=|A_\varepsilon(x)|^2+|\nabla_T\log|\nabla u_\varepsilon(x)||^2,\]
where $A_\varepsilon(x)$ is the second fundamental form of the level set $\{u_\varepsilon=u_\varepsilon(x)\}$ and $\nabla_T$ denotes the tangential derivative along the level set $\{u_\varepsilon=u_\varepsilon(x)\}$.

Assume as $\varepsilon\to0$,
\[\sup_{\mathcal{C}_1\cap\{|u_\varepsilon|\leq 1-b\}}|B(u_\varepsilon)(x)|\to+\infty.\]
Let $x_\varepsilon\in\mathcal{C}_1\cap\{|u_\varepsilon|\leq 1-b\}$ attain the following maxima (we denote $x=(x^\prime,x_n)$)
\begin{equation}
\max_{\mathcal{C}_{3/2}\cap\{|u_\varepsilon|\leq 1-b\}}\left(\frac{3}{2}-|x^\prime|\right)|B(u_\varepsilon)(x)|.
\end{equation}
Denote
\begin{equation}
\label{7.1n}
L_\varepsilon:=|B(u_\varepsilon)(x_\varepsilon)|,\ \ \ \ \ \ \ r_\varepsilon:=\left(\frac{3}{2}-|x_\varepsilon^\prime|\right)/2.
\end{equation}
 Then by definition
\begin{equation}\label{2.1}
L_\varepsilon r_\varepsilon\geq\frac{1}{2}\sup_{\mathcal{C}_1\cap\{|u_\varepsilon|\leq 1-b\}}|B(u_\varepsilon)(x)|\to+\infty.
\end{equation}
In particular, $L_\varepsilon\to+\infty$.

By the choice of $r_\varepsilon$ at (\ref{7.1n}), we have (here $\mathcal{C}_{r_\varepsilon}(x_\varepsilon^\prime):=B_{r_\varepsilon}^{n-1}(x_\varepsilon^\prime)\times(-1,1)$)
\begin{equation}\label{2.2}
\max_{x\in\mathcal{C}_{r_\varepsilon}(x_\varepsilon^\prime)\cap\{|u_\varepsilon|\leq 1-b\}}|B(u_\varepsilon)(x)|\leq 2L_\varepsilon.
\end{equation}

Let $\epsilon:=L_\varepsilon\varepsilon$ and define $u_\epsilon(x):=u_\varepsilon(x_\varepsilon+L_\varepsilon^{-1}x)$.
Then $u_\epsilon$ satisfies \eqref{equation scaled} with parameter $\epsilon$ in $B_{L_\varepsilon r_\varepsilon}(0)$. For any $t\in[-1+b,1-b]$, the level set $\{u_\epsilon=t\}$ consists of $Q$ Lipschitz graphs
\[\left\{x_n=f_{\beta,\epsilon}^t(x^\prime):=L_\varepsilon\left[f_{\beta,\varepsilon}^t(x_\varepsilon^\prime+L_\varepsilon^{-1}x^\prime)-f_{\alpha,\varepsilon}^t(x_\varepsilon^\prime)\right]\right\},\]
where $\alpha$ is chosen so that $x_\varepsilon$ lies in the connected component of $\{|u_\varepsilon|\leq 1-b\}$ containing $\Gamma_{\alpha,\varepsilon}$.

By \eqref{2.2}, we also have
\[
|B(u_\epsilon)|\leq 2, \quad\mbox{for } x\in \mathcal{C}_1\cap\{|u_\epsilon|\leq 1-b\}.
\]

Without loss of generality, by abusing notations, we will assume in the following
\begin{itemize}
\item [{\bf (H4)}] There exist two constants $b\in(0,1)$ and $C>0$ independent of $\varepsilon$ such that $|B(u_\varepsilon)|\leq C$ for any $x\in \mathcal{C}_2\cap\{|u_\varepsilon|\leq 1-b\}$.
\end{itemize}

\section{Fermi coordinates}\label{sec Fermi coordinates}
\setcounter{equation}{0}

\subsection{Definition}
For simplicity of presentation, we now work in the stretched version and do not write the dependence on $\varepsilon$ explicitly.

By denoting $R=\varepsilon^{-1}$,
$u(x)=u_\varepsilon(\varepsilon x)$
satisfies the Allen-Cahn equation \eqref{equation} in $\mathcal{C}_{2R}:=B_{2R}^{n-1}\times(-R,R)$.

Its nodal set $\{u=0\}$ consists of $Q$ connected components, $\Gamma_\alpha$, $1\leq \alpha\leq Q$, which is
represented by the graph $\{x_n:=f_\alpha(x^\prime)\}$.
In $B_{2R}^{n-1}$, there is a constant $C$ independent of $\varepsilon$ such that
\begin{equation}\label{3.1}
|\nabla f_\alpha|\leq C, \quad |\nabla^2f_\alpha|\leq C\varepsilon.
\end{equation}
By {\bf (H2)},
\[-\frac{R}{2}<f_1<\cdots<f_Q<\frac{R}{2}.\]

 The second fundamental form of $\Gamma_\alpha$ with respect to the parametrization $y\mapsto (y,f_\alpha(y))$ is
\[A_{ij}(y,0)=\frac{1}{\sqrt{1+|\nabla f_\alpha(y)|^2}}\frac{\partial}{\partial y_i}\left[\frac{1}{\sqrt{1+|\nabla f_\alpha(y)|^2}}\frac{\partial f_\alpha}{\partial y_j}(y)\right].\]

\medskip

The Fermi coordinate is defined by $(y,z)\mapsto x$ as
\[x=(y,f_\alpha(y))+zN_\alpha(y),\]
where
\[N_\alpha(y)=\frac{(-\nabla^\prime f_\alpha(x^\prime),1)}{\sqrt{1+|\nabla^\prime f_\alpha(x^\prime)|^2}}.\]
Note that here $z$ is nothing else but the signed distance to $\Gamma_\alpha$ which is positive above $\Gamma_\alpha$.
By \eqref{3.1}, there exists a constant $\delta\in(0,1/2)$ independent of $\varepsilon$ such that, the Fermi coordinate is well defined and $C^2$ in the open set $\{|y|<3R/2, |z|<\delta R\}$.

Define the vector field
\[X_i:=\frac{\partial}{\partial y^i}+z\frac{\partial N_\alpha}{\partial y^i}=\sum_{j=1}^{n-1}\left(\delta_{ij}-zA_{ij}\right)\frac{\partial}{\partial y^j}.\]
For any $z\in(-\delta R,\delta R)$, let
$\Gamma_{\alpha,z}:=\{dist(x,\Gamma_\alpha)=z\}$. The Euclidean metric restricted
to $\Gamma_{\alpha,z}$ is denoted by $g_{ij}(y,z)dy^i\otimes dy^j$, where
\begin{eqnarray}\label{metirc tensor}
g_{ij}(y,z)&=&<X_i(y,z),X_j(y,z)> \\
&=&g_{ij}(y,0)-2z\sum_{k=1}^{n-1}A_{ik}(y,0)g_{jk}(y,0)+z^2\sum_{k,l=1}^n g_{kl}(y,0)A_{ik}(y,0)A_{jl}(y,0). \nonumber
\end{eqnarray}
Here
\begin{equation}\label{metirc tensor 0}
g_{ij}(y,0)=\frac{1}{1+|\nabla f_\alpha(y)|^2}\left[\delta_{ij}+\frac{\partial f_\alpha}{\partial y_i}(y)\frac{\partial f_\alpha}{\partial y_j}(y)\right].
\end{equation}

The second fundamental form of $\Gamma_z$ has the form
\begin{equation}\label{A(z)}
A(y,z)=\left(I-zA(y,0)\right)^{-1}A(y,0).
\end{equation}

\subsection{Error in $z$} In this subsection we collect several estimates on the error of various terms in $z$. Recall that $\varepsilon$ is the upper bound on curvatures of level sets of $u$, see \eqref{3.1}.

By \eqref{3.1}, $|A(y,0)|\lesssim \varepsilon$. Thus for $|z|<\delta R$, $|A(y,z)|\lesssim \varepsilon$. We also have
\begin{lem} In $B_{3R/2}^{n-1}$,
\begin{equation}\label{bound on 3rd derivatives}
|\nabla A(y,0)|+|\nabla^2 A(y,0)|\lesssim\varepsilon.
\end{equation}
\end{lem}
\begin{proof}
By Corollary \ref{coro 2.2}, $|\nabla u|\geq c(b)>0$ in $\{|u|<1-b\}$, where $c(b)$ is a constant depending only on $b$.
Hence $\nu=\nabla u/|\nabla u|$ is well defined and smooth in $\{|u|<1-b\}$.

By direct calculation, we have
\begin{equation}\label{Gauss equation}
-\mbox{div}\left(|\nabla u|^2\nabla\nu\right)=|\nabla u|^2|\nabla\nu|^2\nu.
\end{equation}
Recall that $B=\nabla \nu$. Differentiating \eqref{Gauss equation} gives the following Simons type equation
\begin{equation}\label{Simons}
-\mbox{div}\left(|\nabla u|^2\nabla B\right)=|\nabla u|^2|B|^2B+|\nabla u|^2\nabla|B|^2\otimes\nu+|B|^2\nabla|\nabla u|^2\otimes\nu+|\nabla u|^2\nabla^2\log|\nabla u|^2\cdot B.
\end{equation}
For any $x\in\{|u|<1-2b\}$, there exists a constant $h(b)$ such that $B_{2h(b)}(x)\subset\{|u|<1-b\}$. Because $|\nabla u|^2$ has a positive lower and upper bound and it is uniformly continuous in $B_{2h(b)}(x)$, by standard interior gradient estimate,
\[ \sup_{B_{h(b)}}|\nabla B| \lesssim \sup_{B_{2h(b)}}|B|+\sup_{B_{2h(b)}}|\mbox{div}\left(|\nabla u|^2\nabla B\right)|\lesssim\varepsilon.
\]
The bound on $|\nabla^2B|$ is obtained by bootstrapping elliptic estimates.
\end{proof}

By \eqref{A(z)},
\begin{equation}\label{error in z 1}
|A(y,z)-A(y,0)|\lesssim |z||A(y,0)|^2\lesssim\varepsilon^2|z|.
\end{equation}
Similarly, by \eqref{metirc tensor}, the error of metric tensors is
\begin{equation}\label{error in z 2}
|g_{ij}(y,z)-g_{ij}(y,0)|\lesssim \varepsilon|z|,
\end{equation}
\begin{equation}\label{error in z 3}
|g^{ij}(y,z)-g^{ij}(y,0)|\lesssim \varepsilon|z|.
\end{equation}
As a consequence, the error of mean curvature is
\begin{equation}\label{error in z 4}
|H(y,z)-H(y,0)|\lesssim  \varepsilon^2|z|.
\end{equation}

By \eqref{3.1} and \eqref{bound on 3rd derivatives}, for any $|z|<\delta R$,
\begin{equation}\label{derivatives of metric tensor}
|\nabla_yg_{ij}(y,z)|+|\nabla_yg^{ij}(y,z)|\lesssim\varepsilon.
\end{equation}

\medskip

The Laplacian operator in Fermi coordinates has the form
\[\Delta_{\R^N}=\Delta_z-H(y,z)\partial_z+\partial_{zz},\]
where
\begin{eqnarray*}
\Delta_z&=&\sum_{i,j=1}^{n-1}\frac{1}{\sqrt{\mbox{det}(g_{ij}(y,z))}}\frac{\partial}{\partial y_j}\left(\sqrt{\mbox{det}(g_{ij}(y,z))}g^{ij}(y,z)\frac{\partial}{\partial y_i}\right)\\
&=&\sum_{i,j=1}^{n-1}g^{ij}(y,z)\frac{\partial^2}{\partial y_i\partial y_j}+\sum_{i=1}^{n-1}b^i(y,z)\frac{\partial}{\partial y_i}
\end{eqnarray*}
with
\[b^i(y,z)=\frac{1}{2}\sum_{j=1}^{n-1}g^{ij}(y,z)\frac{\partial}{\partial y_j}\log\mbox{det}(g_{ij}(y,z)).\]

By \eqref{error in z 3} and \eqref{derivatives of metric tensor}, we get
\begin{lem}
For any function $\varphi\in C^2(B_{3R/2}^{n-1})$,
\begin{equation}\label{error in z 5}
|\Delta_z\varphi(y)-\Delta_0\varphi(y)|\lesssim\varepsilon|z|\left(|\nabla^2\varphi(y)|+|\nabla\varphi(y)|\right).
\end{equation}
\end{lem}

\subsection{Comparison of distance functions}

For each $\alpha$, the local coordinates on $\Gamma_\alpha$ is fixed to be the same one, $y\in B_{3R/2}^{n-1}$, which represents the point $(y,f_\alpha(y))$. The singed distance to $\Gamma_\alpha$, which is positive in the above, is denoted by $d_\alpha$.
Given a point $X$, if $(y,f_\beta(y))$ is the nearest point on $\Gamma_\beta$ to $X$, we then define $\Pi_\beta(X)=y$.

If $\alpha\neq\beta$, we cannot expect $\Pi_\alpha(X)=\Pi_\beta(X)$. However, the following estimates on their distance hold, when $\Gamma_\alpha$ and $\Gamma_\beta$ are close in some sense.
\begin{lem}\label{comparison of distances}
For any $X\in B_{3R/2}^{n-1}\times(-\delta R,\delta R)$ and $\alpha\neq\beta$, if
$|d_\alpha(X)|\leq K|\log\varepsilon|$ and $|d_\beta(X)|\leq K|\log\varepsilon|$, then we have
\begin{equation}\label{3.2}
\mbox{dist}_{\Gamma_\beta}\left(\Pi_\beta\circ\Pi_\alpha(X),\Pi_\beta(X)\right)\leq C(K)\varepsilon^{1/2}|\log\varepsilon|^{3/2},
\end{equation}
\begin{equation}\label{3.3}
|d_\beta\left(\Pi_\alpha(X)\right)+d_\alpha\left(\Pi_\beta(X)\right)|\leq C(K)\varepsilon^{1/2}|\log\varepsilon|^{3/2},
\end{equation}
\begin{equation}\label{3.4}
|d_\alpha(X)-d_\beta(X)+d_\beta\left(\Pi_\alpha(X)\right)|\leq C(K)\varepsilon^{1/2}|\log\varepsilon|^{3/2},
\end{equation}
\begin{equation}\label{3.5}
|d_\alpha(X)-d_\beta(X)-d_\alpha\left(\Pi_\beta(X)\right)|\leq C(K)\varepsilon^{1/2}|\log\varepsilon|^{3/2},
\end{equation}
\begin{equation}\label{3.6}
1-\nabla d_\alpha(X)\cdot \nabla d_\beta(X)\leq C(K)\varepsilon^{1/2}|\log\varepsilon|^{3/2},
\end{equation}
\end{lem}
\begin{proof}
We divide the proof into three steps.

{\bf Step 1.} After a rotation and a translation, assume $\Pi_\alpha(X)=0$, the tangent plane of $\Gamma_\alpha$ at $(0,0)$ is the horizontal hyperplane and $X=(0,T)$.
Since the curvature of $\Gamma_\alpha$ is of the order $O(\varepsilon)$, $\Gamma_\alpha\cap\mathcal{C}_{\delta R}$ is a Lipschitz graph $\{x_n=f_\alpha(x^\prime)\}$.
By the above choice, $f_\alpha(0)=\nabla f_\alpha(0)=0$.

Because $|d_\beta(0)|\leq K|\log\varepsilon|$, $\Gamma_\alpha$ and $\Gamma_\beta$ are disjoint and their curvature is of the order $O(\varepsilon)$, we can show that $\Gamma_\beta\cap\mathcal{C}_{\delta R}$ is also a Lipschitz graph $\{x_n=f_\beta(x^\prime)\}$, see Lemma \ref{lem graph construction}.

By this Lipschitz property of $f_\alpha$ and $f_\beta$,
\[|f_\beta(0)-f_\alpha(0)|\leq C|d_\beta(0)|\leq  C\left(|d_\alpha(X)|+|d_\beta(X)|\right)\leq 2CK|\log\varepsilon|.\]
Since $f_\beta-f_\alpha\neq 0$ and $|\nabla^2(f_\beta-f_\alpha)|\lesssim\varepsilon$ in $B_{\delta R}^{n-1}(0)$, by an interpolation inequality we get
\[|\nabla f_\beta(0)|=|\nabla(f_\beta-f_\alpha)(0)|\lesssim\sqrt{\varepsilon|\log\varepsilon|}.\]

{\bf Step 2.} Because $\Gamma_\beta\cap\mathcal{C}_{2K|\log\varepsilon|}$ belongs to an $O(\varepsilon|\log\varepsilon|^2)$ neighborhood of the hyperplane $\mathcal{P}_\beta:=\{x_n=f_\beta(0)+\nabla f_\beta(0)\cdot x^\prime\}$,
\begin{eqnarray}\label{3.7}
d_\beta(X)&=&T-f_\beta(0)+O(\sqrt{\varepsilon|\log\varepsilon|}|T|)+O(\varepsilon|\log\varepsilon|^2)\\
&=&T-f_\beta(0)+O(\varepsilon^{1/2}|\log\varepsilon|^{3/2}). \nonumber
\end{eqnarray}
Similarly,
\begin{equation}\label{3.8}
d_\beta(\Pi_\alpha(X))=f_\alpha(0)-f_\beta(0)+O(\varepsilon^{1/2}|\log\varepsilon|^{3/2}).
\end{equation}
Interchanging the position of $\alpha,\beta$ gives
\begin{equation}\label{3.9}
d_\alpha(\Pi_\beta(X))=f_\beta(0)-f_\alpha(0)+O(\varepsilon^{1/2}|\log\varepsilon|^{3/2}).
\end{equation}
Combining \eqref{3.7}, \eqref{3.8} and \eqref{3.9}, we obtain \eqref{3.3}-\eqref{3.5}.

{\bf Step 3.} In our setting, we have
\[1-\nabla d_\alpha(X)\cdot \nabla d_\beta(X)=1-\frac{\partial}{\partial x_n}d_\beta(0,T).\]
For any $t\in(0,f_\beta(0))$, let $(x^\prime(t),f_\beta(x^\prime(t))$ be the unique nearest point on $\Gamma_\beta$ to $(0,t)$. By definition, we have
\begin{equation}\label{3.10}
x^\prime(t)+\left[f_\beta(x^\prime(t))-t\right]\nabla f_{\beta}(x^\prime(t))=0.
\end{equation}
Differentiating this identity in $t$ leads to
\[
\left[1+|\nabla f_\beta(x^\prime(t))|^2\right]\frac{d}{dt}x^\prime(t)+\left[f_\beta(x^\prime(t))-t\right]\nabla^2 f_{\beta}(x^\prime(t))\cdot\frac{d}{dt}x^\prime(t)=\nabla f_\beta(x^\prime(t)).
\]
As in Step 1, we still have $|\nabla f_\beta(x^\prime(t))|\leq C(K)\varepsilon^{1/2}|\log\varepsilon|^{1/2}$. Together with the fact that $|\nabla^2 f_\beta|\lesssim\varepsilon$, we get
\begin{equation}\label{3.11}
\Big|\frac{d}{dt}x^\prime(t)\Big|\leq C(K)\varepsilon^{1/2}|\log\varepsilon|^{1/2}.
\end{equation}
Integrating this in $t$ on $[0,T]$ gives \eqref{3.2}.

Note that \eqref{3.10} also implies that
\begin{equation}\label{3.12}
|x^\prime(t)|\leq C(K)\varepsilon^{1/2}|\log\varepsilon|^{3/2}.
\end{equation}

Because
\[d_\beta(0,t)=\sqrt{|x^\prime(t)|^2+\left(f_\beta(x^\prime(t))-t\right)^2},\]
we have
\begin{eqnarray*}
\frac{d}{dt}d_\beta(0,t)&=&-\frac{f_\beta(x^\prime(t))-t}{\sqrt{|x^\prime(t)|^2+\left(f_\beta(x^\prime(t))-t\right)^2}}
+\frac{x^\prime(t)+\left(f_\beta(x^\prime(t))-t\right)
\nabla f_\beta(x^\prime(t))}{\sqrt{|x^\prime(t)|^2+\left(f_\beta(x^\prime(t))-t\right)^2}}\cdot \frac{d}{dt}d_\beta(0,t)\\
&=&-1+O\left(|x^\prime(t)|\right)+O\left(\Big|\frac{d}{dt}x^\prime(t)\Big|\right)\\
&=&-1+O\left(\varepsilon^{1/2}|\log\varepsilon|^{3/2}\right),
\end{eqnarray*}
which gives \eqref{3.6}.
\end{proof}

\subsection{Some notations}
In the remaining part of this paper the following notations will be employed.
\begin{itemize}
\item Given a point on $\Gamma_\alpha$ with local coordinates $(y,0)$ in the Fermi coordinates, denote
\[D_\alpha(y):=\min\{|d_{\alpha-1}(y,0)|, |d_{\alpha+1}(y,0)|\}.\]

\item For $\lambda\geq0$, let
\[\mathcal{M}_\alpha^\lambda:=\left\{|d_\alpha|<|d_{\alpha-1}|+\lambda \quad \mbox{and } \quad  |d_\alpha|<|d_{\alpha+1}|+\lambda\right\}.\]
In this Part II we take the convention that $d_0=-\delta R$ and $d_{Q+1}=\delta R$.


\item
In the Fermi coordinates with respect to $\Gamma_\alpha$, there exist two smooth functions $\rho_\alpha^\pm(y)$ such that
\[\mathcal{M}_\alpha^0=\{(y,z): \rho_\alpha^-(y)<z<\rho_\alpha^+(y)\}.\]

\item For any $r>0$, let
\[\mathcal{M}_\alpha^0(r):=\{(y,z)\in\mathcal{M}_\alpha^0, \quad |y|<r\}.\]

\item In this Part II we denote
\[\mathcal{D}(r)=\cup_{\alpha=1}^Q\mathcal{M}_\alpha^0(r).\]

\item The covariant derivative on $\Gamma_z$ with respect to the induced metric is denoted by $\nabla_z$.
\end{itemize}

\section{The approximate solution}\label{sec approximate solution}
\setcounter{equation}{0}

\subsection{Optimal approximation}\label{sec optimal approximation}

Fix a function $\zeta\in C_0^\infty(-2,2)$ with $\zeta\equiv 1$ in $(-1,1)$, $|\zeta^\prime|+|\zeta^{\prime\prime}|\leq 16$. Let
\[\bar{g}(x)=\zeta(3|\log\varepsilon|x)g(x)+\left(1-\zeta(3|\log\varepsilon|x)\right)\mbox{sgn}(x), \quad x\in(-\infty,+\infty).\]
Then $\bar{g}$ is an approximate solution to the one dimensional Allen-Cahn equation, that is,
\begin{equation}\label{1d eqn}
\bar{g}^{\prime\prime}=W^\prime\left(\bar{g}\right)+\bar{\xi},
\end{equation}
where $\mbox{spt}(\bar{\xi})\in\{3|\log\varepsilon|<|x|<6|\log\varepsilon|\}$, and $|\bar{\xi}|+|\bar{\xi}^\prime|+|\bar{\xi}^{\prime\prime}|\lesssim \varepsilon^3$.

We also have (for the definition of $\sigma_0$ see Appendix \ref{sec 1d solution})
\begin{equation}\label{1d energy}
\int_{-\infty}^{+\infty}\bar{g}^\prime(t)^2dt=\sigma_0+O(\varepsilon^3).
\end{equation}

\medskip

Without loss of generality assume $u<0$ below $\Gamma_1$. Given a tuple of functions $\textbf{h}:=(h_1(y),\cdots, h_Q(y))$, for each $\alpha$, in the Fermi coordinates with respect to $\Gamma_\alpha$, let
\[g_\alpha(y,z):=\bar{g}\left((-1)^{\alpha-1}(z-h_\alpha(y))\right)=\bar{g}\left((-1)^{\alpha-1}\left(d_\alpha(y,z)-h_\alpha(y)\right)\right).\]
Then we define
\[g(y,z;\textbf{h}):=\sum_{\alpha} g_\alpha+\frac{(-1)^Q+1}{2}.\]

For simplicity of notation, denote
\[g_\alpha^\prime=\bar{g}^\prime\left((-1)^{\alpha-1}(z-h_\alpha(y))\right), \quad g_\alpha^{\prime\prime}=\bar{g}^{\prime\prime}\left((-1)^{\alpha-1}(z-h_\alpha(y))\right), \quad \cdots .\]

\begin{prop}\label{prop optimal approximation}
There exists $\textbf{h}(y)=(h_\alpha(y))$ with $|h_\alpha|\ll 1$ for each $\alpha$, such that for any $\alpha$ and $y\in B_R^{n-1}$,
\begin{equation}\label{2 orthogonal condition}
\int_{-\delta R}^{\delta R}\left[u(y,z)-g(y,z;\textbf{h})\right]\bar{g}^\prime\left((-1)^{\alpha-1}(z-h_\alpha(y))\right)dz=0,
\end{equation}
where $(y,z)$ denotes the Fermi coordinates with respect to $\Gamma_\alpha$.
\end{prop}
\begin{proof}
Denote
\[F(h_1,\cdots, h_Q):=\left(\int_{-\delta R}^{\delta R}\left[u(y,z)-g(y,z;\textbf{h})\right]\bar{g}^\prime\left((-1)^{\alpha-1}(z-h_\alpha(y))\right)dz\right),\]
which is viewed as a map from the Banach space $\mathcal{X}:=C^0(B_R^{n-1}(0))^Q$ to itself.

Clearly $F$ is a $C^1$ map. Furthermore,
\begin{eqnarray*}
\left(DF(\textbf{h})\xi\right)_\alpha&=&(-1)^\alpha\xi_\alpha(y)\int_{-\delta R}^{\delta R}\left[ g_\alpha^\prime\left(y,z\right)^2-\left(u(y,z)-g(y,z;\textbf{h})\right)g_\alpha^{\prime\prime}\left(y,z\right) \right]dz\\
&+&\sum_{\beta\neq\alpha}(-1)^{\beta}\xi_\beta(\Pi_\beta(y,z))\int_{-\delta R}^{\delta R}g_\alpha^\prime\left(y,z\right)g_\beta^\prime\left(y,z\right)\nabla d_\beta\left(y,z\right)\cdot\nabla d_\alpha\left(y,z\right) dz.
\end{eqnarray*}
By Lemma \ref{O(1) scale}, there exists a $\delta>0$ such that for any $\|\textbf{h}\|_{\mathcal{X}}<\delta$,
\[\int_{-\delta R}^{\delta R}\left[ g_\alpha^\prime\left(y,z\right)^2-\left(u(y,z)-g(y,z;\textbf{h})\right)g_\alpha^{\prime\prime}\left(y,z\right) \right]dz\geq \frac{\sigma_0}{2},\]
\[\Big|\int_{-\delta R}^{\delta R}g_\alpha^\prime\left(y,z\right)g_\beta^\prime\left(y,z\right)\nabla d_\beta\left(y,z\right)\cdot\nabla d_\alpha\left(y,z\right) dz\Big|\ll 1.\]
Thus in this ball $DF(\textbf{h})$ is diagonal dominated and invertible with $\|DF(\textbf{h})^{-1}\|_{\mathcal{X}\mapsto\mathcal{X}}\leq C$.
By Lemma \ref{O(1) scale}, for all $\varepsilon$ small enough, $\|F(0)\|_{\mathcal{X}}<<1$. The existence of $\textbf{h}$ then follows from the inverse function theorem.
\end{proof}
\begin{rmk}
The proof shows that $\|\textbf{h}\|_{L^\infty(B_R^{n-1})}=o(1)$. By differentiating \eqref{2 orthogonal condition}, we can show that
$\|\textbf{h}\|_{C^3(B_R^{n-1})}=o(1)$.
\end{rmk}

Denote $g_\ast(y,z):=g(y,z;\textbf{h}(y))$, where $\textbf{h}$ is as in the previous lemma.
Let
\[\phi:=u-g_\ast.\]

In the Fermi coordinates with respect to $\Gamma_\alpha$,
\begin{eqnarray*}
\Delta g_\alpha &=&g_\alpha^{\prime\prime}-(-1)^{\alpha-1}g_\alpha^\prime H^\alpha -(-1)^{\alpha-1}g_\alpha^\prime \Delta_zh_\alpha+g_\alpha^{\prime\prime}|\nabla_zh_\alpha|^2\\
&=&W^\prime(g_\alpha)+\xi_\alpha+(-1)^{\alpha}g_\alpha^\prime \mathcal{R}_{\alpha,1}+g_\alpha^{\prime\prime}\mathcal{R}_{\alpha,2},
\end{eqnarray*}
where
\[\xi_\alpha(y,z)=\bar{\xi}\left((-1)^{\alpha-1}(z-h_\alpha(y))\right),\]
\[
\mathcal{R}_{\alpha,1}(y,z):=H^\alpha(y,z)+\Delta_z h_\alpha(y),
\]
\[\mathcal{R}_{\alpha,2}(y,z):=|\nabla_z h_\alpha(y)|^2.\]

In the Fermi coordinates with respect to $\Gamma_\alpha$, the equation for $\phi$ reads as
\begin{eqnarray}\label{error equation}
&&\Delta_z\phi-H^\alpha(y,z)\partial_z\phi+\partial_{zz}\phi \nonumber\\
&=&W^\prime(g_\ast+\phi)-\sum_{\beta=1}^Q
 W^\prime(g_\beta) -(-1)^{\alpha}g_\alpha^\prime\left[H^\alpha(y,z)+\Delta_z h_\alpha(y)\right]-g_\alpha^{\prime\prime}|\nabla_zh_\alpha|^2\\
&-&\sum_{\beta\neq\alpha} \left[(-1)^{\beta}g_\beta^\prime\mathcal{R}_{\beta,1}\left(\Pi_\beta(y,z),d_\beta(y,z)\right)+
g_\beta^{\prime\prime}\mathcal{R}_{\beta,2}\left(\Pi_\beta(y,z),d_\beta(y,z)\right)\right] -\sum_{\beta}\xi_\beta  \nonumber\\
&=&W^{\prime\prime}(g_\ast)\phi+\mathcal{R}(\phi)+\left[W^\prime(g_\ast)-\sum_{\beta=1}^Q
 W^\prime(g_\beta)\right] -(-1)^{\alpha}g_\alpha^\prime\left[H^\alpha(y,z)+\Delta_z h_\alpha(y)\right]-g_\alpha^{\prime\prime}|\nabla_zh_\alpha|^2 \nonumber\\
&-&\sum_{\beta\neq\alpha} \left[(-1)^{\beta}g_\beta^\prime\mathcal{R}_{\beta,1}\left(\Pi_\beta(y,z),d_\beta(y,z)\right)+
g_\beta^{\prime\prime}\mathcal{R}_{\beta,2}\left(\Pi_\beta(y,z),d_\beta(y,z)\right)\right]-\sum_{\beta}\xi_\beta. \nonumber
\end{eqnarray}
In the above we have denoted
\[\mathcal{R}(\phi):=W^\prime(g_\ast+\phi)-W^\prime(g_\ast)-W^{\prime\prime}(g_\ast)\phi=O(\phi^2).\]

\subsection{Interaction terms}

In this subsection we establish several estimates on the interaction term between different components,
\[W^\prime(g_\ast)-\sum_{\beta=1}^Q W^\prime(g_\beta).\]

\begin{lem}\label{interaction term}
In $\mathcal{M}^4_\alpha$,
\begin{eqnarray}\label{interaction}
W^\prime(g_\ast)-\sum_\beta W^\prime(g_\beta)
&=&\left[W^{\prime\prime}(g_\alpha)-2\right]\left[g_{\alpha-1}-(-1)^\alpha\right]+\left[W^{\prime\prime}(g_\alpha)-2\right]\left[g_{\alpha+1}+(-1)^{\alpha}\right]\\
&+&O\left(e^{-2\sqrt{2}d_{\alpha-1}}+e^{2\sqrt{2}d_{\alpha+1}}\right)+O\left(e^{-\sqrt{2}d_{\alpha-2}-\sqrt{2}|d_\alpha|}+e^{\sqrt{2}d_{\alpha+2}-\sqrt{2}|d_\alpha|}\right). \nonumber
\end{eqnarray}
\end{lem}
\begin{proof}
In $\mathcal{M}^4_\alpha$,
\[g_\ast=g_\alpha+\sum_{\beta<\alpha}\left[g_\beta-(-1)^{\beta-1}\right]+\sum_{\beta>\alpha}\left[g_\beta+(-1)^{\beta-1}\right].\]
By Lemma \ref{O(1) scale}, $g_\beta-(-1)^{\beta-1}$ (for $\beta<\alpha$) and $g_\beta+(-1)^{\beta-1}$ (for $\beta>\alpha$) are all small quantities.

By the Taylor expansion,
\begin{eqnarray*}
W^\prime(g_\ast)&=&W^\prime\left(g_\alpha\right)+
W^{\prime\prime}\left(g_\alpha\right)\left[\sum_{\beta<\alpha}\left(g_\beta-(-1)^{\beta-1}\right)+\sum_{\beta>\alpha}\left(g_\beta+(-1)^{\beta-1}\right)
\right]\\
&&+\sum_{\beta<\alpha}O\left(|g_\beta-(-1)^{\beta-1}|^2\right)+\sum_{\beta>\alpha}O\left(|g_\beta+(-1)^{\beta-1}|^2\right).
\end{eqnarray*}

On the other hand, for $\beta<\alpha$,
\[W^\prime(g_\beta)=2\left(g_\beta-(-1)^{\beta-1}\right)
+O\left(|g_\beta-(-1)^{\beta-1}|^2\right),\]
and for $\beta>\alpha$,
\[W^\prime(g_\beta)=2\left(g_\beta+(-1)^{\beta-1}\right)
+O\left(|g_\beta+(-1)^{\beta-1}|^2\right).\]

Combining these expansions we get
\begin{eqnarray*}
W^\prime(g_\ast)-\sum_{\beta=1}^Q W^\prime(g_\beta)
&=&\sum_{\beta<\alpha}\left[W^{\prime\prime}(g_\alpha)-2\right]\left(g_\beta-(-1)^{\beta-1}\right)
+\sum_{\beta<\alpha}O\left(|g_\beta-(-1)^{\beta-1}|^2\right)\\
&+&\sum_{\beta>\alpha}\left[W^{\prime\prime}(g_\alpha)-2\right]\left(g_\beta+(-1)^{\beta-1}\right)
+\sum_{\beta>\alpha}O\left(|g_\beta+(-1)^{\beta-1}|^2\right).
\end{eqnarray*}
Using the fact that
\[|W^{\prime\prime}(g_\alpha)-W^{\prime\prime}(1)|\lesssim 1-g_\alpha^2\lesssim e^{-\sqrt{2}|d_\alpha|}\]
and similar estimates on $g_\beta$, we get the main order terms and estimates on remainder terms in \eqref{interaction} .
\end{proof}

The following upper bound on the interaction term will be used a lot in the below.

\begin{lem}\label{upper bound on interaction}
In $\mathcal{M}^4_\alpha$,
\[\big|W^\prime(g_\ast(y,z))-\sum_{\beta=1}^QW^\prime(g_\beta(y,z))\big|\lesssim e^{-\sqrt{2}D_{\alpha}(y)}+\varepsilon^2.\]
\end{lem}
\begin{proof}
We need to estimate those terms in the right hand side of \eqref{interaction}. To simplify notations, assume $(-1)^{\alpha-1}=1$.
\begin{itemize}
\item  There exists a constant $C$ depending only on $W$ such that
\[\Big|\left[W^{\prime\prime}(g_\alpha)-2\right]\left(g_{\alpha-1}+1\right)\Big|\lesssim e^{-\sqrt{2}d_{\alpha-1}-\sqrt{2}|d_\alpha|}.\]
Note that $d_{\alpha-1}>0$ in $\mathcal{M}_\alpha^4$. If one of $d_{\alpha-1}$ and $|d_\alpha|$ is larger than $\sqrt{2}|\log\varepsilon|$, we have
\[e^{-\sqrt{2}d_{\alpha-1}-\sqrt{2}|d_\alpha|}\leq\varepsilon^2,\]
and we are done.

If both $d_{\alpha-1}$ and $|d_\alpha|$ are not larger than $\sqrt{2}|\log\varepsilon|$, by Lemma \ref{comparison of distances},
\[d_{\alpha-1}(y,z)=d_{\alpha-1}(y,0)+d_\alpha(y,z)+O(\varepsilon^{1/3}).\]
Therefore
\[e^{-\sqrt{2}d_{\alpha-1}-\sqrt{2}|d_\alpha|}\leq2e^{-\sqrt{2}d_{\alpha-1}(y,0)}.\]

\item In the same way we get
\[\Big|\left[W^{\prime\prime}(g_\alpha)-W^{\prime\prime}(1)\right]\left(g_{\alpha+1}-1\right)\Big|\lesssim  e^{\sqrt{2}d_{\alpha+1}-\sqrt{2}|d_\alpha|}\lesssim \varepsilon^2+e^{\sqrt{2}d_{\alpha+1}(y,0)}.\]

\item  If $|d_{\alpha+1}(y,z)|\geq |\log\varepsilon|$, then $e^{-2\sqrt{2}d_{\alpha+1}}\leq\varepsilon^2$. If $|d_{\alpha+1}(y,z)|\leq |\log\varepsilon|$, we also have $|z|=|d_\alpha(y,z)|\leq |\log\varepsilon|+4$. Hence by Lemma \ref{comparison of distances},
    \[d_{\alpha+1}(y,z)=d_{\alpha+1}(y,0)+d_\alpha(y,z)+O(\varepsilon^{1/3}).\]
    Because $|d_\alpha(y,z)|<|d_{\alpha+1}(y,z)|+4$, we get
\[d_{\alpha+1}(y,z)\leq\frac{1}{2}d_{\alpha+1}(y,0)+4.\]
Therefore
\[e^{2\sqrt{2}d_{\alpha+1}(y,z)}\leq  e^4 e^{\sqrt{2}d_{\alpha+1}(y,0)}.\]

\item Similarly
\[e^{-2\sqrt{2}d_{\alpha-1}(y,z)}\leq\varepsilon^2+e^4 e^{-\sqrt{2}d_{\alpha-1}(y,0)}.\]
\item As in the first two cases,
\[e^{-\sqrt{2}d_{\alpha-2}-\sqrt{2}|d_\alpha|}+e^{\sqrt{2}d_{\alpha+2}-\sqrt{2}|d_\alpha|}\lesssim \varepsilon^2+e^{-\sqrt{2}d_{\alpha-1}(y,0)}+e^{\sqrt{2}d_{\alpha+1}(y,0)}.\]

\end{itemize}
Putting all of these together we finish the proof.
\end{proof}

The H\"{o}lder norm of interaction terms can also be estimated in the following way.
\begin{lem}\label{Holder bound on interaction}
For any $(y,z)\in\mathcal{M}_\alpha^3$,
\[\Big\|W^\prime(g_\ast)-\sum_{\beta=1}^Q(-1)^{\beta-1}W^\prime(g_\beta)\Big\|_{C^{\theta}(B_1(y,z))}\lesssim \sup_{B_1(y)}e^{-\sqrt{2}D_\alpha}+\varepsilon^2.\]
\end{lem}
\begin{proof}
We only need to notice that, for any $(y,z)\in\mathcal{M}_\alpha^3$ and any $\beta\in\{1,\cdots, Q\}$,
\begin{equation}\label{2.9.1}
\|g_\beta^2-1\|_{C^{\theta}(B_1(y,z))}\lesssim\|g_\beta^2-1\|_{Lip(B_1(y,z))}\lesssim e^{-\sqrt{2}|d_\beta(y,z)|}.
\end{equation}
Then we can proceed as in the previous lemma to conclude the proof.
\end{proof}

\subsection{Controls on $\textbf{h}$ using $\phi$}

The choice of optimal approximation in Subsection \ref{sec optimal approximation} has the advantage that
 $h$ is controlled by $\phi$. This will allow us to iterate various elliptic estimates in the below.
\begin{lem}\label{control on h_0}
For each $\alpha$,
\[|h_\alpha(y)|\lesssim |\phi(y,0)|+e^{-\sqrt{2}D_\alpha(y)},\]
\[|\nabla h_\alpha(y)|\lesssim|\nabla \phi(y,0)|+o\left(e^{-\sqrt{2}D_\alpha(y)}\right),\]
\[|\nabla^2 h_\alpha(y)|\lesssim|\nabla^2 \phi(y,0)|+o\left(e^{-\sqrt{2}D_\alpha(y)}\right),\]
\[\|\nabla^2h_\alpha\|_{C^{\theta}(B_1(y))}\lesssim\|\nabla^2\phi\|_{C^{\theta}(B_1(y,0))}+\sup_{B_1(y)}e^{-\sqrt{2}D_\alpha}.\]
\end{lem}
\begin{proof}
Fix an $\alpha\in\{1,\cdots,Q\}$.  In the Fermi coordinates with respect to $\Gamma_\alpha$, because $u(y,0)=0$,
\begin{eqnarray}\label{representation of h_alpha}
\phi(y,0)=-\bar{g}\left((-1)^\alpha h_\alpha(y)\right)&-&\sum_{\beta<\alpha}\left[\bar{g}\left((-1)^{\beta-1}\left(d_\beta(y,0)-h_\beta(\Pi_\beta(y,0))\right)\right)-(-1)^{\beta-1}\right] \nonumber\\
&  -&\sum_{\beta>\alpha}\left[\bar{g}\left((-1)^{\beta-1}\left(d_\beta(y,0)-h_\beta(\Pi_\beta(y,0))\right)\right)+(-1)^{\beta-1}\right].
\end{eqnarray}
Note that for $\beta\neq\alpha$,  $|h_\beta(\Pi_\beta(y,0))|\ll 1$. Thus
\begin{equation}\label{4.1}
|h_\alpha(y)| \lesssim |\phi(y,0)|+\sum_{\beta\neq\alpha}e^{-\sqrt{2}d_\beta(y,0)} \lesssim  |\phi(y,0)|+e^{-\sqrt{2}D_\alpha(y)}.
\end{equation}

Differentiating \eqref{representation of h_alpha}, we get
\[
\nabla_0\phi(y,0)=(-1)^{\alpha+1} \bar{g}^\prime\left((-1)^\alpha h_\alpha(y)\right)\nabla_0h_\alpha(y)+\sum_{\beta\neq\alpha}(-1)^\beta g_\beta^\prime(y,0)\nabla_0\left[ d_\beta(y,0)-h_\beta(\Pi_\beta(y,0))\right],
\]
and
\begin{eqnarray*}
\nabla_0^2\phi(y,0)
&=&(-1)^{\alpha+1} \bar{g}^\prime\left((-1)^\alpha h_\alpha(y)\right)\nabla_0^2h_\alpha(y)-
\bar{g}^{\prime\prime}\left((-1)^\alpha h_\alpha(y)\right)\nabla_0h_\alpha(y)
\otimes\nabla_0h_\alpha(y)\\
&+&\sum_{\beta\neq\alpha}(-1)^{\beta} g_\beta^\prime(y,0)\nabla_0^2\left[ d_\beta(y,0)-h_\beta(\Pi_\beta(y,0))\right]\\
&-&\sum_{\beta\neq\alpha}
g_\beta^{\prime\prime}(y,0)\nabla_0\left[ d_\beta(y,0)-h_\beta(\Pi_\beta(y,0))\right]
\otimes\nabla_0\left[ d_\beta(y,0)-h_\beta(\Pi_\beta(y,0))\right].
\end{eqnarray*}

Note that $|\nabla h_\beta|=o(1)$ and by Lemma \ref{comparison of distances}, if $\bar{g}^\prime(d_\beta(y,0)-h_\beta(\Pi_\beta(y,0))\neq0$,
\[|\nabla_0d_\beta|=\sqrt{1-\nabla d_\beta\cdot\nabla d_\alpha}=O(\varepsilon^{1/6}).\]
Thus
\begin{eqnarray}\label{4.2}
|\nabla_0 h_\alpha(y)|&\lesssim&|\nabla_0 \phi(y,0)|+O(\varepsilon^{1/6}+|\nabla h_\beta(\Pi_\beta(y,0))|)O\left(e^{\sqrt{2}d_{\alpha+1}(y,0)}+e^{-\sqrt{2}d_{\alpha-1}(y,0)}\right) \nonumber\\
&\lesssim&|\nabla_0 \phi(y,0)|+o\left(e^{-\sqrt{2}D_\alpha(y)}\right).
\end{eqnarray}

Similarly, because
$|\nabla^2 h_\beta|=o(1)$ and recalling that $\nabla^2d_\beta$ is the second fundamental form of $\Gamma_{\beta,z}$,
\[|\nabla_0^2d_\beta|\leq|\nabla^2d_\beta|=O(\varepsilon),\]
we have
\begin{eqnarray}\label{4.3}
|\nabla_0^2 h_\alpha(y)|&\lesssim&|\nabla_0^2 \phi(y,0)|+|\nabla_0 h_\alpha(y)|^2+e^{-\sqrt{2}D_\alpha(y)}\left[\varepsilon^{\frac{1}{3}}+\sum_{\beta}\sup_{B_{\varepsilon^{1/3}}}\left(|\nabla^2h_\beta|+|\nabla h_\beta|\right)\right] \nonumber\\
&\lesssim&|\nabla_0^2 \phi(y,0)|+|\nabla_0 h_\alpha(y)|^2+o\left(e^{-\sqrt{2}D_\alpha(y)}\right).
\end{eqnarray}

Finally, by the above formulation and \eqref{2.9.1},
we get a control on $\|\nabla^2h_\alpha\|_{C^{\theta}(B_1(y))}$ using
$\|\nabla^2\phi\|_{C^{\theta}(B_1(y,0))}+\sup_{B_1(y)}e^{-\sqrt{2}D_\alpha}$.
\end{proof}

\section{A Toda system}\label{sec Toda system}
\setcounter{equation}{0}

In the Fermi coordinates with respect to $\Gamma_\alpha$, multiplying \eqref{error equation} by $g_\alpha^\prime$ and integrating in
$z$ leads to
\begin{eqnarray}\label{H eqn}
&&\int_{-\delta R}^{\delta R}g_\alpha^\prime\Delta_z\phi-H^\alpha(y,z)g_\alpha^\prime\partial_z\phi+g_\alpha^\prime\partial_{zz}\phi \nonumber\\
&=&\int_{-\delta R}^{\delta R} \left[W^\prime(g_\ast+\phi)-\sum_\beta W^\prime(g_\beta)\right] g_\alpha^\prime-(-1)^{\alpha}\int_{-\delta R}^{\delta R}\left[H^\alpha(y,z)+\Delta_zh_\alpha(y)\right]g_\alpha^\prime(z)^2\\
&-&\int_{-\delta R}^{\delta R} g_\alpha^{\prime\prime}g_\alpha^\prime|\nabla_z
h_\alpha|^2 -\sum_{\beta\neq\alpha}(-1)^{\beta}\int_{-\delta R}^{\delta R}g_\alpha^\prime g_\beta^\prime \mathcal{R}_{\beta,1} -\sum_{\beta\neq\alpha} \int_{-\delta R}^{\delta R} g_\beta^{\prime\prime}g_\alpha^\prime\mathcal{R}_{\beta,2}  -\sum_{\beta}\int_{-\delta R}^{\delta R}\xi_\beta g_\alpha^\prime .      \nonumber
\end{eqnarray}

By the calculation in Appendix \ref{sec derivation of Toda}, we obtain
\begin{eqnarray}\label{Toda system}
H^\alpha(y,0)+\Delta_0h_\alpha(y)
&=&\frac{4}{\sigma_0}\left[ A_{(-1)^\alpha}^2e^{-\sqrt{2}d_{\alpha-1}(y,0)}-A_{(-1)^{\alpha-1}}^2e^{\sqrt{2}d_{\alpha+1}(y,0)}\right]+O(\varepsilon^2)\nonumber\\
&+&O\left(|h_\alpha(y)|+|h_{\alpha-1}(\Pi_{\alpha-1}(y,z))|+\varepsilon^{1/3}\right)e^{-\sqrt{2}d_{\alpha-1}(y,0)} \nonumber\\
&+&O\left(|h_\alpha(y)|+|h_{\alpha+1}(\Pi_{\alpha+1}(y,z))|+\varepsilon^{1/3}\right)e^{\sqrt{2}d_{\alpha+1}(y,0)}  \\
&+&O(e^{-\frac{3\sqrt{2}}{2}d_{\alpha-1}(y,0)})+O(e^{\frac{3\sqrt{2}}{2}d_{\alpha+1}(y,0)})+O(e^{-\sqrt{2}d_{\alpha-2}(y,0)})+O(e^{\sqrt{2}d_{\alpha+2}(y,0)})\nonumber\\
&+&\sum_{\beta\neq\alpha}|d_\beta(y,0)|e^{-\sqrt{2}|d_\beta(y,0)|}\left[\sup_{B_{\varepsilon^{1/3}}(y)}|H^\beta+\Delta^\beta_0 h_\beta|+\sup_{B_{\varepsilon^{1/3}}(y)}
|\nabla h_\beta|^2\right]\nonumber\\
&+&\sup_{(-6|\log\varepsilon|,6|\log\varepsilon|)}\left(|\nabla_y^2\phi(y,z)|^2+|\nabla_y\phi(y,z)|^2+|\phi(y,z)|^2\right). \nonumber
\end{eqnarray}

By this equation we get an upper bound on $H^\alpha(y,0)+\Delta_0h_\alpha(y)$.
\begin{lem}\label{lem 11.1}
\begin{equation}\label{5.1}
\sup_{B_r}\big|H^\alpha(y,0)+\Delta_0h_\alpha(y)\big|\lesssim\sup_{B_{r+1}}e^{-\sqrt{2}D_{\alpha}}+\varepsilon^2+\|\phi\|_{C^{2,\theta}(\mathcal{D}_{r+1})}^2
+\sum_{\beta\neq\alpha}\sup_{B_{r+1}}\left[|H^\beta+\Delta^\beta h_\beta|^2+e^{-2\sqrt{2}D_\beta}\right].
\end{equation}
\end{lem}
\begin{proof}
In the right hand side of \eqref{Toda system}, those terms in the first four lines are bounded by $O\left(e^{-\sqrt{2}D_{\alpha}}+\varepsilon^2\right)$.

If $d_\beta(y,0)>2|\log\varepsilon|$, the terms in the fifth line is controlled by $O(\varepsilon^2)$.
If $d_\beta(y,0)<2|\log\varepsilon|$, using the Cauchy inequality, they are controlled by
\begin{eqnarray*}
&&|d_\beta(y,0)|^2e^{-2\sqrt{2}|d_\beta(y,0)|}+\sup_{B_{\varepsilon^{1/3}}(y)}|H^\beta+\Delta^\beta_0 h_\beta|^2+\sup_{B_{\varepsilon^{1/3}}(y)}
|\nabla h_\beta|^4\\
&\lesssim& e^{-\sqrt{2}|d_\beta(y,0)|}+\sup_{B_{\varepsilon^{1/3}}(y)}|H^\beta+\Delta^\beta_0 h_\beta|^2+\sup_{\mathcal{M}_{\beta}^0(r+1)}
|\nabla \phi|^4,
\end{eqnarray*}
where we have used the fact that $|d_\beta(y,0)|\gg 1$, Lemma \ref{control on h_0} and the fact that $B_{\varepsilon^{1/3}}^\beta(y,0)\subset\mathcal{M}_{\beta}^0(r+1)$ (by Lemma \ref{comparison of distances}).

The term in the last line of \eqref{Toda system} is controlled by $\|\phi\|_{C^{2,\theta}(\mathcal{D}_{r+1})}^2$.
\end{proof}

\section{ $C^{1,\theta}$ estimate on $\phi$}\label{sec first order}
\setcounter{equation}{0}

In this section we prove the following $C^{1,\theta}$ estimate on $\phi$.
\begin{prop}
There exist constants $L>0$, $\sigma(L)\ll 1$ and $C(L)$ such that
\begin{eqnarray}\label{First order estimate}
\|\phi\|_{C^{1,\theta}(\mathcal{D}_\alpha(r))}
&\leq&\sigma(L)\|\phi\|_{C^{2,\theta}(\mathcal{D}(r+4L)}+C(L)\varepsilon^2+C(L)\sup_{B_{r+4L}}e^{-\sqrt{2}D_\alpha(y)}  \\
&+&\sigma(L)\sum_{\beta=1}^Q\sup_{B_{r+4L}(y)}\big|H^\beta+\Delta_0^\beta h_\beta\big| +C(L)\sum_{\beta\neq\alpha}\sup_{B_{r+4L}}e^{-4\sqrt{2}D_\beta}.  \nonumber
\end{eqnarray}
\end{prop}

To prove this proposition, fix a large constant $L>0$ and define
\[\mathcal{N}_\alpha^1(r):=\{-L<d_\alpha<L\}\cap\mathcal{M}_\alpha^0(r),\quad  \mbox{and} \quad \mathcal{N}_\alpha^2(r):=\{d_\alpha>L/2\}\cap\mathcal{M}_\alpha^0(r).\]
We will estimate the $C^{1,\theta}$ norm of $\phi$ in $\mathcal{N}_\alpha^2(r)$ and $\mathcal{N}_\alpha^2(r)$ separately.

\subsection{$C^{1,\theta}$ estimate in $\mathcal{N}_\alpha^2(r)$}

In $\mathcal{N}_\alpha^2(r)$, the equation for $\phi$ can be written in the following way.
\begin{lem}\label{lem 8.1}
In $\mathcal{N}_\alpha^2(r)$,
\[\Delta_z\phi-H^\alpha(y,z)\partial_z\phi+\partial_{zz}\phi=\left(2+O(e^{-cL})\right)\phi+E_\alpha^2,\]
where
\begin{eqnarray*}
|E_\alpha^2(y,z)|&\lesssim&\varepsilon^2+e^{-\sqrt{2}D_\alpha(y)}+|\nabla_0^2 h_\alpha(y)|^2+|\nabla_0 h_\alpha(y)|^2 +e^{-cL}\big|H^\alpha(y,0)+\Delta_0 h_\alpha(y)\big|\\
&+&\sum_{\beta\neq\alpha}\sup_{B_{\varepsilon^{1/3}}(y)}\left[|H^\beta+\Delta_0^\beta h_\beta|^2+|\nabla h_\beta|^4+|\nabla^2 h_\beta|^4\right].
\end{eqnarray*}
\end{lem}
The proof can be found in Appendix \ref{sec proof of Lem 8.1}.

\medskip

By standard interior elliptic estimate, we deduce that, for any  $r>1$,
\begin{eqnarray*}
&&\|\phi\|_{C^{1,\theta}(\mathcal{N}_\alpha^2(r))}  \\
&\lesssim& e^{-cL}\|\phi\|_{L^\infty(\mathcal{M}_\alpha^L(r+L)\cap\{|y|=r+L\})}+e^{-cL}\|\phi\|_{L^\infty(\{|y|<r+L,z=L/4\})}+e^{-cL}\|\phi\|_{L^\infty(\{|y|<r+L,z=\rho_\alpha^+(y)+L\})}\\
&+&\sup_{B_{r+L}}e^{-\sqrt{2}D_\alpha}+\varepsilon^2 +\sup_{B_{r+L}}\left(|\nabla_0^2h_\alpha(\tilde{y})|^2+|\nabla_0h_\alpha(\tilde{y})|^2\right)\\
&+&e^{-cL}\sup_{B_{r+L}}\big|H^\alpha+\Delta_0 h_\alpha\big|+\sup_{B_{r+L}}\left[|H^\beta+\Delta_0^\beta h_\beta|^2+|\nabla h_\beta|^4+|\nabla^2 h_\beta|^4\right].
\end{eqnarray*}
Substituting \eqref{5.1} into this estimate, after simplification  we obtain
\begin{eqnarray}\label{First order estimate 1}
\|\phi\|_{C^{1,\theta}(\mathcal{N}_\alpha^2(r))}
&\leq& \sigma(L)\|\phi\|_{C^{2,\theta}(\mathcal{D}(r+L))}+C(L)\sup_{B_{r+L}(y)}e^{-\sqrt{2}D_\alpha}+C(L)\varepsilon^2  \\
&+&\sigma(L)\sum_{\beta=1}^Q\sup_{B_{r+L}(y)}\big|H^\beta+\Delta_0^\beta h_\beta\big| +C(L)\sum_{\beta\neq\alpha}\sup_{B_{r+L}}e^{-4\sqrt{2}D_\beta}.  \nonumber
\end{eqnarray}
Here $\sigma(L)\lesssim e^{-cL}+\max_\beta\|H^\beta+\Delta_0^\beta h_\beta\|_{L^\infty}\ll 1$.

\subsection{$C^{1,\theta}$ estimate in $\mathcal{N}_\alpha^1(r)$}

In $\mathcal{N}_\alpha^1(r)$, the equation for $\phi$ can be written in the following way.
\begin{lem}\label{lem 8.2}
In $\mathcal{N}_\alpha^1(r)$,
\[\Delta_z\phi-H^\alpha(y,z)\partial_z\phi+\partial_{zz}\phi=W^{\prime\prime}(g_\alpha)\phi+O(\phi^2)-(-1)^{\alpha}g_\alpha^\prime\left[H^\alpha(y,0)+\Delta_0 h_\alpha(y)\right]+E_\alpha^1,
\]
where
\[
|E^1_\alpha(y,z)|\lesssim e^{-\sqrt{2}D_\alpha(y)}+\varepsilon^2+\left(|\nabla^2h_\alpha(y)|^2+|\nabla h_\alpha(y)|^2\right)e^{-\sqrt{2}|z|}.\]
\end{lem}
The proof of this lemma is given in Appendix \ref{sec proof of Lem 8.2}.

\medskip

Take a  function  $\eta\in C_0^\infty(-2L,2L)$ satisfying $\eta\equiv 1$ in $(-L,L)$, $|\eta^\prime|\lesssim L^{-1}$ and $|\eta^{\prime\prime}|\lesssim L^{-2}$. Let $\phi_\alpha(y,z):=\phi(y,z)\eta(z)$ and $\widetilde{\phi}_\alpha(y,z):=\phi_\alpha(y,z)-c_\alpha(y)g_\alpha^\prime(y,z)$, where
\begin{eqnarray}\label{def of c}
c_\alpha(y)&=&\int_{-\delta R}^{\delta R}\phi_\alpha(y,z)g_\alpha^\prime(y,z)dz\\
&=&\int_{-\delta R}^{\delta R}\phi(y,z)\left(\eta(z)-1\right)g_\alpha^\prime(y,z)dz. \quad \quad \mbox{( by \eqref{2 orthogonal condition})} \nonumber
\end{eqnarray}
Therefore we still have the orthogonal condition
\begin{equation}\label{orthogonal condition 6.1}
\int_{-\delta R}^{\delta R}\widetilde{\phi}_\alpha(y,z) g_\alpha^\prime(y,z)dz=0, \quad\forall \ y\in B_R^{n-1}.
\end{equation}

\begin{lem}[Estimates on $c_\alpha$]\label{estimates on c} There exists a constant $\sigma>0$ small such that
\begin{equation}\label{estimates on c 1}
|c_\alpha(y)|\lesssim e^{-\sigma L}\sup_{L<|z|<6|\log\varepsilon|}e^{-(\sqrt{2}-\sigma)|z|}|\phi(y,z)|,
\end{equation}
\begin{equation}\label{estimates on c 2}
|\nabla c_\alpha(y)|\lesssim e^{-\sigma L}\sup_{L<|z|<6|\log\varepsilon|}e^{-(\sqrt{2}-\sigma)|z|}\left(|\phi(y,z)|+|\nabla_y\phi(y,z)|\right),
\end{equation}
\begin{equation}\label{estimates on c 3}
|\nabla^2 c_\alpha(y)|\lesssim e^{-\sigma L}\sup_{L<|z|<6|\log\varepsilon|}e^{-(\sqrt{2}-\sigma)|z|}\left(|\phi(y,z)|+|\nabla_y\phi(y,z)|+|\nabla_y^2\phi(y,z)|\right).
\end{equation}
\end{lem}
\begin{proof}
By \eqref{def of c} and the definition of $\eta$,
\begin{eqnarray*}
|c_\alpha(y)|
&\lesssim& \left(\sup_{L<|z|<6|\log\varepsilon|}e^{-(\sqrt{2}-\sigma)|z|}|\phi(y,z)|\right)\int_L^{+\infty}e^{-\sigma z}dz\\
&\lesssim& e^{-\sigma L}\sup_{L<|z|<6|\log\varepsilon|}e^{-(\sqrt{2}-\sigma)|z|}|\phi(y,z)|.
\end{eqnarray*}

Differentiating \eqref{def of c} gives
\[
\nabla c_\alpha(y)=\int_{-\delta R}^{\delta R}\nabla_y\phi(y,z)\left(\eta(z)-1\right)g_\alpha^\prime(y,z)dz+(-1)^\alpha\nabla h_\alpha(y)\int_{-\delta R}^{\delta R}\phi(y,z)\left(\eta(z)-1\right)g_\alpha^{\prime\prime}(y,z)dz.
\]
\eqref{estimates on c 2} follows as above. The derivation of \eqref{estimates on c 3} is similar.
\end{proof}

In the Fermi coordinates with respect to $\Gamma_\alpha$, the equation satisfied by $\widetilde{\phi}_\alpha$ reads as
\begin{equation}\label{6.2}
\Delta_z\widetilde{\phi}_\alpha-H^\alpha(y,z)\partial_z\widetilde{\phi}_\alpha+\partial_{zz}\widetilde{\phi}_\alpha=W^{\prime\prime}(g_\alpha)\widetilde{\phi}_\alpha+o(\widetilde{\phi}_\alpha)
+\widetilde{c}_\alpha(y)g_\alpha^\prime+\widetilde{E}_\alpha,
\end{equation}
where
\[
\widetilde{c}_\alpha(y)=(-1)^{\alpha-1}\left[H^\alpha(y,0)+\Delta_0h_\alpha(y)\right]-\Delta_0c_\alpha(y),\]
while
\begin{eqnarray*}
|\widetilde{E}_\alpha(y,z)|&\lesssim&\varepsilon^2\eta +e^{-\sqrt{2}D_\alpha(y)}\eta+e^{-\sqrt{2}|z|}\left(|\nabla^2h_\alpha(y)|^2+|\nabla h_\alpha(y)|^2\right)\eta\\
&+&|\phi(y,z)||c_\alpha(y)|g_\alpha^\prime+g_\alpha^\prime\big|1-\eta\big|\big|H^\alpha(y,0)+\Delta_0h_\alpha(y)\big|\\
&+&|\phi|\left[\varepsilon|\eta^\prime|+|\eta^{\prime\prime}|\right]+|\phi_z||\eta^\prime|\\
&+&\varepsilon L\left[|c_\alpha(y)|+|\nabla c_\alpha(y)|+|\nabla^2c_\alpha(y)|\right]e^{-\sqrt{2}|z|}\\
&+&|c_\alpha(y)||\xi_\alpha|+|c_\alpha(y)|\left[|\nabla^2h_\alpha(y)|+|\nabla h_\alpha(y)|\right]e^{-\sqrt{2}|z|}.
\end{eqnarray*}

\medskip

By this expression, Lemma \ref{lem 8.1} and Lemma \ref{estimates on c}, we obtain
\begin{lem}[$L^2$ estimates on $\widetilde{E}_\alpha$]\label{estimates on e}
For any $y$,
\begin{eqnarray*}
\|\widetilde{E}_\alpha(y,\cdot)\|_{L^2(-\delta R,\delta R)}^2&\lesssim&L\varepsilon^4+Le^{-2\sqrt{2}D_\alpha(y)}+ |\nabla^2h_\alpha(y)|^4+|\nabla h_\alpha(y)|^4 \\
&+&e^{-2\sqrt{2}L}\big|H^\alpha(y,0)+\Delta_0h_\alpha(y)\big|^2+\frac{1}{L}\sup_{L<|z|<2L}\left(|\phi(y,z)|^2+|\phi_z(y,z)|^2\right)\\
&+&e^{-2\sigma L}\sup_{L<|z|<6|\log\varepsilon|}e^{-2(\sqrt{2}-\sigma)|z|}\left(|\phi(y,z)|+|\nabla_y\phi(y,z)|+|\nabla_y^2\phi(y,z)|\right)^2.
\end{eqnarray*}
\end{lem}

Next we prove an $L^2$ estimate on $\widetilde{\phi}$.
\begin{lem}
For any $r>0$,
\begin{eqnarray}\label{L2 decay}
\sup_{\tilde{y}\in B_r}\|\widetilde{\phi}_\alpha(\tilde{y},\cdot)\|_{L^2(-\delta R,\delta R)}^2
&\lesssim& e^{-cL} \sup_{\tilde{y}\in B_{r+L}}\|\widetilde{\phi}_\alpha(\tilde{y},\cdot)\|_{L^2(-\delta R,\delta R)}^2+ L\varepsilon^4+L\sup_{\tilde{y}\in B_{r+L}}e^{-2\sqrt{2}D_\alpha} \nonumber\\
&+& e^{-2\sqrt{2}L}\sup_{\tilde{y}\in B_{r+L}}\big|H^\alpha+\Delta_0h_\alpha\big|^2\\
&+&\frac{1}{L}\sup_{B_{r+L}\times\{L<|z|<2L\}}\left(|\phi|^2+|\nabla\phi|^2\right)
+L^5\sup_{B_{r+L}\times\{|z|<2L\}}\left(|\nabla_y^2\phi|^4+|\nabla_y
\phi|^4\right)\nonumber\\
&+&e^{-2\sigma L}\sup_{B_{r+L}\times\{L<|z|<6|\log\varepsilon|\}}e^{-2(\sqrt{2}-\sigma)|z|}\left(|\phi|+|\nabla_y\phi|+|\nabla_y^2\phi|\right)^2. \nonumber
\end{eqnarray}
\end{lem}
\begin{proof}
Multiplying \eqref{6.2} by $\widetilde{\phi}_\alpha$ and integrating in
$z$, we obtain
\[\int_{-\infty}^{+\infty}\widetilde{\phi}_\alpha\Delta_z\widetilde{\phi}_\alpha+H^\alpha(y,z)\partial_z\widetilde{\phi}_\alpha\widetilde{\phi}_\alpha+\partial_{zz}\widetilde{\phi}_\alpha\widetilde{\phi}_\alpha
=\int_{-\infty}^{+\infty}W^{\prime\prime}(g_\alpha)\widetilde{\phi}_\alpha^2+o(\widetilde{\phi}_\alpha^2)+\widetilde{E}_\alpha\widetilde{\phi}_\alpha.
\]
Integrating by parts and applying Theorem \ref{second eigenvalue for 1d} leads to
\begin{eqnarray*}
\int_{-\infty}^{+\infty}\widetilde{\phi}_\alpha\Delta_z\widetilde{\phi}_\alpha
&=&\int_{-\infty}^{+\infty}|\partial_z\widetilde{\phi}_\alpha|^2+W^{\prime\prime}(g_\alpha)\widetilde{\phi}_\alpha^2+o(\widetilde{\phi}_\alpha^2)+\widetilde{E}_\alpha\widetilde{\phi}_\alpha
+\frac{1}{2}\frac{\partial H^\alpha}{\partial z}\widetilde{\phi}_\alpha^2\\
&\geq&\frac{3\mu}{4}\int_{-\infty}^{+\infty}\widetilde{\phi}_\alpha^2-C\int_{-\infty}^{+\infty}\widetilde{E}_\alpha^2.
\end{eqnarray*}
On the other hand, by direct differentiation we also have
\begin{eqnarray*}
\frac{1}{2}\Delta_0\int_{-\infty}^{+\infty}\widetilde{\phi}_\alpha^2&=&\int_{-\infty}^{+\infty}
\widetilde{\phi}_\alpha(y,z)\Delta_0\widetilde{\phi}_\alpha(y,z)+|\nabla_0\widetilde{\phi}_\alpha(y,z)|^2dz\\
&\geq&\int_{-\infty}^{+\infty}\left(\Delta_0\widetilde{\phi}_\alpha-\Delta_z\widetilde{\phi}_\alpha\right)\widetilde{\phi}_\alpha+\frac{3\mu}{4}\int_{-\infty}^{+\infty}\widetilde{\phi}_\alpha^2-C\int_{-\infty}^{+\infty}\widetilde{e}_\alpha^2\\
&\geq&\frac{\mu}{2}\int_{-\infty}^{+\infty}\widetilde{\phi}_\alpha^2-C\int_{-\infty}^{+\infty}\widetilde{E}_\alpha^2-C\varepsilon^2\int_{-\infty}^{+\infty}z^2\left(|\nabla_y^2\widetilde{\phi}_\alpha(y,z)|^2+|
\nabla_y\widetilde{\phi}_\alpha(y,z)|^2\right)dz.
\end{eqnarray*}
This inequality implies that
\begin{eqnarray}\label{L2 decay 1}
\sup_{\tilde{y}\in B_r}\|\widetilde{\phi}_\alpha(\tilde{y},\cdot)\|_{L^2(-\delta R,\delta R)}^2&\lesssim& e^{-cL} \sup_{\tilde{y}\in B_{r+L}}\|\widetilde{\phi}_\alpha(\tilde{y},\cdot)\|_{L^2(-\delta R,\delta R)}^2+\sup_{\tilde{y}\in B_{r+L}}\|\widetilde{E}_\alpha(\tilde{y},\cdot)\|_{L^2(-\delta R,\delta R)}^2\nonumber\\
&+&
\varepsilon^2\sup_{B_{r+L}(y)}\int_{-\infty}^{+\infty}z^2\left(|\nabla^2\widetilde{\phi}_\alpha|^2+|
\nabla\widetilde{\phi}_\alpha|^2\right)dz.
\end{eqnarray}
Note that
\begin{eqnarray*}
|\nabla_y^2\widetilde{\phi}_\alpha(y,z)|+|
\nabla_y\widetilde{\phi}_\alpha(y,z)|
&\lesssim&|\nabla_y^2\phi(y,z)|\eta(z)+|\nabla_y
\phi(y,z)|\eta(z)+\left(|c_\alpha(y)|+|\nabla c_\alpha(y)|+|\nabla^2c_\alpha(y)|\right)e^{-\sqrt{2}|z|} \nonumber \\
&\lesssim& |\nabla_y^2\phi(y,z)|\eta(z)+|\nabla_y
\phi(y,z)|\eta(z)\\
&+&e^{-\sigma L-\sqrt{2}|z|}\sup_{L<|z|<6|\log\varepsilon|}e^{-(\sqrt{2}-\sigma)|z|}\left(|\phi(y,z)|+|\nabla_y\phi(y,z)|+|\nabla_y^2\phi(y,z)|\right).
\end{eqnarray*}
Therefore
\begin{eqnarray*}
&&\int_{-\infty}^{+\infty}z^2\left(|\nabla_y^2\widetilde{\phi}_\alpha(y,z)|^2+|
\nabla_y\widetilde{\phi}_\alpha(y,z)|^2\right)dz\\
&\lesssim& L^3\sup_{|z|<2L}\left(|\nabla_y^2\phi(y,z)|^2+|\nabla_y
\phi(y,z)|^2\right)\\
&+&e^{-2\sigma L}\sup_{L<|z|<6|\log\varepsilon|}e^{-2(\sqrt{2}-\sigma)|z|}\left(|\phi(y,z)|^2+|\nabla_y\phi(y,z)|^2+|\nabla_y^2\phi(y,z)|^2\right).
\end{eqnarray*}
Substituting this and Lemma \ref{estimates on e} into \eqref{L2 decay 1} gives
\begin{eqnarray}\label{L2 decay 2}
\sup_{\tilde{y}\in B_r}\|\widetilde{\phi}_\alpha(\tilde{y},\cdot)\|_{L^2(-\delta R,\delta R)}^2
&\lesssim& e^{-cL} \sup_{\tilde{y}\in B_{r+L}}\|\widetilde{\phi}_\alpha(\tilde{y},\cdot)\|_{L^2(-\delta R,\delta R)}^2+ L\varepsilon^4+L\sup_{  B_{r+L}}e^{-2\sqrt{2}D_\alpha} \nonumber\\
&+&
\sup_{ B_{r+L}}\left[|\nabla^2h_\alpha|^4+|\nabla h_\alpha|^4\right]+ e^{-2\sqrt{2}L}\sup_{  B_{r+L}}\big|H^\alpha+\Delta_0h_\alpha\big|^2\\
&+&\frac{1}{L}\sup_{B_{r+L}\times\{L<|z|<2L\}}\left(|\phi|^2+|\phi_z|^2\right)+L^3\varepsilon^2\sup_{B_{r+L}\times\{|z|<2L\}}\left(|\nabla_y^2\phi|^2+|\nabla_y
\phi|^2\right)\nonumber\\
&+&e^{-2\sigma L}\sup_{B_{r+L}\times\{L<|z|<6|\log\varepsilon|\}}e^{-2(\sqrt{2}-\sigma)|z|}\left(|\phi|+|\nabla_y\phi|+|\nabla_y^2\phi|\right)^2. \nonumber
\end{eqnarray}
The terms involving $h_\alpha$ can be estimated by using Lemma \ref{control on h_0}, while by the Cauchy inequality we have
\[L^3\varepsilon^2\sup_{|z|<2L}\left(|\nabla_y^2\phi(y,z)|^2+|\nabla_y
\phi(y,z)|^2\right)\\
\lesssim L\varepsilon^4+L^5\sup_{|z|<2L}\left(|\nabla_y^2\phi(y,z)|^4+|\nabla_y
\phi(y,z)|^4\right).
\]
Substituting these into \eqref{L2 decay 2} we get \eqref{L2 decay}.
\end{proof}

By standard elliptic estimates we deduce that
\begin{eqnarray*}
&&\|\widetilde{\phi}_\alpha\|_{C^{1,\theta}(B_1(y)\times(-3L/4,3L/4))}\\
&\lesssim&  \|\widetilde{\phi}_\alpha\|_{L^2(B_L(y)\times(-L,L))}+\|\Delta\widetilde{\phi}_\alpha\|_{L^\infty(B_L(y)\times(-L,L))}  \\
&\lesssim&L^{\frac{n-1}{2}}e^{-cL} \sup_{\tilde{y}\in B_{2L}(y)}\|\widetilde{\phi}_\alpha(\tilde{y},\cdot)\|_{L^2(-\delta R,\delta R)}
+L^{\frac{n+2}{2}}\varepsilon^2+L^{\frac{n+2}{2}}\sup_{B_{2L}(y)}e^{-\sqrt{2}D_\alpha}  \\
&+&L^{\frac{n-1}{2}}\sup_{B_{2L}(y)}\big|H^\alpha+\Delta_0h_\alpha\big|+L^{\frac{n-2}{2}}\sup_{B_{2L}(y)\times\{L<|z|<2L\}}\left(|\phi|+|\nabla\phi|\right)\\
&+&L^{\frac{n+6}{2}}\sup_{B_{r+L}\times\{|z|<2L\}}\left(|\nabla_y^2\phi|^2+|\nabla_y
\phi|^2\right)\\
&+&L^{\frac{n}{2}}e^{-\sigma L}\sup_{B_{2L}(y)\times\{L<|z|<6|\log\varepsilon|\}}e^{-(\sqrt{2}-\sigma)|z|}\left(|\phi(y,z)|+|\nabla_y\phi(y,z)|+|\nabla_y^2\phi(y,z)|\right).
\end{eqnarray*}
By using \eqref{5.1} we get a bound on $\sup_{B_{2L}(y)}\big|H^\alpha+\Delta_0h_\alpha\big|$. Hence we have
\begin{eqnarray*}
&&\|\widetilde{\phi}_\alpha\|_{C^{1,\theta}(B_1(y)\times(-3L/4,3L/4))}\\
&\lesssim&L^{\frac{n+1}{2}}e^{-cL} \sup_{B_{3L}(y)\times(-2L,2L)}|\phi|+L^{\frac{n+2}{2}}\varepsilon^2+L^{\frac{n+2}{2}}\sup_{B_{3L}(y)}e^{-\sqrt{2}D_\alpha}\\
&+&L^{\frac{n-1}{2}}\sup_{B_{3L}(y)}\sum_{\beta\neq\alpha}D_\alpha e^{-\sqrt{2}D_\alpha}\left[|H^\beta+\Delta_0^\beta h_\beta|+|\nabla h_\beta|^2\right]\\
&+&L^{\frac{n-2}{2}}\sup_{B_{3L}(y)\times\{L<|z|<2L\}}\left(|\phi|+|\nabla\phi|\right)+L^{\frac{n+6}{2}}\sup_{B_{3L}(y)\times(-2L,2L)}\left(|\nabla^2\phi|^2+|
\nabla\phi|^2\right)\\
&+&L^{\frac{n}{2}}e^{-\sigma L}\sup_{B_{3L}(y)\times\{L<|z|<6|\log\varepsilon|\}}e^{-(\sqrt{2}-\sigma)|z|}\left(|\phi|+|\nabla_y\phi|+|\nabla_y^2\phi|\right).
\end{eqnarray*}

\medskip

Since this estimate holds for any $y$, it implies that
\begin{eqnarray*}
\|\phi\|_{C^{1,\theta}(\mathcal{N}_\alpha^1(r))}
&\leq&Ce^{-cL} \sup_{B_{r+3L}\times(-2L,2L)}|\phi| +CL^{\frac{n+2}{2}}\varepsilon^2+L^{\frac{n+2}{2}}\sup_{B_{3L}(y)}e^{-\sqrt{2}D_\alpha(y)}  \\
&+&CL^{\frac{n-1}{2}}\sup_{B_{r+3L}}\sum_{\beta\neq\alpha}D_\alpha e^{-\sqrt{2}D_\alpha}\left[|H^\beta+\Delta_0^\beta h_\beta|+|\nabla h_\beta|^2\right] \\
&+&CL^{\frac{n-2}{2}}\sup_{B_{r+3L}\times\{L<|z|<2L\}}\left(|\phi(y,z)|+|\nabla\phi_z(y,z)|\right) \\
&+&C
L^{\frac{n+6}{2}}\sup_{B_{r+3L}\times(-2L,2L)}\left(|\nabla^2\phi|^2+|
\nabla\phi|^2\right) \\
&+&CL^{\frac{n}{2}}e^{-\sigma L}\sup_{B_{r+3L}\times\{L<|z|<6|\log\varepsilon|\}}e^{-(\sqrt{2}-\sigma)|z|}\left(|\phi|+|\nabla_y\phi|+|\nabla_y^2\phi|\right).
\end{eqnarray*}
As before, this can be written as
\begin{eqnarray*}
\|\phi\|_{C^{1,\theta}(\mathcal{N}_\alpha^1(r))}
&\leq&\sigma(L)\|\phi\|_{C^{2,\theta}(\mathcal{D}(r+4L)}+C(L)\|\phi\|_{C^{2,\theta}(\mathcal{D}(r+4L)} ^2  \\
&+&C(L)\varepsilon^2+C(L)\sup_{B_{r+4L}}e^{-\sqrt{2}D_\alpha(y)}  \\
&+&CL^{\frac{n-1}{2}}\sup_{B_{3L}(y)}\sum_{\beta\neq\alpha}D_\alpha e^{-\sqrt{2}D_\alpha}\left[|H^\beta+\Delta_0^\beta h_\beta|+|\nabla h_\beta|^2\right] \\
&+&CL^{\frac{n-2}{2}}\sup_{L<|z|<2L}\left(|\phi(y,z)|+|\nabla\phi_z(y,z)|\right) .
\end{eqnarray*}
The last term can be estimated by \eqref{First order estimate 1}. After simplification this estimate is rewritten as
\begin{eqnarray}\label{First order estimate 2}
\|\phi\|_{C^{1,\theta}(\mathcal{N}_\alpha^1(r))}
&\leq&\sigma(L)\|\phi\|_{C^{2,\theta}(\mathcal{D}(r+4L)}+C(L)\varepsilon^2+C(L)\sup_{B_{r+4L}}e^{-\sqrt{2}D_\alpha(y)}  \\
&+&\sigma(L)\sum_{\beta=1}^Q\sup_{B_{r+4L}(y)}\big|H^\beta+\Delta_0^\beta h_\beta\big| +C(L)\sum_{\beta\neq\alpha}\sup_{B_{r+4L}}e^{-4\sqrt{2}D_\beta}.  \nonumber
\end{eqnarray}

Combining \eqref{First order estimate 1} with \eqref{First order estimate 2} we obtain \eqref{First order estimate}.

\section{ $C^{2,\theta}$ estimate on $\phi$}
\setcounter{equation}{0}

In the equation of $\phi$, \eqref{error equation}, the coefficients before $\phi$ have a universal Lipschitz bound. Concerning the H\"{o}lder bounds on the right hand side of \eqref{error equation}, we have the following estimates.
\begin{lem}\label{Holder for RHS}
For any $(y,z)\in\mathcal{M}_\alpha$,
\begin{eqnarray*}
\|\Delta\phi-W^{\prime\prime}(g_\ast)\phi\|_{C^{\theta}(B_{2/3}(y,z))}
&\lesssim&\varepsilon^2+\sup_{B_1(y)}e^{-\sqrt{2}D_\alpha}+\|\phi\|_{C^{2,\theta}(B_1(y,z))}^2\\
&+&e^{-\sqrt{2}|z|}\|H^\alpha+\Delta_0h_\alpha\|_{C^{\theta}(B_1(y,0))}\\
&+&e^{-\sqrt{2}|d_\beta(y,z)|}\left(\|\phi\|_{C^{2,\theta}(B_2^\beta(y,0))}^2+\sup_{B_2(y)}e^{-2\sqrt{2}D_\beta}\right)\\
&+&e^{-\sqrt{2}|d_\beta(y,z)|}\|H^\beta+\Delta_0^\beta h_\beta\|_{C^\theta(B_2^\beta(y,0))}.
\end{eqnarray*}
\end{lem}
The proof of this lemma is given in Appendix \ref{sec proof of Lem 9.1}.

By Schauder estimates, for any $(y,z)\in\mathcal{M}_\alpha^0(r)$,
\begin{eqnarray*}
\|\phi\|_{C^{2,\theta}(B_{1/2}(y,z))}&\lesssim&\|\phi\|_{C^{\theta}(B_{2/3}(y,z))}+\|\Delta\phi-W^{\prime\prime}(g_\ast)\phi\|_{C^{\theta}(B_{2/3}(y,z))}\\
&\lesssim&\varepsilon^2+\sup_{B_1(y)}e^{-\sqrt{2}D_\alpha}+\|\phi\|_{C^{2,\theta}(B_1(y,z))}^2+\|\phi\|_{C^{2,\theta}(B_1(y,0))}^2\\
&+&e^{-\sqrt{2}|z|}\|H^\alpha+\Delta_0h_\alpha\|_{C^{\theta}(B_1(y,0))}\\
&+&\sum_{\beta\neq\alpha}e^{-\sqrt{2}|d_\beta(y,z)|}\left(\|\phi\|_{C^{2,\theta}(B_2^\beta(y,0))}^2+\sup_{B_2(y)}e^{-2\sqrt{2}D_\beta}\right)\\
&+&\sum_{\beta\neq\alpha}e^{-\sqrt{2}|d_\beta(y,z)|}\|H^\beta+\Delta_0^\beta h_\beta\|_{C^\theta(B_2^\beta(y,0))}.
\end{eqnarray*}
Because in $\mathcal{M}_\alpha^0$, either $e^{-\sqrt{2}|d_\beta(y,z)|}\lesssim\varepsilon$ or $e^{-\sqrt{2}|d_\beta(y,z)|}\lesssim e^{-\frac{\sqrt{2}}{2}D_\alpha(y)}$, from this we deduce that
\begin{eqnarray}\label{Schauder 1}
\|\phi\|_{C^{2,\theta}(\mathcal{M}^0_\alpha(r))}&\lesssim& \varepsilon^2+ \sup_{B_{r+L}}e^{-\sqrt{2}D_\alpha}+ \|\phi\|_{C^{2,\theta}(\mathcal{M}^0_\alpha(r))}^2+ \|H^\alpha+\Delta_0h_\alpha\|_{C^{\theta}(B_{r+L})}   \\
&+& \sum_{\beta\neq\alpha}\left[\|H^\beta+\Delta_0^\beta h_\beta\|^2_{C^{\theta}(B_{r+L})}+\sup_{B_{r+L}}e^{-4\sqrt{2}D_\beta}+\|\phi\|_{C^{2,\theta}(\mathcal{M}_\beta^0(r+L))}^4\right].\nonumber
\end{eqnarray}
Adding in $\alpha$ leads to
\begin{eqnarray}\label{Schauder}
\|\phi\|_{C^{2,\theta}(\mathcal{D}(r))}&\lesssim& \varepsilon^2+ \sum_{\alpha=1}^Q\sup_{B_{r+L}}e^{-\sqrt{2}D_\alpha}+\sigma\|\phi\|_{C^{2,\theta}(\mathcal{D}_{r+L})}\\
&+& \|H^\alpha+\Delta_0h_\alpha\|_{C^{\theta}(B_{r+L})}+\sum_{\beta\neq\alpha}\|H^\beta+\Delta^\beta_0h_\beta\|_{C^{\theta}(B_{r+L})}^2. \nonumber
\end{eqnarray}

We claim that
\begin{equation}\label{Schauder 2}
\|H^\alpha+\Delta_0h_\alpha\|_{C^{\theta}(B_r)}\leq C\varepsilon^2+C\sum_{\beta}\sup_{B_{r+L}}e^{-\sqrt{2}D_\beta}+\sigma\|\phi\|_{C^{2,\theta}(\mathcal{D}_{r+L})}  +\sigma\sum_{\beta=1}^Q\|H_\beta+\Delta_0^\beta h_\beta\|_{C^{\theta}(B_{r+L})}.
\end{equation}
The proof is given in Appendix \ref{sec proof of 9.3}.

Combining \eqref{Schauder} with \eqref{Schauder 2} we obtain
\[
\|\phi\|_{C^{2,\theta}(\mathcal{D}_r)}+\|H^\alpha+\Delta_0h_\alpha\|_{C^{\theta}(B_r)}\leq C\varepsilon^2+C\sum_{\beta}\sup_{B_{r+L}}e^{-\sqrt{2}D_\beta}+\sigma\left(\|\phi\|_{C^{2,\theta}(\mathcal{D}_{r+L})}+\sum_\beta\|H^\beta+\Delta_0^\beta h_\beta\|_{C^{\theta}(B_{r+L})}\right) .
\]
An iteration of this inequality from $r+K|\log\varepsilon|$ to $r$ (with $K$ large but depending only on $L$ and $\sigma$) gives
\begin{equation}\label{Schauder 3}
\|\phi\|_{C^{2,\theta}(\mathcal{D}_r)}+\sum_\beta\|H^\beta+\Delta_0^\beta h_\beta\|_{C^{\theta}(B_r)}\leq C\varepsilon^2+C\sum_{\beta}\sup_{B_{r+K|\log\varepsilon|}}e^{-\sqrt{2}D_\beta}.
\end{equation}

\section{Improved estimates on horizontal derivatives}
\setcounter{equation}{0}

In this section we prove an improvement on the $C^{1,\theta}$ estimates of horizontal derivatives of $\phi$, $\phi_i:=\partial\phi/\partial y_i$. $1\leq i\leq n-1$.

\medskip

Differentiating \eqref{error equation} in $y_i$, we obtain an equation for $\phi_i:=\phi_{y_i}$,
\begin{equation}\label{horizontal error equation}
\Delta_z\phi_i+\partial_{zz}\phi_i=W^{\prime\prime}(g_\alpha)\phi_i-(-1)^\alpha g_\alpha^\prime \left[H_i(y,0)+\Delta_0 h_{\alpha,i}(y)\right]+E_i,
\end{equation}
where $h_{\alpha,i}(y):=\frac{\partial h_\alpha}{\partial y_i}$ and the remainder term
\begin{eqnarray*}
E_i&=&\left(\Delta_z\phi_i-\partial_{y_i}\Delta_z\phi\right)+\frac{\partial H}{\partial y_i}(y,z)\phi_z+H^\alpha(y,z)\partial_z\phi_{y_i}+\left[W^{\prime\prime}(g_\ast+\phi)-W^{\prime\prime}(g_\alpha)\right]\phi_i\\
&+&(-1)^\alpha\left[W^{\prime\prime}(g_\ast+\phi)-W^{\prime\prime}(g_\alpha)\right]g_\alpha^\prime h_{\alpha,i}(y)\\
&+&\sum_{\beta\neq\alpha}(-1)^\beta\left[W^{\prime\prime}(g_\ast+\phi)-W^{\prime\prime}(g_\beta)\right] g_\beta^\prime\left[\frac{\partial d_\beta}{\partial y_i}-\sum_{j=1}^n h_{\beta,j}\left(\Pi_\beta(y,z)\right)\frac{\partial\Pi_\beta^j}{\partial y_i}(y,z)\right]
\\
 &-& g_\alpha^{\prime\prime}h_{\alpha,i}(y)\left[H^\alpha(y,z)+\Delta_z h_\alpha(y)\right]-(-1)^\alpha g_\alpha^\prime \left[\frac{\partial H}{\partial y_i}(y,z)-\frac{\partial H}{\partial y_i}(y,0)+\frac{\partial}{\partial y_i}\left(\Delta_z h_\alpha(y)\right)-\Delta_0h_{\alpha,i}(y)\right]\\
 &-&(-1)^\alpha g_\alpha^{\prime\prime\prime}|\nabla_zh_\alpha|^2h_{\alpha,i}(y)- g_\alpha^{\prime\prime}\frac{\partial}{\partial y_i}|\nabla_zh_\alpha|^2 \\
&-&\sum_{\beta\neq\alpha} \frac{\partial}{\partial y_i} \left[(-1)^{\beta}g_\beta^\prime\mathcal{R}_{\beta,1}\left(\Pi_\beta(y,z),d_\beta(y,z)\right)+
g_\beta^{\prime\prime}\mathcal{R}_{\beta,2}\left(\Pi_\beta(y,z),d_\beta(y,z)\right)\right]\\
&-&\sum_\beta \xi_\beta^\prime \sum_{j=1}^{n-1}h_{\beta,j}\left(\Pi_\beta(y,z)\right)\frac{\partial\Pi_\beta^j}{\partial y_i}(y,z).
\end{eqnarray*}

 Compared with the orthogonal part in the equation of $\phi$, the order of $E_i$ is increased by one due to the appearance of one more term involving horizontal derivatives of $\phi$. More precisely, the following bound on $E_i$ will be established in Appendix \ref{sec proof of Lem 10.1}.
\begin{lem}\label{bound on remainder term}
In $\mathcal{M}_\alpha^2(r)$,
\begin{eqnarray*}
|E_i|&\lesssim&\varepsilon^2+\|\phi\|_{C^{2,\theta}(\mathcal{D}(r+1))}^2+\sup_{B_{r+1}}e^{-2\sqrt{2}D_\alpha}+\sum_{\beta=1}^Q\sup_{B_{r+1}^\beta}\big|H^\beta+\Delta_0^\beta h_\beta\big|^2+\varepsilon^{1/5}\sup_{B_{r+1}}e^{-\sqrt{2}D_\alpha}\\
&+&\left(\sup_{B_{r+1}}e^{-\frac{\sqrt{2}}{2}D_\alpha}\right)\left[\sum_{\beta\neq\alpha}\sup_{B^\beta_{r+1}}\Big|\nabla H_\beta+\Delta_0^\beta \nabla h_\beta\Big| +\sum_{\beta\neq\alpha}\sup_{B^\beta_{r+1}}e^{-2\sqrt{2}D_\beta}  +\|\phi\|_{C^{2,\theta}(\mathcal{D}(r+1))}\right] .
\end{eqnarray*}
\end{lem}

By differentiating \eqref{2 orthogonal condition} we obtain
\begin{equation}\label{orthogonal condition 2}
\int_{-\delta R}^{\delta R} \phi_i g_\alpha^\prime dz=h_{\alpha,i}(y)\int_{-\delta R}^{\delta R} \phi g_\alpha^{\prime\prime}dz=O\left(\|\phi\|_{C^1(\mathcal{D}(r))}^2+\sup_{B_r}e^{-2\sqrt{2}D_\alpha}\right), \quad\forall y\in B_R^{n-1}.
\end{equation}

Take a large constant $p$ so that $W^{2,p}$ embeds into $C^{1,\theta}$. By noting that
\begin{eqnarray*}
\Big|\int_{-\delta R}^{\delta R}\left(\Delta_z\phi_i-\Delta_0\phi_i\right)g_\alpha^\prime\Big|
&\lesssim&\varepsilon\int_{-\delta R}^{\delta R}\left(|\nabla^2\phi_i(y,z)|+|\nabla \phi_i(y,z)|\right)|z|g_\alpha^\prime\\
&\lesssim&\varepsilon \left[\int_{-\delta R}^{\delta R}\left(|\nabla^2\phi_i(y,z)|+|\nabla \phi_i(y,z)|\right)^2 e^{-\sqrt{2}|z|}dz\right]^{\frac{1}{2}},
\end{eqnarray*}
proceeding as in Section \ref{sec Toda system} we obtain for any $y\in B_r$,
\begin{eqnarray}\label{10.2}
\|H_i^\alpha+\Delta_0 h_{\alpha,i}\|_{L^p(B_1(y))}&\lesssim & \|E_i\|_{L^\infty(\mathcal{M}_\alpha^2(r+2))} \nonumber\\
&+&\varepsilon \left[\int_{B_2(y)}\int_{-\delta R}^{\delta R}\left(|\nabla^2\phi_i(y,z)|+|\nabla \phi_i(y,z)|\right)^p e^{-\sigma|z|}dz\right]^{\frac{1}{p}}\\
&+&\left[\int_{B_2(y)}\int_{-\delta R}^{\delta R}\left(|\nabla^2\phi_i(y,z)|+|\nabla \phi_i(y,z)|\right)^p e^{-\sigma|z|}dz\right]^{\frac{2}{p}} \nonumber.
\end{eqnarray}

 On the other hand, for any $(y,z)\in\mathcal{M}_\alpha^1(r)$, by standard elliptic estimates, we have
\begin{eqnarray*}
\|\phi_i\|_{W^{2,p}(B_2(y,z))}&\lesssim&\|\phi_i\|_{L^p(B_{5/2}(y,z))}+\|\Delta_z\phi_i+\partial_{zz}\phi_i\|_{L^p(B_{5/2}(y,z))}\\
&\lesssim&\|\phi_i\|_{L^\infty(B_3(y,z))}+e^{-\sqrt{2}|z|}\|H^\alpha_i+\Delta_0^\alpha h_{\alpha,i}\|_{L^p(B_3(y))}+\|E_i\|_{L^\infty(\mathcal{M}_\alpha^2(r+3))}.
\end{eqnarray*}
Substituting this into \eqref{10.2} leads to, for any $y\in B_r$,
\begin{eqnarray}\label{10.3}
\sum_\beta\|H^\beta_i+\Delta_0^\beta h_{\beta,i}\|_{L^p(B_1(y))}&\leq& \sigma \sup_{\tilde{y}\in B_2(y)}\sum_\beta\|H^\beta_i+\Delta_0^\beta h_{\beta,i}\|_{L^p(B_1(\tilde{y}))}+C\varepsilon^2\\
&+&C\sum_{\beta=1}^Q\sup_{B_4(y)}e^{-2\sqrt{2} D_\beta}+C\sum_{\beta=1}^Q\sup_{B_4^\beta(y)}\big|H^\beta+\Delta_0^\beta h_\beta\big|^2 \nonumber\\
&+&C\varepsilon^{1/5}\sum_{\beta=1}^Q\sup_{B_4(y)}e^{-\sqrt{2}D_\beta}+C\|\phi\|_{C^{2,\theta}(\mathcal{D}(r+4))}\sup_{B_{r+4}}e^{-\frac{\sqrt{2}}{2}D_\alpha} \nonumber.
\end{eqnarray}

\medskip

An iteration of this estimate gives
\begin{eqnarray*}
\sup_{y\in B_r}\sum_\beta\|H^\beta_i+\Delta_0^\beta h_{\beta,i}\|_{L^p(B_1(y))}&\lesssim&  \varepsilon^2+\sum_{\beta=1}^Q\sup_{B_{r+K|\log\varepsilon|}}e^{-2\sqrt{2} D_\beta}+\sum_{\beta=1}^Q\sup_{B_{r+K|\log\varepsilon|}}\big|H^\beta+\Delta_0^\beta h_\beta\big|^2 \\
&+&\varepsilon^{1/5}\sum_{\beta=1}^Q\sup_{B_{r+K|\log\varepsilon|}}e^{-\sqrt{2}D_\beta}+\|\phi\|_{C^{2,\theta}(\mathcal{D}(r+K|\log\varepsilon|))}\sup_{B_{r+K|\log\varepsilon|}}e^{-\frac{\sqrt{2}}{2}D_\alpha} .
\end{eqnarray*}
Substituting \eqref{Schauder 3} into this we obtain
\begin{equation}\label{10.4}
\sup_{y\in B_r}\sum_\beta\|H^\beta_i+\Delta_0^\beta h_{\beta,i}\|_{L^p(B_1(y))}\lesssim  \varepsilon^2+\sum_{\beta=1}^Q\sup_{B_{r+2K|\log\varepsilon|}}e^{-\frac{3\sqrt{2}}{2} D_\beta}+\varepsilon^{1/5}\sum_{\beta=1}^Q\sup_{B_{r+2K|\log\varepsilon|}}e^{-\sqrt{2}D_\beta}.
\end{equation}

Then  using Lemma \ref{bound on remainder term} and \eqref{orthogonal condition 2} and proceeding as in Section \ref{sec first order}, we obtain
\begin{equation}\label{10.5}
\|\phi_i\|_{C^{1,\theta}(\mathcal{M}_\alpha^0(r))}\lesssim \varepsilon^2+\sum_{\beta}\sup_{B_{r+2K|\log\varepsilon|}}e^{-\frac{3\sqrt{2}}{2}D_\beta}+\varepsilon^{1/5}\sum_{\beta}\sup_{B_{r+2K|\log\varepsilon|}}e^{-\sqrt{2}D_\beta}.
\end{equation}

\section{A lower bound on $D_\alpha$}\label{sec lower bound O(ep)}
\setcounter{equation}{0}

Define
\[A_\alpha(r):=\sup_{B_r}e^{-\sqrt{2}D_\alpha}  \quad \mbox{and} \quad
A(r)=\sum_{\alpha=1}^Q A_\alpha(r).\]

By \eqref{4.3} and \eqref{10.5},
\begin{eqnarray*}
\sup_{B_r}|\Delta_0h^\alpha(y)|\lesssim\varepsilon^2+A(r+K|\log\varepsilon|)^{\frac{3}{2}}+\varepsilon^{1/5}A(r+K|\log\varepsilon|).
\end{eqnarray*}

By \eqref{Toda system}, in $B_r$
\[H^\alpha(y,0)=\frac{4}{\sigma_0}\left[ A_{(-1)^\alpha}^2e^{-\sqrt{2}d_{\alpha-1}(y,0)}-A_{(-1)^{\alpha-1}}^2e^{\sqrt{2}d_{\alpha+1}(y,0)}\right]+o\left(A_\alpha(r+K|\log\varepsilon|)\right)+O(\varepsilon^{4/3}).\]
Because $H^\alpha=O(\varepsilon)$,  an induction on $\alpha$ from $1$ to $Q$ gives
\[A(r)\leq C\varepsilon+\frac{1}{2}A(r+K|\log\varepsilon|).\]
An iteration of this estimate from $r=R$ to $r=5R/6$ gives
\[A(5R/6)\lesssim \varepsilon.\]
In particular, for any $y\in B_{5R/6}$ and $\alpha=1,\cdots, Q$,
\[D_\alpha(y)\geq\frac{\sqrt{2}}{2}|\log\varepsilon|-C.\]

With this lower bound at hand, \eqref{Schauder 3} can now be written as
\[\|\phi\|_{C^{2,\theta}(\mathcal{D}(R/2))}+\sum_{\alpha=1}^Q\|H^\alpha+\Delta_0h_\alpha\|_{C^{\theta}(B_{R/2})}\lesssim\varepsilon,\]
and \eqref{10.5} reads as
\[\|\phi_i\|_{C^{1,\theta}(\mathcal{D}(R/2))}\lesssim \varepsilon^{\frac{6}{5}}, \quad \forall i=1,\cdots, n-1.\]
Therefore by Lemma \ref{control on h_0} we get
\[\sum_{\alpha=1}^Q\|\Delta_0h_\alpha\|_{C^{\theta}(B_{R/2})}\lesssim\varepsilon^{\frac{6}{5}}.\]
Now \eqref{Toda system} reads as
\begin{equation}\label{Toda 2}
\mbox{div}\left(\frac{\nabla f_\alpha(y)}{\sqrt{1+|\nabla f_\alpha(y)|^2}}\right)=\frac{4}{\sigma_0}\left[A_{(-1)^{\alpha-1}}^2e^{-\sqrt{2}d_{\alpha-1}(y)}-A_{(-1)^{\alpha}}^2e^{-\sqrt{2}d_{\alpha+1}(y)}\right]+O\left(\varepsilon^{\frac{6}{5}}\right).
\end{equation}

A remark is in order concerning whether we can improve this lower bound.
\begin{rmk}\label{rmk 11.1}
If there exists a constant $M$, $\alpha\in\{1,\cdots, Q\}$ and $y_\varepsilon\in B_{R/2}$ such that
\[d_{\alpha+1}(y_\varepsilon,0)\leq\frac{\sqrt{2}}{2}|\log\varepsilon|+M.\]
After a rotation and translation, we may assume $y_\varepsilon=0$ and $f_\alpha(0)=0$, $\nabla f_\alpha(0)=0$.

Define
\[\tilde{f}_\beta(y):=f_\beta(\varepsilon^{-1/2}y)-\frac{\sqrt{2}}{2}\left(\beta-\alpha\right)|\log\varepsilon|, \quad \forall \beta\in\{1,\cdots,Q\}.\]
By the curvature bound on $\Gamma_\alpha$ and Lemma \ref{comparison of distances}, for any $\beta\in\{1,\cdots, Q\}$, if $\tilde{f}_\beta(0)$ does not go to $\pm\infty$, then in $B_{R^{2/3}}^{n-1}$,
\[|\nabla f_\beta|\lesssim\varepsilon^{\frac{1}{6}}.\]
Subsisting this into \eqref{Toda 2} and performing a rescaling  we obtain
\[\Delta\tilde{f}_\beta(y)=\frac{4}{\sigma_0}\left[A_{(-1)^{\alpha-1}}^2e^{-\sqrt{2}\left(\tilde{f}_\beta(y)-\tilde{f}_{\beta-1}(y)\right)}-A_{(-1)^{\alpha}}^2e^{-\sqrt{2}\left(\tilde{f}_{\beta+1}(y)-\tilde{f}_\beta(y)\right)}\right]
+O(\varepsilon^{1/6}), \quad  \mbox{in }  B_{R^{1/2}}.\]
Moreover, as $\varepsilon\to0$, $\tilde{f}_\beta$ converges in $C^2_{loc}(\mathbb{R}^{n-1})$ to $\bar{f}_\beta$, which is a nontrivial entire solution to \eqref{Toda entire}.
This blow up procedure will be employed in Section \ref{sec completion of proof}.
\end{rmk}

\section{Multiplicity one case}
\setcounter{equation}{0}

If there is only one connected component of $\{u=0\}$, the estimates can be simplified a lot. For example, now \eqref{Schauder 3} reads as
\[\|\phi\|_{C^{2,\theta}(\mathcal{D}(3R/4))}+\|H+\Delta_0h\|_{C^{\theta}(B_{3R/4}^{n-1})}\lesssim\varepsilon^2.\]
On the other hand, by  Lemma \ref{control on h_0} we get
\[\|h\|_{C^{2,\theta}(B^{n-1}_{3R/4})}\lesssim\varepsilon^2.\]
Hence
\[\|H\|_{C^{\theta}(B^{n-1}_{3R/4})}\lesssim\varepsilon^2.\]
After a scaling, this implies that
\[\Big\|\mbox{div}\left(\frac{\nabla f_\varepsilon}{\sqrt{1+|\nabla f_\varepsilon|^2}}\right)\Big\|_{C^{\theta}(B^{n-1}_{3/4})}\lesssim\varepsilon^{1-\theta}.\]
Because $\sup_{B_{3/4}}|\nabla f_\varepsilon|\leq C$, by standard elliptic estimates on mean curvature equations \cite[Chapter 16]{GT}, we get
\[\|f_\varepsilon\|_{C^{2,\theta}(B_{2/3}^{n-1})}\leq C,\]
where the constant $C$ is independent of $\varepsilon$. This completes the proof of Theorem \ref{main result 4}.

\section{Arbitrary Riemannian metric}
\setcounter{equation}{0}

In the previous analysis the background metric is an Euclidean one. Now we consider an arbitrary Riemannian metric. Since we are concerned with local problems, we will work in the following setting. Assume $B_1(0)\subset\R^n$ is equipped with a $C^3$ Riemannian metric
\[g=g_{ij}(x)dx^i\otimes dx^j.\]
We assume the exponential map is a globally defined diffeomorphism.

Assume
\begin{itemize}
\item $u_\varepsilon\in C^3(B_1)$ is a sequence of solutions to the Allen-Cahn equation
\[\varepsilon\Delta_g u_\varepsilon=\frac{1}{\varepsilon}W^\prime(u_\varepsilon).\]
\item The nodal set $\{u_\varepsilon=0\}$ consists of $Q$ components,
\[\Gamma_{\alpha,\varepsilon}=\{(x^\prime, f_{\alpha,\varepsilon}(x^\prime))\}, \quad \alpha=1,\cdots, Q,\]
where $f_{1,\varepsilon}<f_{2,\varepsilon}<\cdots<f_{Q,\varepsilon}$.

\item For each $\alpha$, the curvature of $\Gamma_{\alpha,\varepsilon}$ is uniformly bounded as $\varepsilon\to0$.

\end{itemize}

By this curvature bound, the Fermi coordinates with respect to $\Gamma_{\alpha,\varepsilon}$ is well defined and $C^2$ in a $\delta$-neighborhood of $\Gamma_{\alpha,\varepsilon}$.  In the Fermi coordinates, the Laplace-Beltrami operator $\Delta_g$ has the same expansion as in Section \ref{sec Fermi coordinates},  for more details see \cite[Section 2]{DKWY}. By denoting $H_{\alpha,\varepsilon}$ the mean curvature of $\Gamma_{\alpha,\varepsilon}$, we get the following Toda system
\begin{equation}\label{Toda scaled}
H_{\alpha,\varepsilon}=\frac{4}{\sigma_0\varepsilon}\left[A_{(-1)^{\alpha-1}}^2e^{-\frac{\sqrt{2}}{\varepsilon}d_{\alpha-1,\varepsilon}}-A_{(-1)^{\alpha}}^2e^{-\frac{\sqrt{2}}{\varepsilon}d_{\alpha+1,\varepsilon}}\right]
+O\left(\varepsilon^{\frac{1}{6}}\right).
\end{equation}
If $\Gamma_{\alpha,\varepsilon}$ collapse to a same minimal hypersurface $\Gamma_\infty$, this system can be written as a Jacobi-Toda system on $\Gamma_\infty$ as in \cite{DKWY}.

\medskip

Now we come to the proof of Theorem \ref{main result 3}.
\begin{proof}[Proof of Theorem \ref{main result 3}]
First by results in Section \ref{sec 2}, we can assume $f_{\alpha,\varepsilon}$ are uniformly bounded in $C^{1,1}(B_2^{n-1})$. Hence we can assume they converge to $f_\infty$ in $C^{1,\theta}(B_2^{n-1})$ for any $\theta\in(0,1)$. Assume there exists $\alpha\in\{1,\cdots,Q\}$ such that $f_{\alpha,\varepsilon}$ do not converge to $f_\infty$ in $C^2(B_1^{n-1})$.

Using the Fermi coordinates $(y,z)$ with respect to $\Gamma_\infty:=\{x_n=f_\infty(x^\prime)\}$, $\Gamma_{\alpha,\varepsilon}$ is represented by the graph $\{z=f_{\alpha,\varepsilon}(y)\}$, where $f_{\alpha,\varepsilon}$ converges to $0$ in $C^1$ but not in $C^2$. Assume $|\nabla^2f_{\alpha,\varepsilon}(y_\varepsilon)|$ does not converge to $0$, then we can preform the same blow up analysis as in Remark \ref{rmk 11.1}, with the base point at $y_\varepsilon$. This procedure results in a nontrivial solution of \eqref{Toda entire}.
\end{proof}

\part{Second order estimate for stable solutions: Proof of Theorem \ref{second order estimate}}

This part is devoted to the proof of Theorem \ref{second order estimate} which in turn implies Theorem  \ref{curvature decay} and Theorem \ref{quadratic curvature decay}. Throughout this part we are in dimension two and $u_\varepsilon$ denotes a solution satisfying the hypothesis of Theorem \ref{second order estimate}. We shall use the stability condition of $u_\varepsilon$ to prove the uniform second order estimates.

\section{A lower bound on the intermediate distance}\label{sec intermediate distance}
\setcounter{equation}{0}

In this section we use the stability condition to prove
\begin{prop}\label{first lower bound}
For any $\sigma>0$, there exists a universal constant $C(\sigma)$ such that for any $\alpha$, $x_1\in(-5/6,5/6)$ and $f_{\alpha,\varepsilon}(x_1)\in(-5/6,5/6)$,
\[\mbox{dist}\left((x_1,f_{\alpha,\varepsilon}(x_1)), \Gamma_{\alpha+1,\varepsilon}\right)\geq\frac{\sqrt{2}-\sigma}{2}\varepsilon|\log\varepsilon|-C(\sigma)\varepsilon.\]
\end{prop}
The idea of proof is to choose a direction derivative of $u_\varepsilon$ to construct a subsolution to the linearized equation and perform a surgery  as in Lemma \ref{finiteness of nodal domains II}.

\subsection{An upper bound on $\mathcal{Q}(\varphi_\varepsilon)$}
Without loss of generality, assume 
$u_\varepsilon>0$ in $\{f_{\alpha,\varepsilon}(x_1)<x_2<f_{\alpha+1,\varepsilon}(x_1)\}\cap\mathcal{C}_{6/7}$. Recall that Lemma \ref{O(1) scale} still holds. Hence near $\{x_2=f_{\alpha,\varepsilon}(x_1)\}$, $\frac{\partial u_\varepsilon}{\partial x_2}>0$, while near $\{x_2=f_{\alpha+1,\varepsilon}(x_1)\}$, $\frac{\partial u_\varepsilon}{\partial x_2}<0$.

Let $\mathcal{D}_{\alpha,\varepsilon}$ be the connected component of $\{\frac{\partial u_\varepsilon}{\partial x_2}>0\}\cap\mathcal{C}_{6/7}$ containing $\{x_2=f_{\alpha,\varepsilon}(x_1)\}$. Let
$\varphi_\varepsilon$ be the restriction of $\frac{\partial u_\varepsilon}{\partial x_2}$ to this domain, extended to be $0$ outside. After such an extension, $\varphi_\varepsilon$ is a nonnegative continuous function and in $\{\varphi_\varepsilon>0\}$ it satisfies the linearized equation
\begin{equation}\label{subsol to linearized}
\varepsilon\Delta\varphi_\varepsilon=\frac{1}{\varepsilon}W^{\prime\prime}(u_\varepsilon)\varphi_\varepsilon.
\end{equation}

Concerning $\mathcal{D}_{\alpha,\varepsilon}$ we have
\begin{lem}\label{lem 6.1}
$\mathcal{D}_{\alpha,\varepsilon}\cap\{|u_\varepsilon|<1-b\}\cap\mathcal{C}_{6/7}$ belongs to an $O(\varepsilon)$ neighborhood of $\{x_2=f_{\alpha,\varepsilon}(x_1)\}$.
\end{lem}
\begin{proof}
By Lemma \ref{O(1) scale}, $\{|u_\varepsilon|<1-b\}\cap\mathcal{C}_{6/7}$ belongs to an $O(\varepsilon)$ neighborhood of $\{x_2=f_{\alpha,\varepsilon}(x_1)\}$, where $\frac{\partial u_\varepsilon}{\partial x_2}>0$.  On the other hand, since $\frac{\partial u_\varepsilon}{\partial x_2}<0$ in a neighborhood of $\{x_2=f_{\alpha\pm1,\varepsilon}(x_1)\}$, $\mathcal{D}_{\alpha,\varepsilon}\cap\mathcal{C}_{6/7}$ belongs to the set $\{f_{\alpha-1,\varepsilon}<x_2<f_{\alpha+1,\varepsilon}\}$.
\end{proof}

Choose an arbitrary point $x_\varepsilon\in\{x_2=f_{\alpha,\varepsilon}(x_1), |x_1|<5/6, |x_2|<5/6\}$.
Take an $\eta_1\in C_0^\infty(B_{1/100}(x_\varepsilon))$, satisfying $\eta_1\equiv 1$ in $B_{1/200}(x_\varepsilon)$ and $|\nabla\eta_1|\leq 1000$.
Multiplying \eqref{subsol to linearized} by $\varphi_\varepsilon\eta_1^2$ and integrating by parts leads to
\begin{eqnarray}\label{2}
\int_{B_{1/100}(x_\varepsilon)}\varepsilon|\nabla\left(\varphi_\varepsilon\eta_1\right)|^2+\frac{1}{\varepsilon}W^{\prime\prime}(u_\varepsilon)\varphi_\varepsilon^2\eta_1^2
=\int_{B_{1/100}(x_\varepsilon)}\varepsilon\varphi_\varepsilon^2|\nabla\eta_1|^2
&\leq&C\int_{B_{1/100}(x_\varepsilon)}\varepsilon\varphi_\varepsilon^2\\
&\leq&C.\nonumber
\end{eqnarray}
In the above we have used the following fact.
\begin{lem} There exists a universal constant $C$ such that
\[\int_{B_{1/100}(x_\varepsilon)}\varepsilon\varphi_\varepsilon^2\leq C.\]
\end{lem}
\begin{proof}
We divide the estimate into two parts: $\{|u_\varepsilon|<1-b\}$ and $\{|u_\varepsilon|>1-b\}$.

{\bf Step 1.}  There exists a universal constant $C$ such that
\begin{equation}\label{6.3}
\int_{\mathcal{D}_{\alpha,\varepsilon}\cap\{|u_\varepsilon|<1-b\}\cap B_{1/50}(x_\varepsilon)}\varepsilon\big|\frac{\partial u_\varepsilon}{\partial x_2}\big|^2\leq C.
\end{equation}

Because $|\nabla u_\varepsilon|\lesssim \varepsilon^{-1}$, this estimate follows from the fact that
\[\Big|\mathcal{D}_{\alpha,\varepsilon}\cap\{|u_\varepsilon|<1-b\}\cap\mathcal{C}_{6/7}\Big|\leq C\varepsilon,\]
which in turn is a consequence of Lemma \ref{lem 6.1}, the co-area formula and the following two facts: (i) for any $t\in[-1+b,1-b]$, $\{u_\varepsilon=t\}$ is a smooth curve with uniformly bounded curvature and hence its length is uniformly bounded; (ii) by Lemma \ref{O(1) scale}, $\frac{\partial u_\varepsilon}{\partial x_2}\geq c\varepsilon^{-1}$ in $\{|u_\varepsilon|<1-b\}$.

{\bf Step 2.}  There exists a universal constant $C$ such that
\begin{equation}\label{6.4}
\int_{\{|u_\varepsilon|>1-b\}\cap B_{1/100}(x_\varepsilon)}\varepsilon\varphi_\varepsilon^2\leq C(b).
\end{equation}

In order to prove this estimate, take a cut-off function $\eta_2\in C_0^\infty(B_{1/50}(x_\varepsilon))$ with $\eta_2\equiv 1$ in $B_{1/100}(x_\varepsilon)$ and $|\nabla\eta_2|\leq 1000$, and $\zeta\in C^\infty(-1,1)$ with $\zeta\equiv 1$ in $(-1,-1+b)\cup(1-b,1)$, $\zeta\equiv 0$ in $(-1+2b,1-2b)$ and $|\zeta^\prime|\leq 2b^{-1}$. Multiplying \eqref{subsol to linearized} by $\varphi_\varepsilon\eta_2^2\zeta(u_\varepsilon)^2$ and integrating by parts leads to
\begin{eqnarray*}
&&\int_{B_{1/50}(x_\varepsilon)}\varepsilon|\nabla\left(\varphi_\varepsilon\eta_2\zeta(u_\varepsilon)\right)|^2+\frac{1}{\varepsilon}W^{\prime\prime}(u_\varepsilon)\varphi_\varepsilon^2\eta_2^2\zeta(u_\varepsilon)^2\\
&=&\int_{B_{1/50}(x_\varepsilon)}\varepsilon\varphi_\varepsilon^2|\nabla\left(\eta_2\zeta(u_\varepsilon)\right)|^2\\
&\lesssim&\int_{B_{1/50}(x_\varepsilon)}\varepsilon\varphi_\varepsilon^2\left[|\nabla\eta_2|^2\zeta(u_\varepsilon)^2+2\eta_2\zeta(u_\varepsilon)|\zeta^\prime(u_\varepsilon)||\nabla\eta_2||\nabla u_\varepsilon|+\eta_2^2\zeta^\prime(u_\varepsilon)^2|\nabla u_\varepsilon|^2\right]\\
&\lesssim&\varepsilon\int_{B_{1/50}(x_\varepsilon)}\varphi_\varepsilon^2+\frac{1}{\varepsilon}\int_{\{1-2b<|u_\varepsilon|<1-b\}\cap B_{1/50}(x_\varepsilon)}\varphi_\varepsilon^2,
\end{eqnarray*}
where we have substituted the estimate $|\nabla u_\varepsilon|\lesssim \varepsilon^{-1}$ in the last line.

Since $W^{\prime\prime}(u_\varepsilon)\geq c(b)>0$ in $\{|u_\varepsilon|>1-b\}$, we obtain
\begin{eqnarray*}
\int_{\{|u_\varepsilon|>1-b\}\cap B_{1/100}(x_\varepsilon)}\varepsilon\varphi_\varepsilon^2&\leq&C\varepsilon\int_{B_{1/50}(x_\varepsilon)}W^{\prime\prime}(u_\varepsilon)\varphi_\varepsilon^2\eta_2^2\zeta(u_\varepsilon)^2\\
&\leq&C\varepsilon^3\int_{B_{1/50}(x_\varepsilon)}|\nabla u_\varepsilon|^2+C\varepsilon\int_{\{1-2b<|u_\varepsilon|<1-b\}\cap B_{1/50}(x_\varepsilon)}\varphi_\varepsilon^2\\
&\leq&C.
\end{eqnarray*}

Combining Step 1 and Step 2 we finish the proof.
\end{proof}

\subsection{A surgery on $\varphi_\varepsilon$}
Next we use the smoothing modification in the proof of Proposition \ref{number of nodal domains} to decrease the left hand side of \eqref{2}.

Without loss of generality, assume $f_{\alpha,\varepsilon}(0)=0$, $f_{\alpha+1,\varepsilon}(0)=\rho_\varepsilon$ and $\rho_\varepsilon\leq \varepsilon|\log\varepsilon|$. By Lemma \ref{O(1) scale}, $\rho_\varepsilon\gg\varepsilon$. For any fixed constant $L>0$, $u_\varepsilon>1-b$ in $\Omega_{\alpha,\varepsilon}:=\{|x_1|<L\varepsilon, L\varepsilon<x_2<\rho_\varepsilon-L\varepsilon\}$. Let $\widetilde{\varphi}_\varepsilon$ be the solution of
\begin{equation*} \label{eq:1}
\left\{ \begin{aligned}
         \varepsilon\Delta\widetilde{\varphi}_\varepsilon &= \frac{1}{\varepsilon}W^{\prime\prime}(u_\varepsilon)\varphi_\varepsilon, \quad\mbox{in } \Omega_{\alpha,\varepsilon}, \\
                  \widetilde{\varphi}_\varepsilon&=\varphi_\varepsilon, \quad \mbox{on }\partial\Omega_{\alpha,\varepsilon}.
\end{aligned} \right.
\end{equation*}
By the stability of $u_\varepsilon$, such an $\widetilde{\varphi}_\varepsilon$ exists uniquely.

A direct integration by parts gives
\begin{eqnarray}\label{6.5}
&&\left[\int_{\Omega_{\alpha,\varepsilon}}\varepsilon|\nabla\varphi_\varepsilon|^2+\frac{1}{\varepsilon}W^{\prime\prime}(u_\varepsilon)\varphi_\varepsilon^2\right]
-\left[\int_{\Omega_{\alpha,\varepsilon}}\varepsilon|\nabla\widetilde{\varphi}_\varepsilon|^2+\frac{1}{\varepsilon}W^{\prime\prime}(u_\varepsilon)\widetilde{\varphi}_\varepsilon^2\right] \nonumber\\
&=&-\int_{\Omega_{\alpha,\varepsilon}}\varepsilon|\nabla\left(\varphi_\varepsilon-\widetilde{\varphi}_\varepsilon\right)|^2+\frac{1}{\varepsilon}W^{\prime\prime}(u_\varepsilon)
\left(\varphi_\varepsilon-\widetilde{\varphi}_\varepsilon\right)^2\\
&\leq&-\frac{c}{\varepsilon}\int_{\Omega_{\alpha,\varepsilon}}
\left(\varphi_\varepsilon-\widetilde{\varphi}_\varepsilon\right)^2. \nonumber
\end{eqnarray}

Because $u_\varepsilon>1-b$ in $\Omega_{\alpha,\varepsilon}$, $2-\delta(b)<W^{\prime\prime}(u_\varepsilon)<2+\delta(b)$ in $\Omega_{\alpha,\varepsilon}$, where $\delta(b)$ is constant satisfying $\lim_{b\to0}\delta(b)=0$. Therefore,
\[\Delta\widetilde{\varphi}_\varepsilon\leq\frac{2+\delta(b)}{\varepsilon^2}\widetilde{\varphi}_\varepsilon, \quad\mbox{in }\Omega_{\alpha,\varepsilon}.\]
On $\partial\Omega_{\alpha,\varepsilon}\cap\{x_2=L\varepsilon\}$, by Lemma \ref{O(1) scale},
\[\widetilde{\varphi}_\varepsilon=\varphi_\varepsilon=\frac{\partial u_\varepsilon}{\partial x_2}\geq\frac{c}{\varepsilon}.\]
By constructing an explicit subsolution, we obtain
\begin{equation}\label{6.6}
\widetilde{\varphi}_\varepsilon(x_1,x_2)\geq\frac{c}{\varepsilon}e^{-\sqrt{2+\delta(b)+\frac{C}{L^2}}\frac{x_2-L\varepsilon}{\varepsilon}}, \quad \mbox{in } \left\{|x_1|<\frac{L\varepsilon}{2}, \frac{\rho_\varepsilon}{4}<x_2<\frac{3\rho_\varepsilon}{4}\right\}.
\end{equation}

\begin{lem}
For any $\delta$ fixed, if $\varepsilon$ is small enough,
$\varphi_\varepsilon=0$ in $\{|x_1|<\frac{L\varepsilon}{2}, \frac{(1+\delta)\rho_\varepsilon}{2}<x_2<\frac{3\rho_\varepsilon}{4}\}$.
\end{lem}
\begin{proof}
Let $\epsilon:=\varepsilon/\rho_\varepsilon\ll 1$ and $u_\epsilon(x):=u_\varepsilon(\rho_\varepsilon^{-1}x)$,
which is a solution of \eqref{equation scaled}  with parameter $\epsilon$.

The nodal set of $u_\epsilon$ has the form $\cup_\beta\{x_2=\tilde{f}_{\beta,\epsilon}(x_1)\}$,
where $\tilde{f}_{\beta,\epsilon}(x_1)=f_{\beta,\varepsilon}(\rho_\varepsilon x_1)/\rho_\varepsilon$. Thus for any $x_1\in(-\rho_\varepsilon^{-1},\rho_\varepsilon^{-1})$,
\[|\tilde{f}_{\alpha,\epsilon}^{\prime\prime}(x_1)|\leq 4\rho_\varepsilon, \quad |\tilde{f}_{\alpha+1,\epsilon}^{\prime\prime}(x_1)|\leq 4\rho_\varepsilon.\]
and $\tilde{f}_{\alpha,\epsilon}(0)=\tilde{f}_{\alpha,\epsilon}^\prime(0)=0$,  $\tilde{f}_{\alpha+1,\epsilon}(0)=1$.

Since $\rho_\varepsilon\to0$, $\tilde{f}_{\alpha,\epsilon}\to0$ uniformly on any compact set of $\R$. Because different components of $\{u_\epsilon=0\}$ do not intersect, $\tilde{f}_{\alpha+1,\epsilon}\to 1$ uniformly on any compact set of $\R$.

Consider the distance type function $\Psi_\epsilon$, which is defined by the relation
\[u_\epsilon=g\left(\frac{\Psi_\epsilon}{\epsilon}\right).\]
By the vanishing viscosity method, in any compact set of $\{|x_2|<1\}$, $\Psi_\epsilon$ converges uniformly to
\makeatletter
\let\@@@alph\@alph
\def\@alph#1{\ifcase#1\or \or $'$\or $''$\fi}\makeatother
\begin{equation*}
{\Psi_\infty(x_1,x_2):=}
\begin{cases}
1-|x_2|, &x_2\geq 1/2, \\
x_2, &-1/2\leq x_2\leq 1/2,\\
-1-|x_2|, &x_2\leq-1/2.
\end{cases}
\end{equation*}
\makeatletter\let\@alph\@@@alph\makeatother
Moreover, because $\Psi_\infty$ is $C^1$ in $\{|x_1|<1, (1+\delta)/2<x_2<3/4\}$, $\Psi_\epsilon$ converges in $C^1(\{|x_1|<1, (1+\delta)/2<x_2<3/4\})$. In particular, for all $\epsilon$ small,
\[\frac{\partial u_\epsilon}{\partial x_2}=\frac{1}{\epsilon}g^\prime\left(\frac{\Psi_\epsilon}{\epsilon}\right)\frac{\partial\Psi_\epsilon}{\partial x_2}>0, \quad\mbox{in }\left\{|x_1|<1, \frac{1+\delta}{2}<x_2<\frac{3}{4}\right\}.\]
Rescaling back we finish the proof.
\end{proof}
\begin{rmk}\label{rmk 17.5}
The above proof also shows that
\[\mbox{dist}\left((x_1,f_{\alpha,\varepsilon}(x_1)), \Gamma_{\alpha+1,\varepsilon}\right)=\left(1+o(1)\right)\left(f_{\alpha+1,\varepsilon}(x_1)-f_{\alpha,\varepsilon}(x_1)\right).\]
\end{rmk}

By this lemma and \eqref{6.6}, we obtain
\begin{equation}\label{6.7}
\int_{\Omega_{\alpha,\varepsilon}}
\left(\varphi_\varepsilon-\widetilde{\varphi}_\varepsilon\right)^2 \geq \int_{\{|x_1|<\frac{L\varepsilon}{2}, \frac{(1+\delta)\rho_\varepsilon}{2}<x_2<\frac{3\rho_\varepsilon}{4}\}}\widetilde{\varphi}_\varepsilon^2 \geq\frac{c(L)}{\varepsilon}e^{-\frac{(1+\delta)}{2}\sqrt{2+\delta(b)+\frac{C}{L^2}}\frac{\rho_\varepsilon}{\varepsilon}}.
\end{equation}

As in the proof of Proposition \ref{number of nodal domains}, by combining \eqref{2}, \eqref{6.7} and the stability of $u_\varepsilon$, we obtain
\[\frac{c(L)}{\varepsilon}e^{-\frac{(1+\delta)}{2}\sqrt{2+\delta(b)+\frac{C}{L^2}}\frac{\rho_\varepsilon}{\varepsilon}}\leq C.\]
By choosing $\delta$, $\delta(b)$ sufficiently small, $L$ sufficiently large (depending only on $\sigma$), this implies that
\[\rho_\varepsilon\geq \frac{\sqrt{2}-\sigma}{2}\varepsilon|\log\varepsilon|-C(\sigma)\varepsilon,\]
which in view of Remark \ref{rmk 17.5} finishes the proof of Proposition \ref{first lower bound}.

\section{Toda system}
\setcounter{equation}{0}

\subsection{Optimal approximation}

As in Section \ref{sec Fermi coordinates},  we still work in the stretched version, i.e. after the rescaling $x\mapsto \varepsilon^{-1}x$. The analysis in Section \ref{sec Fermi coordinates} still holds, although now $\{u=0\}=\cup_\alpha\Gamma_\alpha$,
where the cardinality of the index set $\alpha$ could go to infinity.

\medskip

Given a sequence of functions $h_\alpha\in C^2(-R,R)$, let (note here a sign difference with Section \ref{sec approximate solution})
\[g_\alpha(y,z):=\bar{g}\left((-1)^\alpha\left(z-h_\alpha(y)\right)\right),\]
where $(y,z)$ is the Fermi coordinates with respect to $\Gamma_\alpha$.
Define the function $g(y,z;h_\alpha)$ in the following way:
\[g(y,z;h_\alpha):= g_\alpha+\sum_{\beta<\alpha}\left(g_\beta+(-1)^\beta\right)+\sum_{\beta>\alpha}\left(g_\beta-(-1)^\beta\right) \quad \mbox{in } \mathcal{M}_\alpha.\]
By the definition of $\bar{g}$ and Proposition \ref{first lower bound}, the above sum involves only finitely many terms (at most $25$ terms).

Similar to Proposition \ref{prop optimal approximation}, we have
\begin{prop}
There exists $(h_\alpha)$ such that for any $|\alpha|>100$, $h_\alpha\equiv 0$, while for any $|\alpha|\leq 100$, in the Fermi coordinates with respect to $\Gamma_{\alpha}$,
\begin{equation}\label{orthogonal condition III}
\int_{-\delta R}^{\delta R} \left[u(y,z)-g(y,z;h_\alpha)\right] \bar{g}^\prime\left((-1)^\alpha\left(z-h_\alpha(y)\right)\right)dz=0, \quad\forall y\in(-5R/6,5R/6).
\end{equation}
\end{prop}

Denote $g_\ast(y,z):=g(y,z;\textbf{h}(y))$, where $\textbf{h}$ is as in the previous lemma.
Let
\[\phi:=u-g_\ast.\]
As in Subsection \ref{sec optimal approximation}, denote
\[g_\alpha^\prime(y,z)=\bar{g}^\prime\left((-1)^\alpha\left(z-h_\alpha(y)\right)\right), \quad g_\alpha^{\prime\prime}(y,z)=\bar{g}^{\prime\prime}\left((-1)^\alpha\left(z-h_\alpha(y)\right)\right), \quad \cdots .\]

In the Fermi coordinates with respect to $\Gamma_\alpha$,
$\phi$ satisfies the following equation
\begin{eqnarray}\label{error equation III}
&&\Delta_z\phi-H^\alpha(y,z)\partial_z\phi+\partial_{zz}\phi \\
&=&W^\prime(g_\ast+\phi)-\sum_{\beta}
 W^\prime(g_\beta) +(-1)^{\alpha}g_\alpha^\prime\left[H^\alpha(y,z)+\Delta_z h_\alpha(y)\right]-g_\alpha^{\prime\prime}|\nabla_zh_\alpha|^2\nonumber\\
&+&\sum_{\beta\neq\alpha} \left[(-1)^{\beta}g_\beta^\prime\mathcal{R}_{\beta,1}\left(\Pi_\beta(y,z),d_\beta(y,z)\right)-
g_\beta^{\prime\prime}\mathcal{R}_{\beta,2}\left(\Pi_\beta(y,z),d_\beta(y,z)\right)\right] -\sum_{\beta}\xi_\beta,  \nonumber
\end{eqnarray}
where $\mathcal{R}_{\beta,1}$, $\mathcal{R}_{\beta,2}$ and $\xi_\beta$ are defined as in Subsection \ref{sec optimal approximation}.

\subsection{Estimates on $\phi$}

By \eqref{Schauder 3} and Proposition \ref{first lower bound}, for any $\sigma>0$ and $|\alpha|\leq 90$,
\begin{equation}\label{estimate 2}
\|\phi\|_{C^{2,1/2}(\mathcal{M}_\alpha(r))}+\|H^\alpha+\Delta_0^\alpha h_\alpha\|_{C^{1/2}(-r,r)}
\lesssim \varepsilon^{2-2\sigma}+\sup_{(-r-K|\log\varepsilon|,r+K|\log\varepsilon|)}e^{-\sqrt{2}D_\alpha}.
\end{equation}

Substituting Proposition \ref{first lower bound} into \eqref{estimate 2} gives a first (non-optimal)  bound
\begin{equation}\label{estimate first bound}
\|\phi\|_{C^{2,1/2}(\mathcal{M}_\alpha(6R/7))}+\|H^\alpha+\Delta_0^\alpha h_\alpha\|_{C^{1/2}(B_{6R/7})}
\lesssim \varepsilon^{1-\sigma}, \quad \forall |\alpha|\leq 90.
\end{equation}


By \eqref{10.5}, we can improve the estimates on $\phi_y:=\partial\phi/\partial y$ to
\begin{equation}\label{estimate on derivatives first bound}
\|\phi_y\|_{C^{1,1/2}(\mathcal{M}_\alpha(6R/7))}
\lesssim \varepsilon^{7/6}.
\end{equation}

\subsection{A Toda system}

Denote
\[A_\alpha(r):=\sup_{(-r,r)}e^{-\sqrt{2}D_\alpha}.\]
By Proposition \ref{first lower bound}, for any $r<5R/6$, $A_\alpha(r)\lesssim\varepsilon^{1-\sigma}$.

In $(-6R/7,6R/7)$, \eqref{Toda system} reads as
\begin{equation}\label{Toda system III}
\frac{f_{\alpha}^{\prime\prime}(x_1)}{\left[1+|f_{\alpha}^\prime(x_1)|^2\right]^{3/2}}
=\frac{4}{\sigma_0}\left[A_{(-1)^{\alpha-1}}^2e^{-\sqrt{2}d_{\alpha-1}(x_1,f_\alpha(x_1))}
-A_{(-1)^\alpha}^2e^{-\sqrt{2}d_{\alpha+1}(x_1,f_\alpha(x_1))}\right]+O\left(\varepsilon^{7/6}\right).
\end{equation}

\section{Reduction of the stability condition}\label{sec reduction of stability}
\setcounter{equation}{0}

In this section we show that the blow up procedure in Remark \ref{rmk 11.1} preserves the stability condition. More precisely, if $u$ is stable, $(f_\alpha)$ satisfies a kind of stability condition.

Take a large constant $L$ and a smooth function $\chi$ satisfying $\chi\equiv 1$ in $\mathcal{M}_\alpha$, $\chi\equiv 0$ outside $\{\rho_\alpha^-(y)-L<z<\rho_\alpha^+(y)+L\}$ and
in $\{\rho_\alpha^+(y)<z<\rho_\alpha^+(y)+L\}$,
\[\chi(y,z)=\eta_3\left(\frac{z-\rho_\alpha^+(y)}{L}\right)\]
where $\eta_3$ is a smooth function defined on $\R$ satisfying $\eta_3\equiv 1$ in $(-\infty,0)$, $\eta_3\equiv 1$ in $(1,+\infty)$ and $|\eta_3^\prime|\leq 2$.
Clearly we have $|\nabla\chi|\lesssim L^{-1}$, $|\nabla^2\chi|\lesssim L^{-2}$.

For any $\eta\in C_0^\infty(-5R/6,5R/6)$, let
\[\varphi(y,z):=\eta(y)g_\alpha^\prime(y,z)\chi(y,z).\]
The stability condition for $u$ implies that
\[\int_{\mathcal{C}_{5R/6}}|\nabla\varphi|^2+W^{\prime\prime}(u)\varphi^2\geq0.\]
The purpose of this section is to rewrite this inequality as a stability condition for the Toda system \eqref{Toda system III}.

In the Fermi coordinates with respect to $\Gamma_\alpha$, we have (Recall that now in \eqref{metirc tensor}, the metric tensor $g_{ij}$ has only one component, which is denoted by $\lambda(y,z)$ here)
\[|\nabla\varphi(y,z)|^2=\Big|\frac{\partial\varphi}{\partial z}(y,z)\Big|^2+\lambda(y,z)\Big|\frac{\partial\varphi}{\partial y}(y,z)\Big|^2.\]

\subsection{The horizontal part}\label{subsec 9.1}
A direct differentiation leads to
\[\frac{\partial\varphi}{\partial y}=\eta^\prime(y)g_\alpha^\prime\chi-\eta(y)g_\alpha^{\prime\prime}\chi h_\alpha^\prime(y)+\eta(y) g_\alpha^\prime\chi_y.\]
Since $c\leq\lambda(y,z)\leq C$,
\begin{eqnarray*}
\int_{\mathcal{C}_{5R/6}}\Big|\frac{\partial\varphi}{\partial y}(y,z)\Big|^2\lambda(y,z)dzdy
&\lesssim&\int_{-5R/6}^{5R/6}\int_{-\delta R}^{\delta R}\eta^\prime(y)^2|g_\alpha^\prime|^2\chi^2+\eta(y)^2|g_\alpha^{\prime\prime}|^2\chi^2h_\alpha^\prime(y)^2+\eta^2|g_\alpha^\prime|^2\chi_y^2\\
&\lesssim&\int_{-5R/6}^{5R/6}\eta^\prime(y)^2dy+\varepsilon^2\int_{-5R/6}^{5R/6}\eta(y)^2dy\\
&&+\frac{1}{L}\int_{-5R/6}^{5R/6}\eta(y)^2\left[e^{-2\sqrt{2}\rho_\alpha^+(y)}+e^{2\sqrt{2}\rho_\alpha^-(y)}\right]dy.
\end{eqnarray*}

Here the last term follows from the following two facts:
\begin{itemize}
\item in $\{\chi_y\neq0\}$, which is exactly $\{\rho_\alpha^+(y)<z<\rho_\alpha^+(y)+L\}\cup\{\rho_\alpha^-(y)-L<z<\rho_\alpha^-(y)\}$, $|\chi_y|\lesssim L^{-1}$;

\item in $\{\rho_\alpha^+(y)<z<\rho_\alpha^+(y)+L\}$ (respectively, $\{\rho_\alpha^-(y)-L<z<\rho_\alpha^-(y)\}$), $g_\alpha^\prime\lesssim e^{-\sqrt{2}\rho_\alpha^+(y)}$ (respectively, $g_\alpha^\prime\lesssim e^{\sqrt{2}\rho_\alpha^-(y)}$);

    \item By \eqref{estimate first bound} and Lemma \ref{control on h_0}, for $y\in(-6R/7,6R/7)$,
\[h_\alpha^\prime(y)^2\lesssim \varepsilon^{2-2\sigma}.\]
\end{itemize}

\subsection{The vertical part}
As before we have
\[\varphi_z=\eta g_\alpha^{\prime\prime}\chi+\eta g_\alpha^\prime \chi_z.\]
Thus by a direct expansion and integrating by parts, we have
\begin{eqnarray*}
\int_{\mathcal{C}_{5R/6}}\varphi_z^2\lambda(y,z)dzdy&=&\int_{-5R/6}^{5R/6}\eta(y)^2\left[\int_{-\delta R}^{\delta R}|g_\alpha^{\prime\prime}|^2\chi^2\lambda+2g_\alpha^\prime g_\alpha^{\prime\prime}\chi\chi_z\lambda+
|g_\alpha^\prime|^2\chi_z^2\lambda dz\right]dy\\
&=&-\int_{-5R/6}^{5R/6}\eta(y)^2\left[\int_{-\delta R}^{\delta R}W^{\prime\prime}(g_\alpha)|g_\alpha^\prime|^2\chi^2\lambda+g_\alpha^\prime\xi_\alpha^\prime\chi^2\lambda dz\right]dy\\
&&-\int_{-5R/6}^{5R/6}\eta(y)^2\left[\int_{-\delta R}^{\delta R}g_\alpha^\prime g_\alpha^{\prime\prime}\chi^2\lambda_z-|g_\alpha^\prime|^2\chi_z^2\lambda dz\right]dy.
\end{eqnarray*}

In the right hand side, except the first term, other terms can be estimated in the following way.
\begin{itemize}
\item Concerning the second term, because $\xi_\alpha^\prime=O(\varepsilon^3)$,
\[\int_{-5R/6}^{5R/6}\eta(y)^2\left[\int_{-\delta R}^{\delta R}g_\alpha^\prime(y,z)\xi_\alpha^\prime(y,z)\chi(y,z)^2\lambda(y,z)dz\right]dy=O(\varepsilon^2)\int_{-5R/6}^{5R/6}\eta(y)^2dy.\]

\item Concerning the third term, an integration by parts in $z$ leads to
\begin{eqnarray*}
&&-\int_{-5R/6}^{5R/6}\eta(y)^2\left[\int_{-\delta R}^{\delta R}g_\alpha^\prime g_\alpha^{\prime\prime}\chi^2\lambda_z dz\right]dy\\
&=&2\int_{-5R/6}^{5R/6}\eta(y)^2\left[\int_{-\delta R}^{\delta R}|g_\alpha^\prime|^2\chi \chi_z\lambda_z dz\right]dy
+\int_{-5R/6}^{5R/6}\eta(y)^2\left[\int_{-\delta R}^{\delta R}|g_\alpha^\prime|^2\chi^2\lambda_{zz}dz\right]dy\\
&\lesssim&\varepsilon\int_{-5R/6}^{5R/6}\eta(y)^2\left[e^{-2\sqrt{2}\rho_\alpha^+(y)}+e^{2\sqrt{2}\rho_\alpha^-(y)}\right]dy+\varepsilon^2\int_{-5R/6}^{5R/6}\eta(y)^2dy,
\end{eqnarray*}
where in the last line for the first term we have used the same facts as in Subsection \ref{subsec 9.1}  and the estimate
\[\lambda_z=-2\lambda(y,0)H^\alpha(y,0)\left(1-zH^\alpha(y,0)\right)=O(\varepsilon).\]
For the second term we have used the fact that
\[\lambda_{zz}=2H^\alpha(y,0)^2\lambda(y,0)=O(\varepsilon^2).\]

\item By the same reasoning as in Subsection \ref{subsec 9.1},
\[\int_{-5R/6}^{5R/6}\eta(y)^2\left[\int_{-\delta R}^{\delta R}|g_\alpha^\prime|^2\chi_z^2\lambda dz\right]dy\lesssim\frac{1}{L}\int\eta(y)^2\left[e^{-2\sqrt{2}\rho_\alpha^+(y)}+e^{2\sqrt{2}\rho_\alpha^-(y)}\right]dy.\]

\end{itemize}

In conclusion, we get
\begin{eqnarray*}
\int\varphi_z^2\lambda(y,z)dzdy&=&-\int_{-5R/6}^{5R/6}\eta(y)^2\left[\int_{-\delta R}^{\delta R}W^{\prime\prime}(g_\alpha)|g_\alpha^\prime|^2\chi^2\lambda dz\right]dy\\
&+&O(\varepsilon^2)\int\eta(y)^2dy+O\left(\frac{1}{L}+\varepsilon\right)\int\eta(y)^2\left[e^{-2\sqrt{2}\rho_\alpha^+(y)}+e^{2\sqrt{2}\rho_\alpha^-(y)}\right]dy.
\end{eqnarray*}

\medskip

Now the stability condition for $u$ is transformed into
\begin{eqnarray}\label{9.3.1}
0&\leq&C\int_{-5R/6}^{5R/6}\eta^\prime(y)^2dy+C\varepsilon^{2-2\sigma}\int_{-5R/6}^{5R/6}\eta(y)^2dy  +C\left(\frac{1}{L}+\varepsilon\right)\int_{-5R/6}^{5R/6}\eta(y)^2\left[e^{-2\sqrt{2}\rho_\alpha^+(y)}+e^{2\sqrt{2}\rho_\alpha^-(y)}\right]dy\nonumber\\
&+&\int_{-5R/6}^{5R/6}\eta(y)^2\left[\int_{-\delta R}^{\delta R}\left(W^{\prime\prime}(u)-W^{\prime\prime}(g_\alpha)\right)|g_\alpha^\prime|^2\chi^2\lambda dz\right]dy.
\end{eqnarray}
It remains to rewrite the last integral.

\subsection{The interaction part}
Differentiating \eqref{error equation III} in $z$ leads to
\begin{eqnarray}\label{error equation 3}
&&\frac{\partial}{\partial z}\Delta_z^\alpha\phi-\frac{\partial}{\partial z}\left(H^\alpha(y,z)\partial_z^\alpha\phi\right)+\partial^\alpha_{zzz}\phi \nonumber\\
&=&W^{\prime\prime}(u)\left[(-1)^\alpha g_\alpha^\prime+\phi_z+\sum_{\beta\neq\alpha}(-1)^\beta g_\beta^\prime\frac{\partial d_\beta}{\partial z}\right]
-(-1)^\alpha W^{\prime\prime}(g_\alpha)g_\alpha^\prime
-\sum_{\beta\neq\alpha}(-1)^{\beta}
 W^{\prime\prime}(g_\beta)g_\beta^\prime\frac{\partial d_\beta}{\partial z}  \nonumber\\
 &-&(-1)^\alpha\frac{\partial}{\partial z}\left[ g_\alpha^\prime\left(H^\alpha(y,z)+\Delta_z h_\alpha(y)\right)\right]- \frac{\partial}{\partial z}\left( g_\alpha^{\prime\prime}|\nabla_zh_\alpha|^2\right) \\
&-&\sum_{\beta\neq\alpha}\frac{\partial}{\partial z} \left[(-1)^{\beta}g_\beta^\prime\mathcal{R}_{\beta,1}\left(\Pi_\beta(y,z),d_\beta(y,z)\right)+
g_\beta^{\prime\prime}\mathcal{R}_{\beta,2}\left(\Pi_\beta(y,z),d_\beta(y,z)\right)\right]+\sum_{\beta}\frac{\partial\xi_\beta}{\partial z}\nonumber.
\end{eqnarray}

Multiplying this equation by $\eta^2g_\alpha^\prime\chi^2\lambda$ and then integrating in $y$ and $z$ gives
\begin{eqnarray*}
&&\int_{-5R/6}^{5R/6}\eta(y)^2\int_{-\delta R}^{\delta R}\left[\frac{\partial}{\partial z}\Delta_z^\alpha\phi-\frac{\partial}{\partial z}\left(H^\alpha(y,z)\partial_z^\alpha\phi\right)\right]g_\alpha^\prime\chi^2\lambda dzdy+\int_{-5R/6}^{5R/6}\eta(y)^2\left[\int_{-\delta R}^{\delta R}\partial^\alpha_{zzz}\phi g_\alpha^\prime\eta^2\chi^2\lambda dz\right]dy\\
&=&(-1)^\alpha\int_{-5R/6}^{5R/6}\eta(y)^2\int_{-\delta R}^{\delta R}\left[W^{\prime\prime}(u)-W^{\prime\prime}(g_\alpha)\right]|g_\alpha^\prime|^2\chi^2\lambda dzdy+\int_{-5R/6}^{5R/6}\eta(y)^2\left[\int_{-\delta R}^{\delta R} W^{\prime\prime}(u)\phi_zg_\alpha^\prime \chi^2\lambda dz\right]dy\\
&+&\sum_{\beta\neq\alpha}(-1)^\beta \int_{-5R/6}^{5R/6}\eta(y)^2\int_{-\delta R}^{\delta R} \left[W^{\prime\prime}(u)-W^{\prime\prime}(g_\beta)\right]g_\alpha^\prime g_\beta^\prime\frac{\partial d_\beta}{\partial z} \chi^2\lambda dzdy\\
 &-&(-1)^\alpha \int_{-5R/6}^{5R/6}\eta(y)^2\int_{-\delta R}^{\delta R} \frac{\partial}{\partial z}\left[ g_\alpha^\prime\left(H^\alpha(y,z)-\Delta_z h_\alpha(y)\right)\right]g_\alpha^\prime \chi^2\lambda dzdy\\
 &-&  \int_{-5R/6}^{5R/6}\eta(y)^2\left[\int_{-\delta R}^{\delta R} \frac{\partial}{\partial z}\left( g_\alpha^{\prime\prime}|\nabla_zh_\alpha|^2\right) g_\alpha^\prime  \chi^2\lambda dz\right]dy\\
 &-&\sum_{\beta\neq\alpha}(-1)^{\beta}\int_{-5R/6}^{5R/6}\eta(y)^2\int_{-\delta R}^{\delta R} g_\alpha^\prime \chi^2\lambda\frac{\partial}{\partial z} \left[g_\beta^\prime\mathcal{R}_{\beta,1}\left(\Pi_\beta(y,z),d_\beta(y,z)\right)\right]dzdy\\
&-&\sum_{\beta\neq\alpha} \int_{-5R/6}^{5R/6}\eta(y)^2\int_{-\delta R}^{\delta R} g_\alpha^\prime \chi^2\lambda\frac{\partial}{\partial z} \left[
g_\beta^{\prime\prime}\mathcal{R}_{\beta,2}\left(\Pi_\beta(y,z),d_\beta(y,z)\right)\right]dzdy+\sum_{\beta} \int_{-5R/6}^{5R/6}\eta(y)^2\int_{-\delta R}^{\delta R} g_\alpha^\prime \chi^2\lambda\frac{\partial\xi_\beta}{\partial z} dzdy.
\end{eqnarray*}

We need to estimate each term.
\begin{enumerate}
\item Integrating by parts in $z$ leads to
\[\int_{-\delta R}^{\delta R}\frac{\partial}{\partial z}\Delta_z\phi g_\alpha^\prime\chi^2\lambda(y,z)dz\\
=-\int_{-\delta R}^{\delta R}\Delta_z\phi \frac{\partial}{\partial z}\left( g_\alpha^\prime\chi^2\lambda(y,z)\right)dz.
\]
Using \eqref{estimate on derivatives first bound} and the exponential decay of $g_\alpha^\prime$ and $g_\alpha^{\prime\prime}$, we get
\[\Big|\int_{-\delta R}^{\delta R}\frac{\partial}{\partial z}\Delta_z\phi g_\alpha^\prime\chi^2\lambda(y,z)dz\Big|\lesssim \varepsilon^{\frac{7}{6}}.\]

\item Integrating by parts and using the exponential decay of $g_\alpha^\prime$ and $g_\alpha^{\prime\prime}$, we get
\begin{eqnarray*}
\Big|\int_{-\delta R}^{\delta R}\frac{\partial}{\partial z}\left(H^\alpha(y,z)\phi_z\right)g_\alpha^\prime\chi^2\lambda\Big|
=\Big|\int_{-\delta R}^{\delta R}H^\alpha(y,z)\phi_z\frac{\partial}{\partial z}\left(g_\alpha^\prime\chi^2\lambda\right)\Big|
&\lesssim&\varepsilon\sup_{\rho_\alpha^-(y)-L<z<\rho_\alpha^+(y)+L}|\phi_z(y,z)|\\
&\lesssim&\varepsilon^{2-\sigma}.
\end{eqnarray*}

\item
The second term in the left hand side and the second one in the right hand side can be canceled with a remainder term of higher order. More precisely,
\begin{eqnarray*}
&&\int_{-5R/6}^{5R/6}\eta(y)^2\left[\int_{-\delta R}^{\delta R}\phi_{zzz} g_\alpha^\prime\chi^2\lambda dz\right]dy\\
&=&\int_{-5R/6}^{5R/6}\eta(y)^2\left[\int_{-\delta R}^{\delta R}\phi_z \left(g_\alpha^{\prime\prime\prime}\chi^2\lambda +2g_\alpha^\prime \left(\chi_z^2+\chi\chi_{zz}\right)\lambda+g_\alpha^\prime\chi^2\lambda_{zz} \right) dz\right]dy\\
&+&\int_{-5R/6}^{5R/6}\eta(y)^2\left[\int_{-\delta R}^{\delta R}\phi_z\left(4g_\alpha^{\prime\prime}\chi\chi_z\lambda +2g_\alpha^{\prime\prime} \chi^2\lambda_z+4g_\alpha^\prime\chi\chi_z\lambda_z \right)dz\right]dy\\
&=&\int_{-5R/6}^{5R/6}\eta(y)^2\left[\int_{-\delta R}^{\delta R}\phi_z g_\alpha^{\prime\prime\prime}\chi^2\lambda dz+h.o.t.\right]dy\\
&=&\int_{-5R/6}^{5R/6}\eta(y)^2\left[\int_{-\delta R}^{\delta R}W^{\prime\prime}(g_\alpha)\phi_zg_\alpha^\prime\chi^2\lambda dz+h.o.t\right]dy.
\end{eqnarray*}
In the above those higher order terms can be bounded by $O(\varepsilon^{\frac{3}{2}-2\sigma})$. We only show how to prove
\begin{equation}\label{8.2}
\int_{-\delta R}^{\delta R}\phi_zg_\alpha^{\prime\prime}\chi\chi_z\lambda=O(\varepsilon^{\frac{3}{2}-2\sigma}).
\end{equation}
In $\mbox{spt}(\chi_z)$, $|\chi_z|\lesssim L^{-1}$,
\[|g_\alpha^{\prime\prime}|\lesssim e^{-\sqrt{2}|z|}\lesssim e^{-\frac{\sqrt{2}}{2}D_\alpha(y)}+\varepsilon\lesssim \varepsilon^{\frac{1-\sigma}{2}},\]
\[|\phi_z|\lesssim \|\phi\|_{C^{2,1/2}(\mathcal{M}_\alpha(r))}\lesssim\varepsilon^{1-\sigma}.\]
Combining these three estimates we get \eqref{8.2}.
Similarly, other terms are bounded by $O(\varepsilon^2)+O\left(|\nabla \phi(y,z)|_{L^\infty(\mathcal{M}_\alpha)}^2\right)=O(\varepsilon^{2-2\sigma})$.

Next we show that
\[\int_{-\delta R}^{\delta R}\left[W^{\prime\prime}(u)-W^{\prime\prime}(g_\alpha)\right]\phi_zg_\alpha^\prime\chi^2\lambda dz=O\left(\varepsilon^{2-3\sigma}\right).\]
This is because in $\{\chi\neq0\}$,
\[u=g_\alpha+\phi+\sum_{\beta<\alpha} \left(g_\beta-(-1)^\beta\right)+\sum_{\beta>\alpha}\left(g_\beta+(-1)^\beta \right),\]
hence this integral is bounded by
\begin{eqnarray*}
\int_{-\delta R}^{\delta R}\left(|\phi||\phi_z|g_\alpha^\prime\chi^2\lambda+\sum_{\beta\neq\alpha}|\phi_z|g_\beta^\prime g_\alpha^\prime \chi^2\lambda\right) dz
&\lesssim&\|\phi\|_{C^{1,1/2}(\mathcal{M}_\alpha(6R/7))}^2+\sup_{(-r,r)}D_\alpha^2 e^{-2\sqrt{2}D_\alpha}\\
&\lesssim&\varepsilon^{2-3\sigma}.
\end{eqnarray*}

\item In $\{\chi\neq0\}$,
\[W^{\prime\prime}(u)=W^{\prime\prime}(g_\alpha)+O(|\phi|)+\sum_{\beta\neq\alpha}O\left(g_\beta^\prime\right),\]
and for $\beta\neq\alpha$,
\[W^{\prime\prime}(g_\beta)=W^{\prime\prime}(1)+O\left(g_\beta^\prime\right).\]
Hence
\[\int_{-\delta R}^{\delta R} \left[W^{\prime\prime}(u)-W^{\prime\prime}(g_\beta)\right]g_\alpha^\prime g_\beta^\prime\frac{\partial d_\beta}{\partial z} \chi^2\lambda dz
=\int_{-\delta R}^{\delta R} \left[W^{\prime\prime}(g_\alpha)-W^{\prime\prime}(1)\right]g_\alpha^\prime g_\beta^\prime\lambda(y,0) dz+h.o.t.,
\]
where higher order terms are controlled by
\begin{eqnarray*}
&&\int_{-\delta R}^{\delta R}|\phi|g_\alpha^\prime g_\beta^\prime\chi^2+\sum_{\beta\neq\alpha}\int_{-\delta R}^{\delta R}g_\alpha^\prime \big|g_\beta^\prime\big|^2\chi^2+\sum_{\beta\neq\alpha}\int_{-\delta R}^{\delta R} \Big|g_\alpha^\prime\Big|^2 g_\beta^\prime\left(1-\chi^2\right)\\
&+&\Big\|\frac{\partial d_\beta}{\partial z}-1\Big\|_{L^\infty(\mathcal{M}_\alpha(r))}\int_{-\delta R}^{\delta R}\big|g_\alpha^\prime\big|^2g_\beta^\prime\chi^2+\varepsilon\int_{-\delta R}^{\delta R}|z|\big|g_\alpha^\prime\big|^2 g_\beta^\prime\chi^2dz\\
&\lesssim&\varepsilon^{1-\sigma} D_\alpha(y)e^{-\sqrt{2}D_\alpha(y)}+e^{-\frac{3\sqrt{2}}{2}D_\alpha(y)}+\varepsilon e^{-\sqrt{2}D_\alpha(y)}\\
&\lesssim&\varepsilon^{\frac{3}{2}-3\sigma}.
\end{eqnarray*}

\item
Integrating by parts gives
\begin{eqnarray*}
&&\int_{-\delta R}^{\delta R} \frac{\partial}{\partial z}\left[ g_\alpha^\prime\left(H^\alpha(y,z)+\Delta_z h_\alpha(y)\right)\right]g_\alpha^\prime \chi^2\lambda dz\\
&=&-\int_{-\delta R}^{\delta R}  g_\alpha^\prime\left[H^\alpha(y,z)+\Delta_z h_\alpha(y)\right]\left[g_\alpha^{\prime\prime} \chi^2\lambda+2g_\alpha^\prime\chi\chi_z\lambda+g_\alpha^\prime\chi^2\lambda_z\right] dz\\
&=& \frac{1}{2}\int_{-\delta R}^{\delta R} \big|g_\alpha^\prime\big|^2\frac{\partial}{\partial z}\left[\left(H^\alpha(y,z)+\Delta_z h_\alpha(y)\right)\chi^2\lambda\right]\\
&-&2\int_{-\delta R}^{\delta R}\big|g_\alpha^\prime\big|^2\left[H^\alpha(y,z)+\Delta_z h_\alpha(y)\right]\chi\chi_z\lambda-\int_{-\delta R}^{\delta R} \big|g_\alpha^\prime\big|^2\left[H^\alpha(y,z)+\Delta_z h_\alpha(y)\right]\chi^2\lambda_zdz\\
&=& \frac{1}{2}\int_{-\delta R}^{\delta R} \big|g_\alpha^\prime\big|^2\chi^2\lambda\frac{\partial}{\partial z}\left[H^\alpha(y,z)+\Delta_z h_\alpha(y)\right]\\
&-&\int_{-\delta R}^{\delta R}\big|g_\alpha^\prime\big|^2\left[H^\alpha(y,z)+\Delta_z h_\alpha(y)\right]\chi\chi_z\lambda-\frac{1}{2}\int_{-\delta R}^{\delta R} \big|g_\alpha^\prime\big|^2\left[H^\alpha(y,z)+\Delta_z h_\alpha(y)\right]\chi^2\lambda_zdz.
\end{eqnarray*}

Because
\[\frac{\partial}{\partial z}H^\alpha(y,z)=\frac{H^\alpha(y,0)^2}{1-zH^\alpha(y,0)}=O\left(\varepsilon^2\right),\]
\[
\big|\frac{\partial}{\partial z}\Delta_zh_\alpha(y)\big|\lesssim\varepsilon \left(\big|h_\alpha^{\prime\prime}(y)\big|+\big|h_\alpha^\prime(y)\big|\right)\lesssim\varepsilon^{2-\sigma} ,\]
the first integral is bounded by $O\left(\varepsilon^{2-2\sigma}\right)$.

In $\{\chi_z\neq0\}$,
\[\big|g_\alpha^\prime\big|^2\lesssim\varepsilon^2+e^{-\sqrt{2}D_\alpha}\lesssim \varepsilon^{1-\sigma}.\]
Because
\[H^\alpha(y,z)+\Delta_z h_\alpha(y)=O(\varepsilon^{1-\sigma}),\]
the second integral is bounded by $O(\varepsilon^{2-2\sigma})$.

Because $\lambda_z=O(\varepsilon)$, the third integral is bounded by $O(\varepsilon^{2-\sigma})$.

\item Integrating by parts in $z$ and using \eqref{estimate first bound} leads to
\begin{eqnarray*}
\int_{-\delta R}^{\delta R} \frac{\partial}{\partial z}\left( g_\alpha^{\prime\prime}|\nabla_zh_\alpha|^2\right) g_\alpha^\prime  \chi^2\lambda dz=-\int_{-\delta R}^{\delta R}  g_\alpha^{\prime\prime}|\nabla_zh_\alpha|^2 \frac{\partial}{\partial z}\left(g_\alpha^\prime  \chi^2\lambda \right)dz
&\lesssim&|h_\alpha^\prime(y)|^2\\
&\lesssim&\varepsilon^{2-2\sigma}.
\end{eqnarray*}

\item
For $\beta\neq\alpha$,
\[\int_{-\delta R}^{\delta R} g_\alpha^\prime \chi^2\lambda\frac{\partial}{\partial z} \left[g_\beta^\prime\mathcal{R}_{\beta,1}\left(\Pi_\beta(y,z),d_\beta(y,z)\right)\right]dz=-\int_{-\delta R}^{\delta R} \frac{\partial}{\partial z} \left[g_\alpha^\prime \chi^2\lambda\right]g_\beta^\prime\mathcal{R}_{\beta,1}\left(\Pi_\beta(y,z),d_\beta(y,z)\right)dz.
\]
Because
\begin{eqnarray*}
\Big|\mathcal{R}_{\beta,1}\left(\Pi_\beta(y,z),d_\beta(y,z)\right)\Big|\lesssim\varepsilon^{1-\sigma},
\end{eqnarray*}
the above integral can be controlled by
\[
\varepsilon^{1-\sigma}\int_{-\delta R}^{\delta R}\chi e^{-\sqrt{2}\left(|d_\alpha(y,z)|+|d_\beta(y,z)|\right)}dz
\lesssim \varepsilon^{1-\sigma}D_\alpha(y)e^{-\sqrt{2}D_\alpha(y)}
\lesssim \varepsilon^{2-3\sigma}.
\]

\item Similarly, because $|\mathcal{R}_{\beta,2}|\lesssim\varepsilon^{2-2\sigma}$,
we have
\begin{eqnarray*}
\int_{-\delta R}^{\delta R} g_\alpha^\prime \chi^2\lambda\frac{\partial}{\partial z} \left[g_\beta^{\prime\prime}\mathcal{R}_{\beta,2}\left(\Pi_\beta(y,z),d_\beta(y,z)\right)\right]dz
&=&-\int_{-\delta R}^{\delta R} \frac{\partial}{\partial z} \left[g_\alpha^\prime \chi^2\lambda\right]g_\beta^{\prime\prime}\mathcal{R}_{\beta,2}\left(\Pi_\beta(y,z),d_\beta(y,z)\right)dz\\
&=&O(\varepsilon^{2-2\sigma}).
\end{eqnarray*}

\item Finally, by the defintion of $\xi_\beta$,
\[ \int_{-5R/6}^{5R/6}\eta(y)^2\int_{-\delta R}^{\delta R} g_\alpha^\prime \chi^2\lambda\frac{\partial\xi_\beta}{\partial z} dzdy=O(\varepsilon^2) \int_{-5R/6}^{5R/6}\eta(y)^2dy.\]

\end{enumerate}

\medskip

Combining all of these estimates together, we obtain
\begin{eqnarray*}
&&\int_{-5R/6}^{5R/6}\eta(y)^2\left[\int_{-\delta R}^{\delta R}\left[W^{\prime\prime}(u)-W^{\prime\prime}(g_\alpha)\right]|g_\alpha^\prime|^2\chi^2\lambda dz\right]dy\\
&=&\sum_{\beta\neq\alpha}\int_{-5R/6}^{5R/6}\eta(y)^2\left[\int_{-\delta R}^{\delta R} \left[W^{\prime\prime}(g_\alpha)-W^{\prime\prime}(1)\right]g_\alpha^\prime g_\beta^\prime dz\right]\lambda(y,0) dy
+O(\varepsilon^{\frac{4}{3}})\int_{-5R/6}^{5R/6}\eta(y)^2dy.
\end{eqnarray*}
The last integral can be computed by applying Lemma \ref{lem form of interaction}, which leads to
\begin{eqnarray}\label{interaction part}
&&\int_{-5R/6}^{5R/6}\eta(y)^2\left[\int_{-\delta R}^{\delta R}\left[W^{\prime\prime}(u)-W^{\prime\prime}(g_\alpha)\right]|g_\alpha^\prime|^2\chi^2\lambda dz\right]dy  \\
&=&-4\int_{-5R/6}^{5R/6}\eta(y)^2\left[A_{(-1)^{\alpha-1}}^2e^{-\sqrt{2}d_{\alpha-1}(y,0)}+A_{(-1)^{\alpha}}^2e^{\sqrt{2}d_{\alpha+1}(y,0)}\right]\lambda(y,0) dy+O(\varepsilon^{\frac{4}{3}})\int_{-5R/6}^{5R/6}\eta(y)^2dy. \nonumber
\end{eqnarray}

\subsection{A stability condition for the Toda system}
Substituting \eqref{interaction part} into \eqref{9.3.1} we get
\begin{eqnarray}\label{9.3.2}
0&\leq&C\int_{-5R/6}^{5R/6}\eta^\prime(y)^2dy+C\varepsilon^{\frac{4}{3}}\int_{-5R/6}^{5R/6}\eta(y)^2dy   +C\left(\frac{1}{L}+\varepsilon\right)\int_{-5R/6}^{5R/6}\eta(y)^2\left[e^{-2\sqrt{2}\rho_\alpha^+(y)}+e^{2\sqrt{2}\rho_\alpha^-(y)}\right]dy\nonumber\\
&-&c\int_{-5R/6}^{5R/6}\eta(y)^2\left[A_{(-1)^{\alpha-1}}^2e^{-\sqrt{2}d_{\alpha-1}(y,0)}+A_{(-1)^{\alpha}}^2e^{\sqrt{2}d_{\alpha+1}(y,0)}\right]\lambda(y,0) dy.
\end{eqnarray}
First we have the following estimates.
Because $d_\alpha(y,\rho_\alpha^+(y))=d_{\alpha+1}(y,\rho_\alpha^+(y))$, if they are smaller than $\sqrt{2}|\log\varepsilon|$, by Lemma \ref{comparison of distances},
\[d_\alpha(y,\rho_\alpha^+(y))=d_{\alpha+1}(y,\rho_\alpha^+(y))=-\frac{1}{2}d_{\alpha+1}(y,0)+o(1).\]
Hence
\[e^{-2\sqrt{2}\rho_\alpha^+(y)}\lesssim \varepsilon^2+e^{\sqrt{2}d_{\alpha+1}(y,0)}.\]
A similar estimate holds for $e^{2\sqrt{2}\rho_\alpha^-(y)}$. From these we deduce that
\[\int_{-5R/6}^{5R/6}\eta(y)^2\left[e^{-2\sqrt{2}\rho_\alpha^+(y)}+e^{2\sqrt{2}\rho_\alpha^-(y)}\right]dy\leq C\varepsilon^2\int_{-5R/6}^{5R/6}\eta(y)^2dy+C\int_{-5R/6}^{5R/6}\eta(y)^2e^{-\sqrt{2}D_\alpha(y)}dy.
\]

Substituting these estimates into \eqref{9.3.2} leads to
\begin{eqnarray}\label{9.3.3}
&&\left(c-\frac{C}{L} \right)\int_{-5R/6}^{5R/6}\eta(y)^2\left[A_{(-1)^{\alpha-1}}^2e^{-\sqrt{2}d_{\alpha-1}(y,0)}+A_{(-1)^{\alpha}}^2e^{\sqrt{2}d_{\alpha+1}(y,0)}\right]\lambda(y,0) dy\\
&\leq&C\int_{-5R/6}^{5R/6}\eta^\prime(y)^2dy+C\varepsilon^{\frac{4}{3}}\int_{-5R/6}^{5R/6}\eta(y)^2dy   \nonumber.
\end{eqnarray}

By choosing $L$ large enough, we get
\begin{equation}\label{reduction of stability}
\int_{-5R/6}^{5R/6}\eta(y)^2\left[e^{-\sqrt{2}d_{\alpha-1}(y,0)}+e^{\sqrt{2}d_{\alpha+1}(y,0)}\right] dy\leq C\int_{-5R/6}^{5R/6}\eta^\prime(y)^2dy+C\varepsilon^{\frac{4}{3}}\int_{-5R/6}^{5R/6}\eta(y)^2dy.
\end{equation}

\begin{rmk}
With a little more work and passing to the blow up limit as in Remark \ref{rmk 11.1}, we get exactly the stability condition for the Toda system \eqref{Toda entire}.
\end{rmk}

\section{Proof of Theorem \ref{second order estimate}}\label{sec completion of proof}
\setcounter{equation}{0}

In this section we prove
\begin{prop}\label{improved distance}
For any $\alpha$ and $y\in(-R/2,R/2)$, if $|f_\alpha(y)|<2\delta R$, then
\[D_{\alpha}(y)\geq\frac{4\sqrt{2}}{7}|\log\varepsilon|.\]
\end{prop}

\medskip

First let us use this proposition to prove Theorem \ref{second order estimate}.
\begin{proof}[Proof of Theorem \ref{second order estimate}]
Substituting Proposition \ref{improved distance}
 into \eqref{estimate 2}, we get
\begin{equation}\label{estimate 3}
\|\phi\|_{C^{2,1/2}(\mathcal{M}_\alpha(R/2))}+\|H^\alpha+\Delta_0^\alpha h_\alpha\|_{C^{1/2}(B_{R/2})}
\lesssim\varepsilon^{8/7}.
\end{equation}
By Lemma \ref{control on h_0},
\[\|H^\alpha\|_{L^\infty(-R/2,R/2)}\lesssim\|\phi\|_{C^{2,1/2}(\mathcal{M}_\alpha(R/2))}+\|H^\alpha+\Delta_0^\alpha h_\alpha\|_{C^{1/2}(B_{R/2})}\lesssim\varepsilon^{8/7}.\]
Then for any $|y|<R/2$ and $|z|<\delta R$,
\[|H^\alpha(y,z)|\lesssim |H^\alpha(y,0)|\lesssim\varepsilon^{8/7}.\]

In $\mathcal{M}_\alpha(R/2)$,
\[
\nabla u=(-1)^\alpha g_\alpha^\prime\left(\frac{\partial}{\partial z}-h_\alpha^\prime(y)\frac{\partial}{\partial y}\right)+\nabla\phi
+\sum_{\beta\neq\alpha}(-1)^\beta g_\beta^\prime\left(\nabla d_\beta-h_\beta^\prime(\Pi_\beta(y,z))\nabla\Pi_\beta(y,z)\right),
\]
\begin{eqnarray*}
\nabla^2 u&=&-(-1)^\alpha g_\alpha^\prime  h_\alpha^{\prime\prime}(y)\frac{\partial}{\partial y}\otimes \frac{\partial}{\partial y}-(-1)^\alpha g_\alpha^\prime  h_\alpha^\prime(y)\nabla\frac{\partial}{\partial y}+(-1)^\alpha g_\alpha^\prime\nabla\frac{\partial}{\partial z}\\
&+&  g_\alpha^{\prime\prime}\left(\frac{\partial}{\partial z}-h_\alpha^\prime(y)\frac{\partial}{\partial y}\right)\otimes\left(\frac{\partial}{\partial z}-h_\alpha^\prime(y)\frac{\partial}{\partial y}\right)\\
&+&\sum_{\beta\neq\alpha}(-1)^\beta g_\beta^\prime(y,z)(\Pi_\beta(y,z))\mathcal{R}_{\beta,3}+\sum_{\beta\neq\alpha}  g_\beta^{\prime\prime}(y,z)\mathcal{R}_{\beta,4}+\nabla^2\phi,
\end{eqnarray*}
where in the Fermi coordinates with respect to $\Gamma_\beta$,
\[\mathcal{R}_{\beta,3}=-h_\beta^{\prime\prime}\frac{\partial}{\partial y}\otimes \frac{\partial}{\partial y}-h_\beta^\prime\nabla\frac{\partial}{\partial y}+\nabla\frac{\partial}{\partial z},\]
\[\mathcal{R}_{\beta,4}=\left(\frac{\partial}{\partial z}-h_\alpha^\prime(y)\frac{\partial}{\partial y}\right)\otimes\left(\frac{\partial}{\partial z}-h_\alpha^\prime(y)\frac{\partial}{\partial y}\right).\]

By the estimates on $\phi$ and $h$ we obtain
\[\nabla u=(-1)^\alpha g_\alpha^\prime\frac{\partial}{\partial z}+O\left(\varepsilon^{8/7}\right),\]
\[\nabla^2 u=(-1)^\alpha g_\alpha^{\prime\prime}\frac{\partial}{\partial z}\otimes \frac{\partial}{\partial z}+O\left(\varepsilon^{8/7}\right).\]

Using these forms we obtain, for any $L>0$, in $\mathcal{M}_0(R/2)\cap\{|z|<L\}$
\[\frac{|\nabla^2u|^2-|\nabla|\nabla u||^2}{|\nabla u|^2}\leq C(L)\varepsilon^{16/7}.\]
For any $b\in(0,1)$, there exists an $L(b)$ such that $\mathcal{M}_0(R/2)\cap\{|u|<1-b\}\subset\mathcal{M}_0(R/2)\cap\{|z|<L(b)\}$. Therefore this estimates holds for this domain, too.

After a rescaling we obtain $|B(u_\varepsilon)|\leq C\varepsilon^{1/7}$ in $\{|u_\varepsilon|<1-b\}\cap\{|x_1|<1/2, |x_2|<1/2\}$, and Theorem \ref{second order estimate} is proven.
\end{proof}

Recalling the definition of $A_\alpha(r)$ and $D_\alpha(y)$ in Section \ref{sec Fermi coordinates} and Section \ref{sec lower bound O(ep)}.
Note that $A_\alpha(r)$ is non-decreasing in $r$ while $D_\alpha(r)$ is non-increasing in $r$.

To prove Proposition \ref{improved distance}, we assume $\alpha=0$ and by the contrary
\begin{equation}\label{absurd assumtion}
A_0(R/2)\geq\varepsilon^{8/7}.
\end{equation}
This implies that for any $r\in[R/2,4R/5]$, $A_0(r)\geq\varepsilon^{8/7}$.

We will establish the following decay estimate.
\begin{lem}\label{decay estimate}
There exists a constant $K$ such that for any $r\in [R/2,4R/5]$, we have
\[A_0\left(r-KR^{\frac{4}{7}}\right)\leq\frac{1}{2}A_0(r).\]
\end{lem}

An iteration of this decay estimate from $r=4R/5$ to $R/2$ leads to
\[A_0(R/2)\leq 2^{-CK^{-1}R^{-4/7}}A_0(4R/5)\leq Ce^{-c\varepsilon^{-4/7}}\ll\varepsilon^2.\]
This is a contradiction with the assumption \eqref{absurd assumtion}. Thus we finish the proof of Proposition \ref{improved distance}, provided that Lemma \ref{decay estimate} holds true.

\medskip

Now let us prove Lemma \ref{decay estimate}. Fix an $r\in[R/2,4R/5]$ and denote $\epsilon:=A_0(r)$.
We will prove
\begin{equation}\label{decay estimate 2}
A_0\left(r-K\epsilon^{-1/2}\right)\leq\frac{\epsilon}{2}.
\end{equation}
By \eqref{absurd assumtion}, $\epsilon\geq \varepsilon^{8/7}$. Thus
\[K\epsilon^{-1/2}\leq K\varepsilon^{-\frac{4}{7}}=KR^{\frac{4}{7}},\]
and
\[A_0\left(r-KR^{\frac{4}{7}}\right)\leq A_0\left(r-K\epsilon^{-1/2}\right)\leq\frac{\epsilon}{2},\]
which is Lemma \ref{decay estimate}.

\medskip

To prove \eqref{decay estimate 2}, we need to prove that for any $x_\ast\in [-r+K\epsilon^{1/2},r-K\epsilon^{1/2}]$,
\begin{equation}\label{decay estimate 3}
e^{-\sqrt{2}D_0(x_\ast)}\leq\frac{\epsilon}{2}.
\end{equation}

After a rotation and  a translation, we may assume $x_\ast=0$ and
\begin{equation}\label{10.1}
f_0(0)=f_0^\prime(0)=0.
\end{equation}

By the Toda system \eqref{Toda system}, for any $y\in[-K\epsilon^{-1/2},K\epsilon^{-1/2}]$,
\begin{equation}\label{assumption on second derivative 1}
|f_0^{\prime\prime}(y)|\lesssim e^{-\sqrt{2}D_0(y)}+\varepsilon^{7/6}\lesssim\epsilon.
\end{equation}
We also have a semi-bound on $f_{\pm}$.
\begin{lem}
\begin{equation}\label{assumption on second derivative 2}
f_{-1}^{\prime\prime}(y)\gtrsim - e^{-\sqrt{2}d_{-1}(y)}-\varepsilon^{7/6}\gtrsim -\epsilon,
\end{equation}
\begin{equation}\label{assumption on second derivative 3}
f_1^{\prime\prime}(y)\lesssim  e^{\sqrt{2}d_1(y)}+\varepsilon^{7/6}\lesssim \epsilon.
\end{equation}
\end{lem}
\begin{proof}
By \eqref{Toda system III},
\begin{eqnarray*}
\frac{f_1^{\prime\prime}}{\left(1+|f_1^\prime|^2\right)^{3/2}}&=&\frac{4}{\sigma_0}\left[A_1^2e^{-\sqrt{2}|d_0^1|}-A_{-1}^2e^{-\sqrt{2}|d_2^1|}\right]+O(\varepsilon^{7/6})\\
&\leq&\frac{4A_1^2}{\sigma_0}e^{-\sqrt{2}|d_0^1|}+O(\varepsilon^{7/6}).
\end{eqnarray*}

By Lemma \ref{comparison of distances}, either $|d_0^1(y)|\leq\sqrt{2}|\log\varepsilon|$ or
\[d_0^1(y)=d_1(y)+O(\varepsilon^{1/3}).\]
The bound \eqref{assumption on second derivative 3} then follows from \eqref{assumption on second derivative 1}. In the same way we get \eqref{assumption on second derivative 2}.
\end{proof}

By \eqref{10.1} and \eqref{assumption on second derivative 1},  for any $y\in[-K\epsilon^{-1/2},K\epsilon^{-1/2}]$,
\begin{equation}\label{assumption on first derivative 1}
|f_0^\prime(y)|\leq C\epsilon|y|\leq CK\epsilon^{1/2}.
\end{equation}

Substituting these into the Toda system \eqref{Toda system III} we obtain in $(-K\epsilon^{-1/2},K\epsilon^{-1/2})$
\begin{eqnarray}\label{Toda system 2 III}
f_0^{\prime\prime}&=&\frac{f_0^{\prime\prime}}{\left(1+|f_0^\prime|^2\right)^{3/2}}+O(\epsilon^2) \nonumber\\
&=&\frac{4}{\sigma_0}\left(A_{-1}^2e^{-\sqrt{2}d_{-1}}-A_1^2e^{\sqrt{2}d_1}\right)+O(\epsilon^2)+O(\varepsilon^{7/6})\\
&=&\frac{4}{\sigma_0}\left(A_{-1}^2e^{-\sqrt{2}d_{-1}}-A_1^2e^{\sqrt{2}d_1}\right)+O(\epsilon^{49/48}). \nonumber
\end{eqnarray}

\begin{lem}\label{comparison of distance 2}
For $y\in[-2\epsilon^{-1/2},2\epsilon^{-1/2}]$, if $|d_{-1}(y)|\leq\sqrt{2}|\log\varepsilon|$, then we have
\[e^{-\sqrt{2}|d_{-1}(y)|}=e^{-\sqrt{2}\left(f_0(y)-f_{-1}(y)\right)}+O(\epsilon^{49/48});\]
if $|d_1(y)|\leq\sqrt{2}|\log\varepsilon|$, then we have
\[e^{-\sqrt{2}|d_1(y)|}=e^{-\sqrt{2}\left(f_1(y)-f_0(y)\right)}+O(\epsilon^{49/48}).\]
\end{lem}
\begin{proof}
We only prove the second identity. The first one can be proved in the same way.

 As in Lemma \ref{comparison of distances}, if $|d_1(y)|\leq\sqrt{2}|\log\varepsilon|$,
we have
\[\sup_{(y-1,y+1)}\big|f_1^\prime-f_0^\prime\big|\lesssim\varepsilon^{1/2}|\log\varepsilon|^2\lesssim\epsilon^{1/4}.\]
Because $\big|f_0^\prime\big|\lesssim\epsilon^{1/2}$ in $[-K\epsilon^{-1/2},K\epsilon^{-1/2}]$, this implies that
\[\sup_{(y-1,y+1)}\big|f_1^\prime\big|\lesssim\epsilon^{1/4}.\]
As in Lemma \ref{comparison of distances}, from this we deduce that
\[\bar{d}_1(y)=f_0(y)-f_1(y)+O(\epsilon^{1/16}).\]
Then
\[e^{\sqrt{2}\bar{d}_1(y)}=e^{-\sqrt{2}\left(f_1(y)-f_0(y)\right)}+O(\epsilon^{17/16}).\]
This  finishes the proof.
\end{proof}

By this lemma and the fact that $e^{-\sqrt{2}D_0(y)}\leq\epsilon$, in $(-2\epsilon^{-1/2},2\epsilon^{-1/2})$, \eqref{Toda system 2 III} can be rewritten as
\begin{eqnarray}\label{Toda system 3 III}
f_0^{\prime\prime}(y)=\frac{4}{\sigma_0}\left(A_{-1}^2e^{-\sqrt{2}\left[f_0(y)-f_{-1}(y)\right]}-A_1^2e^{\sqrt{2}\left[f_1(y)-f_0(y)\right]}\right)+O(\epsilon^{1/48}).
\end{eqnarray}

 Now define the functions in $[-K,K]$,
\[\tilde{f}_{\alpha}(y):=f_{\alpha}\left(\epsilon^{-1/2} y\right)-\alpha\frac{\sqrt{2}}{2}|\log\epsilon|, \quad \alpha=-1,0,1.\]
They satisfy
\begin{itemize}
\item $\tilde{f}_0(0)=\tilde{f}_0^\prime(0)=0$.

\item In $(-K,K)$, $|\tilde{f}_0^{\prime\prime}|\leq C$, $\tilde{f}_1^{\prime\prime}\leq C$ and $\tilde{f}_{-1}^{\prime\prime}\geq -C$.


\item In $(-2,2)$,
\begin{equation}\label{Toda system 3}
\tilde{f}_0^{\prime\prime}=\frac{4}{\sigma_0}\left[A_{-1}^2e^{-\sqrt{2}\left(\tilde{f}_0-\tilde{f}_{-1}\right)}-A_1^2e^{-\sqrt{2}\left(\tilde{f}_1-\tilde{f}_0\right)}\right]+O(\epsilon^{1/48}).
\end{equation}
\end{itemize}

\begin{lem}
\begin{equation}\label{intergral small}
\int_{-2}^2\left[e^{-\sqrt{2}\left(\tilde{f}_0-\tilde{f}_{-1}\right)}+e^{-\sqrt{2}\left(\tilde{f}_1-\tilde{f}_0\right)}\right]\leq\frac{C}{K}+CK\epsilon^{1/16}.
\end{equation}
\end{lem}
\begin{proof}
Take a function $\bar{\eta}\in C_0^\infty(-K,K)$ satisfying $\bar{\eta}\equiv 1$ in $(-2,2)$ and $|\bar{\eta}^\prime|\lesssim K^{-1}$.
Taking the test function $\eta$ in \eqref{reduction of stability} to be $\bar{\eta}(\epsilon^{-1/2}y)$, we obtain
\begin{eqnarray*}
\int_{-2}^2\left(e^{-\sqrt{2}d_{-1}(\epsilon^{-1/2}y)}+e^{\sqrt{2}d_1(\epsilon^{-1/2}y)}\right)dy&\leq&\int_{-K}^K\bar{\eta}(y)^2\left(e^{-\sqrt{2}d_{-1}(\epsilon^{-1/2}y)}+e^{\sqrt{2}d_1(\epsilon^{-1/2}y)}\right)dy\\
&\leq&C\epsilon\int_{-K}^K\bar{\eta}^\prime(y)^2+C\epsilon^{16/15}\int_{-K}^K\bar{\eta}(y)^2\\
&\leq&\frac{C}{K}\epsilon+CK\epsilon^{16/15}.
\end{eqnarray*}
After using Lemma \ref{comparison of distance 2} and a rescaling, the left hand side can be transformed into the required form.
\end{proof}

The following lemma establishes \eqref{decay estimate 2}, thus completes the proof of Lemma \ref{decay estimate}.
\begin{lem}
If $K$ is large enough (but independent of $\epsilon$), then\[\max_{[-1,1]}\left(e^{-\sqrt{2}\left(\tilde{f}_0-\tilde{f}_{-1}\right)}+e^{-\sqrt{2}\left(\tilde{f}_1-\tilde{f}_0\right)}\right)
\leq \frac{1}{2}.\]
\end{lem}
\begin{proof}
By \eqref{Toda system III}, in $(-2\epsilon^{-1/2},2\epsilon^{-1/2})$,
\[f_1^{\prime\prime}\leq \frac{4A_{-1}^2}{\sigma_0}e^{-\sqrt{2}\left(f_1-f_0\right)}\left(1+|f_1^\prime|^2\right)^{3/2}+O\left(\varepsilon^{7/6}\right).\]
By the proof of Lemma \ref{comparison of distance 2},
\begin{itemize}
\item either
$f_1-f_0\geq\sqrt{2}|\log\varepsilon|$, which implies that
 \[e^{-\sqrt{2}\left(f_1-f_0\right)}\lesssim\varepsilon^2;\]
\item or $|f_1^\prime-f_0^\prime|\leq\epsilon^{1/4}$, which together with \eqref{assumption on first derivative 1} implies that
\[\big|f_1^\prime\big|\leq 2\epsilon^{1/4}.\]
\end{itemize}
Therefore, because $e^{-\sqrt{2}\left(\tilde{f}_1-\tilde{f}_0\right)}\leq\epsilon$, we obtain
\[f_1^{\prime\prime}\leq \frac{4A_{-1}^2}{\sigma_0}e^{-\sqrt{2}\left(f_1-f_0\right)}+O(\epsilon^{7/6}).\]
After a rescaling this gives
\[\tilde{f}_1^{\prime\prime}\leq\frac{4A_{-1}^2}{\sigma_0}e^{-\sqrt{2}\left(\tilde{f}_1-\tilde{f}_0\right)}+O(\epsilon^{1/16}), \quad\mbox{in } (-2,2).\]
By \eqref{Toda system 3},
\begin{equation}\label{exponential equation}
\left(\tilde{f}_1-\tilde{f}_0\right)^{\prime\prime}\leq \frac{8A_{-1}^2}{\sigma_0}e^{-\sqrt{2}\left(\tilde{f}_1-\tilde{f}_0\right)}+O(\epsilon^{1/16}), \quad\mbox{in } (-2,2).
\end{equation}

Then
\begin{eqnarray*}
\frac{d^2}{dy^2}e^{-\sqrt{2}\left(\tilde{f}_1-\tilde{f}_0\right)}&\geq&-\sqrt{2}e^{-\sqrt{2}\left(\tilde{f}_1-\tilde{f}_0\right)}\left(\tilde{f}_1-\tilde{f}_0\right)^{\prime\prime}\\
&\geq&-Ce^{-2\sqrt{2}\left(\tilde{f}_1-\tilde{f}_0\right)}-C\epsilon^{1/16}e^{-\sqrt{2}\left(\tilde{f}_1-\tilde{f}_0\right)}.
\end{eqnarray*}
By the estimate of Choi-Schoen \cite{Choi-Schoen}, there exists a universal constant $\eta_\ast$ such that if
\begin{equation}\label{ep regularity}
\int_{-2}^2e^{-\sqrt{2}\left(\tilde{f}_1-\tilde{f}_0\right)}\leq \eta_\ast,
\end{equation}
then
\[\sup_{[-1,1]}e^{-\sqrt{2}\left(\tilde{f}_1-\tilde{f}_0\right)}\leq\frac{1}{4}.\]
In \eqref{intergral small}, we can first choose $K$ small and then let $\epsilon$ be small enough so that \eqref{ep regularity} holds. Then the claim follows by proving the same bound on $\sup_{[-1,1]}e^{-\sqrt{2}\left(\tilde{f}_0-\tilde{f}_{-1}\right)}$.
\end{proof}

\bigskip

\appendix{}

\section{Some facts about the one dimensional solution}\label{sec 1d solution}
\setcounter{equation}{0}

It is known that the following identity holds for $g$,
\begin{equation}\label{first integral}
g^\prime(t)=\sqrt{2W(g(t))}>0, \quad \forall t\in\R.
\end{equation}
Moreover, as $t\to\pm\infty$, $g(t)$ converges exponentially to $\pm1$ and  the following quantity is well defined
\[\sigma_0:=\int_{-\infty}^{+\infty}\left[\frac{1}{2}g^\prime(t)^2+W(g(t))\right]dt\in (0,+\infty).\]

In fact,  as $t\to\pm\infty$, the following expansions hold. There exists a positive constants $A_{1}$ such that
for all $t>0$ large,
\[g(t)=1-A_1 e^{-\sqrt{2} t}+O(e^{-2\sqrt{2}  t}),\]
\[g^\prime(t)=\sqrt{2} A_1 e^{-\sqrt{2}  t}+O(e^{-2\sqrt{2}  t}),\]
\[g^{\prime\prime}(t)=-2 A_1 e^{-\sqrt{2}  t}+O(e^{-2\sqrt{2}  t}),\]
and a similar expansion holds as $t\to-\infty$ with $A_1$ replaced by another positive constant $A_{-1}$.

\medskip

The following result describes the interaction between two one dimensional profiles.
\begin{lem}\label{lem form of interaction}
For all $T>0$ large, we have the following expansion:
\begin{equation}\label{exapnsion 3}
\int_{-\infty}^{+\infty}\left[W^{\prime\prime}(g(t))-W^{\prime\prime}(1)\right]\left[g(-t-T)+1\right]g^\prime(t)dt=-4A_{-1}^2e^{-\sqrt{2}T}+O\left(e^{-\frac{4\sqrt{2}}{3}T}\right).
\end{equation}
\begin{equation}\label{exapnsion 4}
\int_{-\infty}^{+\infty}\left[W^{\prime\prime}(g(t))-W^{\prime\prime}(1)\right]\left[g(T-t)-1\right]g^\prime(t)dt=4A_1^2e^{-\sqrt{2}T}+O\left(e^{-\frac{4\sqrt{2}}{3}T}\right).
\end{equation}
\end{lem}
\begin{proof}
We only prove the first expansion.

{\bf Step 1.} Note that
\[\big|W^{\prime\prime}(g(t))-2\big|\lesssim g^\prime(t).\]
Therefore the integral in $(-\infty,-3T/4)$ is controlled by
\[
\int_{-\infty}^{-3T/4}g^\prime(t)^2g^\prime(-t-T)dt \lesssim  \int_{-\infty}^{-3T/4}e^{2\sqrt{2}t}dt\lesssim e^{-\frac{3\sqrt{2}}{2}T}.
\]

{\bf Step 2.}  Similarly the integral in $(3T/4,+\infty)$ is controlled by
\[
\int_{3T/4}^{+\infty}g^\prime(t)^2g^\prime(-t-T)dt \lesssim \int_{3T/4}^{+\infty}e^{-3\sqrt{2}t-\sqrt{2}T}dt \lesssim  e^{-\frac{3\sqrt{2}}{2}T}.
\]

{\bf Step 3.} In $(-3T/4,3T/4)$,
\[g(-t-T)+1=A_{-1} e^{-\sqrt{2}t-\sqrt{2}T}+O\left(e^{-2\sqrt{2}t-2\sqrt{2}T}\right).\]
Because
\begin{eqnarray*}
\Big|\int_{-3T/4}^{3T/4}\left[W^{\prime\prime}(g(t))-2\right]g^\prime(t)e^{-2\sqrt{2}t-2\sqrt{2}T}dt\Big| \lesssim e^{-2\sqrt{2}T}\int_{-3T/4}^{3T/4}g^\prime(t)^2e^{-2\sqrt{2}t}dt
&\lesssim& Te^{-2\sqrt{2}T}\\
&\lesssim&e^{-\frac{3\sqrt{2}}{2}T},
\end{eqnarray*}
we have
\[\int_{-\infty}^{+\infty}\left[W^{\prime\prime}(g(t))-2\right]g^\prime(t)g^\prime(t+T)dt=A_{-1} e^{-\sqrt{2}T}\int_{-3T/4}^{3T/4}\left[W^{\prime\prime}(g(t))-2\right]g^\prime(t)e^{-\sqrt{2}t}dt+O\left(e^{-\frac{3\sqrt{2}}{2}T}\right).
\]
As in Step 1 and Step 2, we have
\[\Big|\int_{-\infty}^{-3T/4}\left[W^{\prime\prime}(g(t))-2\right]g^\prime(t)e^{-\sqrt{2}t}dt\Big|\lesssim e^{-\frac{3\sqrt{2}}{4}T},\]
\[\Big|\int_{3T/4}^{+\infty}\left[W^{\prime\prime}(g(t))-2\right]g^\prime(t)e^{-\sqrt{2}t}dt\Big|\lesssim e^{-\frac{3\sqrt{2}}{4}T}.\]
Therefore
\[\int_{-\infty}^{+\infty}\left[W^{\prime\prime}(g(t))-2\right]\left[g(-t-T)+1\right]g^\prime(t)dt=A_{-1} e^{-\sqrt{2}T}\int_{-\infty}^{+\infty}\left[W^{\prime\prime}(g(t))-2\right]g^\prime(t)e^{-\sqrt{2}t}dt+O\left(e^{-\frac{3\sqrt{2}}{2}T}\right).
\]

{\bf Step 4.} It remains to determine the integral
\[\int_{-\infty}^{+\infty}
\left[W^{\prime\prime}(g(t))-2\right]g^\prime(t)e^{-\sqrt{2}t}dt.\]
Note that the integrand converges to $0$ exponentially as $t\to\pm\infty$. As in Step 1 and Step 2, we have
\[\int_{-\infty}^{+\infty}
\left[W^{\prime\prime}(g(t))-2\right]g^\prime(t)e^{-\sqrt{2}t}dt
=\lim_{L\to+\infty}\int_{-L}^{L}
\left[W^{\prime\prime}(g(t))-2\right]g^\prime(t)e^{-\sqrt{2}t}dt.\]
Because $W^{\prime\prime}(g(t))g^\prime(t)=g^{\prime\prime\prime}(t)$, integrating by parts gives
\begin{eqnarray*}
\int_{-L}^{L}
\left[W^{\prime\prime}(g(t))-2\right]g^\prime(t)e^{-\sqrt{2}t}dt
&=&g^{\prime\prime}(L)e^{-\sqrt{2}L}+\sqrt{2}g^\prime(L)e^{-\sqrt{2}L}-\left[g^{\prime\prime}(-L)e^{\sqrt{2}L}+\sqrt{2}g^\prime(-L)e^{\sqrt{2}L}\right]\\
&=&-4A_{-1}+O(e^{-2\sqrt{2}L}).
\end{eqnarray*}
After letting $L\to+\infty$ we finish the proof.
\end{proof}

Next we discuss the spectrum  of the linearized operator at $g$,
\[\mathcal{L}=-\frac{d^2}{dt^2}+W^{\prime\prime}(g(t)).\]
 By a direct differentiation we see $g^\prime(t)$ is an eigenfunction of $\mathcal{L}$ corresponding to eigenvalue $0$.
By \eqref{first integral}, $0$ is the lowest eigenvalue. In other words, $g$ is stable.

Concerning the second eigenvalue, we have
\begin{thm}\label{second eigenvalue for 1d}
There exists a constant $\mu>0$ such that for any $\varphi\in H^1(\R)$ satisfying
\begin{equation}\label{orthogonal condition 1d}
\int_{-\infty}^{+\infty}\varphi(t)g^\prime(t)dt=0,
\end{equation}
we have
\[\int_{-\infty}^{+\infty}\left[\varphi^\prime(t)^2+W^{\prime\prime}(g(t))\varphi(t)^2\right]dt\geq\mu\int_{-\infty}^{+\infty}\varphi(t)^2dt.\]
\end{thm}
This can be proved via a contradiction argument.

\section{Derivation of \eqref{Toda system}}\label{sec derivation of Toda}
\setcounter{equation}{0}

We estimate those
 terms  in \eqref{H eqn} one by one.
\begin{enumerate}
\item Differentiating \eqref{orthogonal condition 2} twice leads to
\begin{equation}\label{orthogonal condition 2.1}
\int_{-\delta R}^{\delta R}\frac{\partial\phi}{\partial y^i}g_\alpha^\prime+(-1)^\alpha\phi
g_\alpha^{\prime\prime}\frac{\partial h_\alpha}{\partial y^i}=0,
\end{equation}
\begin{equation}\label{orthogonal condition 2.2}
\int_{-\delta R}^{\delta R}\frac{\partial^2\phi}{\partial y^i\partial
y^j}g_\alpha^\prime+(-1)^\alpha\frac{\partial\phi}{\partial y^i}
g^{\prime\prime}_\alpha\frac{\partial h_\alpha}{\partial
y^j}+(-1)^\alpha\frac{\partial\phi}{\partial y^j}
g^{\prime\prime}_\alpha\frac{\partial h_\alpha}{\partial y^i}+(-1)^\alpha\phi
g^{\prime\prime}_\alpha\frac{\partial^2 h_\alpha}{\partial y^i\partial
y^j}+\phi g^{\prime\prime\prime}_\alpha\frac{\partial h_\alpha}{\partial
y^i}\frac{\partial h_\alpha}{\partial y^j}=0.
\end{equation}
Therefore
\begin{equation}\label{orthogonal condition 2.3}
\int_{-\delta R}^{\delta R}\Delta_0\phi(y,z)g_\alpha^\prime=(-1)^{\alpha-1}\Delta_0h_\alpha\int_{-\delta R}^{\delta R}\phi
g_\alpha^{\prime\prime}+2(-1)^\alpha\int_{-\delta R}^{\delta R}
g^{ij}(y,0)\frac{\partial\phi}{\partial y^i}\frac{\partial
h_\alpha}{\partial y^j}g_\alpha^{\prime\prime}-|\nabla_0h_\alpha|^2\int_{-\delta R}^{\delta R}\phi
g_\alpha^{\prime\prime\prime}.
\end{equation}
Then by \eqref{error in z 5},
\begin{eqnarray}
\int_{-\delta R}^{\delta R}\Delta_z\phi(y,z)g_\alpha^\prime&=&\int_{-\delta R}^{\delta R}\Delta_0\phi(y,z)g_\alpha^\prime+O\left(\varepsilon\right)\int_{-\delta R}^{\delta R}\left(|\nabla_y^2\phi(y,z)| +|\nabla_y\phi(y,z)|\right)|z|e^{-\sqrt{2}|z|}dz \nonumber\\
&=&(-1)^{\alpha-1}\Delta_0h_\alpha\int_{-6|\log\varepsilon|}^{6|\log\varepsilon|}\phi g_\alpha^{\prime\prime}+O(|\nabla h_\alpha(y)|^2)\int_{-6|\log\varepsilon|}^{6|\log\varepsilon|}|\phi(y,z)|e^{-\sqrt{2}|z|}dz \nonumber\\
&+&O(|\nabla
h_\alpha(y)|+\varepsilon)\int_{-6|\log\varepsilon|}^{6|\log\varepsilon|}\left(|\nabla_y^2\phi(y,z)|+|\nabla_y\phi(y,z)|\right)\left(1+|z|\right)e^{-\sqrt{2}|z|}dz\\
&=&(-1)^{\alpha-1}\Delta_0h_\alpha\int_{-6|\log\varepsilon|}^{6|\log\varepsilon|}\phi g_\alpha^{\prime\prime}  \nonumber\\
&+&O(|\nabla
h_\alpha(y)|+\varepsilon)\sup_{(-6|\log\varepsilon|,6|\log\varepsilon|)}\left(|\nabla_y^2\phi(y,z)|+|\nabla_y\phi(y,z)|+|\phi(y,z)|\right)e^{-\left(\sqrt{2}-\sigma\right)|z|}. \nonumber
\end{eqnarray}

\item By integrating by parts in $z$ and using \eqref{error in z 4}, we obtain
\begin{eqnarray*}
\int_{-\delta R}^{\delta R}H^\alpha(y,z)\phi_z
g_\alpha^\prime&=&-\int_{-\delta R}^{\delta R}\frac{\partial H^\alpha}{\partial z}(y,z)\phi
g_\alpha^\prime+(-1)^{\alpha-1}H^\alpha(y,z)\phi g_\alpha^{\prime\prime}\\
&=&-\int_{-\delta R}^{\delta R}\frac{\partial H^\alpha}{\partial z}(y,z)\phi
g_\alpha^\prime+(-1)^{\alpha-1}\left(H^\alpha(y,z)-H^\alpha(y,0)\right)\phi g_\alpha^{\prime\prime}\\
&&-(-1)^{\alpha-1}H^\alpha(y,0)\int_{-\delta R}^{\delta R}\phi g_\alpha^{\prime\prime}\\
&=&-(-1)^{\alpha-1}H^\alpha(y,0)\int_{-\delta R}^{\delta R}\phi
g_\alpha^{\prime\prime}+O(\varepsilon^2)\sup_{|z|<6|\log\varepsilon|}\left(|\phi(y,z)|e^{-\left(\sqrt{2}-\sigma\right)|z|}\right)\\
&=&-(-1)^{\alpha-1}H^\alpha(y,0)\int_{-\delta R}^{\delta R}\phi
g_\alpha^{\prime\prime}+O(\varepsilon^2).
\end{eqnarray*}

\item By integrating by parts in $z$ and using \eqref{1d eqn}, we obtain
\[
\int_{-\delta R}^{\delta R}g_\alpha^\prime\partial_{zz}\phi=\int_{-\delta R}^{\delta R}W^{\prime\prime}(g_\alpha)g_\alpha^\prime\phi
+\bar{\xi}_\alpha\phi=\int_{-\delta R}^{\delta R}W^{\prime\prime}(g_\alpha)g_\alpha^\prime\phi
+O(\varepsilon^2).
\]

\item
By \eqref{error in z 4}, we have
\begin{eqnarray*}
\int_{-\delta R}^{\delta R}H^\alpha(y,z)|g_\alpha^\prime|^2&=&H^\alpha(y,0)\int_{-\delta R}^{\delta R}|g_\alpha^\prime|^2+O\left(\varepsilon^2\int_{-\delta R}^{\delta R}|z||g_\alpha^\prime|^2\right)\\
&=&\left(\sigma_0+O(\varepsilon)^2\right)
H^\alpha(y,0)+O\left(\varepsilon^2\int_{-\delta R}^{\delta R}|z||g_\alpha^\prime|^2\right)\\
&=&\sigma_0
H^\alpha(y,0)+O(\varepsilon^2).
\end{eqnarray*}

\item
By \eqref{error in z 5}, we obtain
 \begin{eqnarray*}
\int_{-\delta R}^{\delta R}\Delta_zh_\alpha(y)|g_\alpha^\prime|^2&=&\Delta_0h_\alpha(y)\int_{-\delta R}^{\delta R}|g_\alpha^\prime|^2+O\left(\varepsilon(|\nabla^2h_\alpha(y)|+|\nabla h_\alpha(y)|)\right)\int_{-\delta R}^{\delta R}|z||g_\alpha^\prime|^2\\
&=&\left(\sigma_0+O(\varepsilon)^2\right)
\Delta_0h_\alpha(y)+O\left(\varepsilon(|\nabla^2h_\alpha(y)|+|\nabla h_\alpha(y)|)\right)\\
&=&\sigma_0\Delta_0h_\alpha(y)+O(\varepsilon^2)+O\left(|\nabla^2h_\alpha(y)|^2+|\nabla h_\alpha(y)|^2\right).
\end{eqnarray*}

\item By integrating by parts in $z$ and using \eqref{metirc tensor}, we obtain
\[
\int_{-\delta R}^{\delta R} g_\alpha^{\prime\prime}g_\alpha^\prime|\nabla_z
h_\alpha|^2
=-\frac{1}{2}\int_{-\delta R}^{\delta R}\frac{dg^{ij}(y,z)}{dz}\frac{\partial
h_\alpha}{\partial y^i}(y)
\frac{\partial h_\alpha}{\partial y^j}(y)|g_\alpha^\prime|^2=O\left(\varepsilon |\nabla h_\alpha(y)|^2\right).\]

\item
For $\beta\neq\alpha$, if $g_\alpha^\prime\neq0$ and $g_\beta^\prime\neq0$ at the same time, then $|z|\leq 6|\log\varepsilon|$, $|d_\beta(y,z)|\leq 6|\log\varepsilon|$. Therefore by Lemma \ref{comparison of distances}, \[d_\beta(y,z)=z+d_\beta(y,0)+O(\varepsilon^{1/3}).\]

Then
\begin{eqnarray*}
g_\beta^\prime(y,z)&=&\bar{g}^\prime\left((-1)^{\beta-1}\left(d_\beta(y,z)-h_\beta\left(\Pi_\beta(y,z)\right)\right)\right)\\
&=&\bar{g}^\prime\left((-1)^{\beta-1}\left(z+d_\beta(y,0)-h_\beta\left(\Pi_\beta(y,z)\right)+O(\varepsilon^{1/3})\right)\right).
\end{eqnarray*}

By \eqref{error in z 4} and \eqref{error in z 5}, we have
\[g_\alpha^\prime\mathcal{R}_{\alpha,1}=g_\alpha^\prime\left[H^\alpha(y,0)+\Delta_0^\alpha h_\alpha(y)\right]+O\left(\varepsilon^2+\varepsilon\left(|\nabla^2h_\alpha(y)|+|\nabla h_\alpha(y)|\right)\right)|z|g_\alpha^\prime.\]
Therefore
\begin{eqnarray*}
&&\int_{-\delta R}^{\delta R}g_\alpha^\prime g_\beta^\prime \mathcal{R}_{\beta,1}\\
&\lesssim&\int_{-\delta R}^{\delta R}\big|H^\beta\left(\Pi_\beta(y,z)\right)+\Delta_0^\beta h_\beta\left(\Pi_\beta(y,z)\right)\big|\bar{g}^\prime\left((-1)^{\alpha-1}(z-h_\alpha(y))\right)\\
&&\quad \times\bar{g}^\prime\left((-1)^{\beta-1}\left(z+d_\beta(y,0)-h_\beta\left(\Pi_\beta(y,z)\right)+O(\varepsilon^{1/3})\right)\right)dz\\
&+&O(\varepsilon^2)\int_{-\delta R}^{\delta R} |z+d_\beta(y,0)+O(\varepsilon^{1/3})|\bar{g}^\prime\left((-1)^{\alpha-1}(z-h_\alpha(y))\right)\\
&&\quad \times\bar{g}^\prime\left((-1)^{\beta-1}\left(z+d_\beta(y,0)-h_\beta\left(\Pi_\beta(y,z)\right)+O(\varepsilon^{1/3})\right)\right)dz\\
&+&O(\varepsilon)\int_{-\delta R}^{\delta R} \left(|\nabla^2h_\beta\left(\Pi_\beta(y,z)\right)|+|\nabla h_\beta\left(\Pi_\beta(y,z)\right)|\right)\big|z+d_\beta(y,0)+O(\varepsilon^{1/3})\big|\\
&&\times\bar{g}^\prime\left((-1)^{\alpha-1}(z-h_\alpha(y))\right)\bar{g}^\prime\left((-1)^{\beta-1}\left(z+d_\beta(y,0)-h_\beta\left(\Pi_\beta(y,z)\right)+O(\varepsilon^{1/3})\right)\right)dz\\
&\lesssim& |d_\beta(y,0)|e^{-\sqrt{2}|d_\beta(y,0)|}\left[\sup_{B_{\varepsilon^{1/3}}(y)}|H^\beta+\Delta^\beta h_\beta|+\varepsilon^2
\left(|d_\beta(y,0)|+\varepsilon^{1/3}\right)\right]\\
&+&\varepsilon|d_\beta(y,0)|e^{-\sqrt{2}|d_\beta(y,0)|}
\left(|d_\beta(y,0)|+\varepsilon^{1/3}\right)\sup_{B_{\varepsilon^{1/3}}(y)}\left(|\nabla^2h_\beta|+|\nabla h_\beta|\right)\\
&\lesssim& |d_\beta(y,0)|e^{-\sqrt{2}|d_\beta(y,0)|}\left[\sup_{B_{\varepsilon^{1/3}}(y)}|H^\beta+\Delta^\beta h_\beta|+\varepsilon |\log\varepsilon|\sup_{B_{\varepsilon^{1/3}}(y)}\left(|\nabla^2h_\beta|+|\nabla h_\beta|\right)\right]+\varepsilon^2.
\end{eqnarray*}

\item
By the same reasoning,
for $\beta\neq\alpha$,
\[\int_{-\delta R}^{\delta R}g_\alpha^\prime g_\beta^{\prime\prime} \mathcal{R}_{\beta,2}
\lesssim|d_\beta(y,0)|e^{-\sqrt{2}|d_\beta(y,0)|}\sup_{B_{\varepsilon^{1/3}}(y)}
|\nabla h_\beta\big|^2.\]

\item By Taylor expansion,
\begin{eqnarray*}
\int_{-\delta R}^{\delta R}\left[W^\prime(g_\ast+\phi)- W^\prime(g_\ast)\right]g_\alpha^\prime
&=&\int_{-\delta R}^{\delta R}W^{\prime\prime}(g_\ast)\phi g_\alpha^\prime+O\left(\int_{-\delta R}^{\delta R}\phi^2 g_\alpha^\prime\right)\\
&=&\int_{-\delta R}^{\delta R}W^{\prime\prime}(g_\alpha)g_\alpha^\prime\phi
+O\left(\int_{-\delta R}^{\delta R}\phi^2 g_\alpha^\prime\right)\\
&+&O\left(\int_{-\delta R}^{\delta R}|\phi|\left(|g_{\alpha-1}^2-1|+|g_{\alpha+1}^2-1|\right) g_\alpha^\prime\right)\\
\end{eqnarray*}
The first integral cancels with the term in (3). Next
\[\int_{-\delta R}^{\delta R}\phi^2 g_\alpha^\prime\lesssim\sup_{z\in(-6|\log\varepsilon|,6|\log\varepsilon|)}|\phi(y,z)|^2.\]
Finally as in (7),
\begin{eqnarray*}
\int_{-\delta R}^{\delta R}|\phi|\left(|g_{\alpha-1}^2-1|+|g_{\alpha+1}^2-1|\right) g_\alpha^\prime
&\lesssim&D_\alpha e^{-\sqrt{2}D_\alpha}\sup_{z\in(-6|\log\varepsilon|,6|\log\varepsilon|)}|\phi(y,z)|\\
&\lesssim&e^{-\frac{3\sqrt{2}}{2}D_\alpha}+\sup_{z\in(-6|\log\varepsilon|,6|\log\varepsilon|)}|\phi(y,z)|^2.
\end{eqnarray*}

\item
To determine the integral
\[\int_{-\delta R}^{\delta R}\left[W^\prime(g_\ast)-\sum_{\beta} W^\prime(g_\beta)\right]g_\alpha^\prime,\]
 consider for each $\beta$, the integral on $(-\delta R,\delta R)\cap\mathcal{M}_\beta$, which is an interval $(\rho_\beta^-(y),\rho_\beta^+(y))$.
If $\beta\neq\alpha$, by Lemma \ref{upper bound on interaction}, in $(\rho_\beta^-(y),\rho_\beta^+(y))$,
\[
\Big|W^\prime(g_\ast)-\sum_{\beta} W^\prime(g_\beta)\Big|
\lesssim e^{-\sqrt{2}(|d_\beta|+|d_{\beta-1}|)}+e^{-\sqrt{2}(|d_\beta|+|d_{\beta+1}|)}+\varepsilon^2.\]
We only consider the case $\beta>\alpha$ and estimate
\[\int_{\rho_\beta^-(y)}^{\rho_\beta^+(y)}e^{-\sqrt{2}(|d_\beta|+|d_{\beta-1}|)}g_\alpha^\prime.\]

If $|z|$, $|d_\beta|$ and $|d_{\beta-1}|$ are all smaller than $6|\log\varepsilon|$, by Lemma \ref{comparison of distances},
\begin{equation}\label{B1}
d_\beta(y,z)=z+d_\beta(y,0)+O(\varepsilon^{1/3}),
\end{equation}
\begin{equation}\label{B2}
d_{\beta-1}(y,z)=z+d_{\beta-1}(y,0)+O(\varepsilon^{1/3}).
\end{equation}
Note that since $\beta>\alpha$, $z>0$ while $d_\beta(y,0)<d_{\beta-1}(y,0)\leq 0$.

We have
\begin{eqnarray*}
\int_{\rho_\beta^-(y)}^{\rho_\beta^+(y)}e^{-\sqrt{2}(|d_\beta|+|d_{\beta-1}|)}g_\alpha^\prime&\lesssim&\int_{\rho_\beta^-(y)}^{\rho_\beta^+(y)}e^{-\sqrt{2}\left(|z|+|z+d_{\beta-1}(y,0)|+|z+d_\beta(y,0)|\right)}\\
&\lesssim&\int_{\rho_\beta^-(y)}^{-d_\beta(y,0)}e^{-\sqrt{2}\left(z+d_{\beta-1}(y,0)-d_\beta(y,0)\right)}+\int_{-d_\beta(y,0)}^{\rho_\beta^+(y)}e^{-\sqrt{2}\left(3z+d_{\beta-1}(y,0)+d_\beta(y,0)\right)}\\
&\lesssim&e^{-\sqrt{2}\left(d_{\beta-1}(y,0)-d_\beta(y,0)\right)-\sqrt{2}\rho_\beta^-(y)}+e^{-\sqrt{2}\left(d_{\beta-1}(y,0)-2d_\beta(y,0)\right)}.
\end{eqnarray*}
By definition,
\[-d_\beta(y,\rho_\beta^-(y))=d_{\beta-1}(y,\rho_\beta^-(y)).\]
Thus by \eqref{B1} and \eqref{B2},
\[\rho_\beta^-(y)=-\frac{d_{\beta-1}(y,0)+d_\beta(y,0)}{2}+O(\varepsilon^{1/3}).\]
Substituting this into the above estimate gives
\[
\int_{\rho_\beta^-(y)}^{\rho_\beta^+(y)}e^{-\sqrt{2}(|d_\beta|+|d_{\beta-1}|)}g_\alpha^\prime \lesssim e^{-\frac{\sqrt{2}}{2}\left(d_{\beta-1}(y,0)-3d_\beta(y,0)\right)}+e^{-\sqrt{2}\left(d_{\beta-1}(y,0)-2d_\beta(y,0)\right)}.
\]

If $\beta=\alpha+1$, because $d_{\beta-1}(y,0)=0$, this is bounded by $O(e^{\frac{3\sqrt{2}}{2}d_{\alpha+1}(y,0)})$.

If $\beta\geq\alpha+2$, this is bounded by $O(e^{\sqrt{2}d_{\alpha+2}(y,0)})$.

It remains to consider the integration in $(\rho_\alpha^-(y),\rho_\alpha^+(y))$. In this case we use Lemma \ref{interaction term}, which gives
\begin{eqnarray}\label{7.1}
&&\int_{\rho_\alpha^-(y)}^{\rho_\alpha^+(y)}\left[W^\prime(g_\ast)-\sum_{\beta}  W^\prime(g_\beta)\right]g_\alpha^\prime \\
&=&\int_{\rho_\alpha^-(y)}^{\rho_\alpha^+(y)}\left[W^{\prime\prime}(g_\alpha)-2\right]\left[g_{\alpha-1}-(-1)^{\alpha-1}\right]g_\alpha^\prime+
\left[W^{\prime\prime}(g_\alpha)-2\right]\left[g_{\alpha+1}+(-1)^\alpha\right]g_\alpha^\prime\nonumber\\
&+&\int_{\rho_\alpha^-(y)}^{\rho_\alpha^+(y)}\left[O\left(e^{-2\sqrt{2}d_{\alpha-1}}+e^{2\sqrt{2}d_{\alpha+1}}\right)+O\left(e^{-\sqrt{2}d_{\alpha-2}-\sqrt{2}|z|}+e^{\sqrt{2}d_{\alpha+2}-\sqrt{2}|z|}\right)\right]g_\alpha^\prime.\nonumber
\end{eqnarray}

Because $g_\alpha^\prime\lesssim e^{-\sqrt{2}|z|}$ and
\[e^{-2\sqrt{2}d_{\alpha-1}}\lesssim e^{-2\sqrt{2}d_{\alpha-1}(y,0)-2\sqrt{2}z}+\varepsilon^2,\]
we get
\begin{eqnarray*}
\int_{\rho_\alpha^-(y)}^{\rho_\alpha^+(y)}e^{-2\sqrt{2}d_{\alpha-1}}g_\alpha^\prime
&\lesssim&\varepsilon^2+e^{-2\sqrt{2}d_{\alpha-1}(y,0)}\left[\int_{\rho_\alpha^-(y)}^0e^{-\sqrt{2}z}dz+ \int_0^{\rho_\alpha^+(y)}e^{-3\sqrt{2}z}dz\right]\\
&\lesssim&\varepsilon^2+e^{-2\sqrt{2}d_{\alpha-1}(y,0)-\sqrt{2}\rho_\alpha^-(y)}\\
&\lesssim&\varepsilon^2+e^{-\frac{3}{2}\sqrt{2}d_{\alpha-1}(y,0)}.
\end{eqnarray*}

Similarly, we have
\[\int_{\rho_\alpha^-(y)}^{\rho_\alpha^+(y)}e^{2\sqrt{2}d_{\alpha+1}}g_\alpha^\prime\lesssim\varepsilon^2+e^{\frac{3}{2}\sqrt{2}d_{\alpha+1}(y,0)},\]
\[\int_{\rho_\alpha^-(y)}^{\rho_\alpha^+(y)}O\left(e^{-\sqrt{2}d_{\alpha-2}-\sqrt{2}|z|}+e^{\sqrt{2}d_{\alpha+2}-\sqrt{2}|z|}\right)g_\alpha^\prime\lesssim e^{-\sqrt{2}d_{\alpha-2}}+e^{\sqrt{2}d_{\alpha+2}}.\]

To determine the first integral in \eqref{7.1}, arguing as above, if both $g_\alpha^\prime$ and $g_{\alpha-1}-(-1)^{\alpha-1}$ are nonzero, then
\[g_{\alpha-1}(y,z)=\bar{g}\left((-1)^{\alpha}(z+d_{\alpha-1}(y,0)+h_{\alpha-1}(\Pi_{\alpha-1}(y,z))+O(\varepsilon^{1/3}))\right).\]
Therefore
\begin{eqnarray*}
&&\int_{\rho_\alpha^-(y)}^{\rho_\alpha^+(y)}\left[W^{\prime\prime}(g_\alpha)-2\right]\left(g_{\alpha-1}-(-1)^{\alpha-1}\right)g_\alpha^\prime\\
&=&\int_{\rho_\alpha^-(y)}^{\rho_\alpha^+(y)}\left[W^{\prime\prime}\left(\bar{g}\left((-1)^{\alpha-1}(z-h_\alpha(y))\right)\right)-2\right]\bar{g}^\prime\left((-1)^{\alpha-1}(z-h_\alpha(y))\right)\\
&&\times\left[\bar{g}\left((-1)^{\alpha}(z+d_{\alpha-1}(y,0)+h_{\alpha-1}(\Pi_{\alpha-1}(y,z))+O(\varepsilon^{1/3}))\right)-(-1)^{\alpha-1}\right]
dz\\
&=&\int_{-\infty}^{+\infty}\left[W^{\prime\prime}\left(\bar{g}\left((-1)^{\alpha-1}(z-h_\alpha(y))\right)\right)-2\right]\bar{g}^\prime\left((-1)^{\alpha-1}(z-h_\alpha(y))\right)\\
&&\times\left[\bar{g}\left((-1)^{\alpha}(z+d_{\alpha-1}(y,0)+h_{\alpha-1}(\Pi_{\alpha-1}(y,z))+O(\varepsilon^{1/3}))\right)-(-1)^{\alpha-1}\right]
dz\\
&+&O(e^{-\frac{3\sqrt{2}}{2}d_{\alpha-1}(y,0)}) \\
&=&(-1)^\alpha 4A_{(-1)^\alpha}^2e^{-\sqrt{2}d_{\alpha-1}(y,0)}+O\left(|h_\alpha(y)|+|h_{\alpha-1}(\Pi_{\alpha-1}(y,z))|+\varepsilon^{1/3}\right)e^{-\sqrt{2}d_{\alpha-1}(y,0)}\\
&+&O(e^{-\frac{3\sqrt{2}}{2}d_{\alpha-1}(y,0)}) .
\end{eqnarray*}

In conclusion we get
\begin{eqnarray*}
&&\int_{-\delta R}^{\delta R}\left[W^\prime(g_\ast )-\sum_{\beta} W^\prime(g_\beta)\right]g_\alpha^\prime\\
&=&(-1)^\alpha\left[ 4A_{(-1)^\alpha}^2e^{-\sqrt{2}d_{\alpha-1}(y,0)}-4A_{(-1)^{\alpha-1}}^2e^{\sqrt{2}d_{\alpha+1}(y,0)}\right]+O(\varepsilon^2)\\
&+&O\left(|h_\alpha(y)|+|h_{\alpha-1}(\Pi_{\alpha-1}(y,z))|+\varepsilon^{1/3}\right)e^{-\sqrt{2}d_{\alpha-1}(y,0)}\\
&+&O\left(|h_\alpha(y)|+|h_{\alpha+1}(\Pi_{\alpha+1}(y,z))|+\varepsilon^{1/3}\right)e^{\sqrt{2}d_{\alpha+1}(y,0)}\\
&+&O(e^{-\frac{3\sqrt{2}}{2}d_{\alpha-1}(y,0)})+O(e^{\frac{3\sqrt{2}}{2}d_{\alpha+1}(y,0)})+O(e^{-\sqrt{2}d_{\alpha-2}(y,0)})+O(e^{\sqrt{2}d_{\alpha+2}(y,0)}).
\end{eqnarray*}

\item Finally, by the definition of $\xi_\beta$ (see Section \ref{sec optimal approximation}), for any $\beta$,
\[\int_{-\delta R}^{\delta R}\xi_\beta g_\alpha^\prime=O(\varepsilon^2).\]

\end{enumerate}

\medskip

Substituting all of these into \eqref{H eqn}, we obtain
\begin{eqnarray*}
&&(-1)^{\alpha-1}\left(\Delta_0h_\alpha+H^\alpha(y,0)\right)\int_{-\delta R}^{\delta R}\phi g_\alpha^{\prime\prime}+O(\varepsilon^2)\\
&+&O(|\nabla
h_\alpha(y)|+\varepsilon)\sup_{(-6|\log\varepsilon|,6|\log\varepsilon|)}\left(|\nabla_y^2\phi(y,z)|+|\nabla_y\phi(y,z)|+|\phi(y,z)|\right)e^{-\left(\sqrt{2}-\sigma\right)|z|}\\
&=&(-1)^\alpha\left[ 4A_{(-1)^\alpha}^2e^{-\sqrt{2}d_{\alpha-1}(y,0)}-4A_{(-1)^{\alpha-1}}^2e^{\sqrt{2}d_{\alpha+1}(y,0)}\right]\\
&+&O\left(|h_\alpha(y)|+|h_{\alpha-1}(\Pi_{\alpha-1}(y,z))|+\varepsilon^{1/3}\right)e^{-\sqrt{2}d_{\alpha-1}(y,0)}\\
&+&O\left(|h_\alpha(y)|+|h_{\alpha+1}(\Pi_{\alpha+1}(y,z))|+\varepsilon^{1/3}\right)e^{\sqrt{2}d_{\alpha+1}(y,0)}\\
&+&O(e^{-\frac{3\sqrt{2}}{2}d_{\alpha-1}(y,0)})+O(e^{\frac{3\sqrt{2}}{2}d_{\alpha+1}(y,0)})+O(e^{-\sqrt{2}d_{\alpha-2}(y,0)})+O(e^{\sqrt{2}d_{\alpha+2}(y,0)})\\
&+&\sup_{z\in(-6|\log\varepsilon|,6|\log\varepsilon|)}|\phi(y,z)|^2\\
&+&(-1)^\alpha\sigma_0\left[H^\alpha(y,0)+\Delta_0h_\alpha(y)\right]+O(|\nabla^2h_\alpha(y)|^2+|\nabla h_\alpha(y)|^2)+O(\varepsilon^2)\\
&+&\sum_{\beta\neq\alpha}|d_\beta(y,0)|e^{-\sqrt{2}|d_\beta(y,0)|}\left[\sup_{B_{\varepsilon^{1/3}}(y)}|H^\beta+\Delta^\beta h_\beta|+\varepsilon |\log\varepsilon|\sup_{B_{\varepsilon^{1/3}}(y)}\left(|\nabla^2h_\beta|+|\nabla h_\beta|\right)\right]\\
&+&\sum_{\beta\neq\alpha}|d_\beta(y,0)|e^{-\sqrt{2}|d_\beta(y,0)|}\sup_{B_{\varepsilon^{1/3}}(y)}
|\nabla h_\beta|^2.
\end{eqnarray*}
Then \eqref{Toda system} follows after some simplifications. In particular, we use Lemma \ref{control on h_0} to bound
$|\nabla^2h_\alpha(y)|^2+|\nabla h_\alpha(y)|^2$, and by considering the cases $|d_\beta(y,0)|>2|\log\varepsilon|$ and $|d_\beta(y,0)|<2|\log\varepsilon|$ separately, we get
\begin{eqnarray*}
|d_\beta(y,0)|e^{-\sqrt{2}|d_\beta(y,0)|}\varepsilon |\log\varepsilon|\sup_{B_{\varepsilon^{1/3}}(y)}\left(|\nabla^2h_\beta|+|\nabla h_\beta|\right)
&\lesssim &\varepsilon^{1/2}|d_\beta(y,0)|e^{-\sqrt{2}|d_\beta(y,0)|}\\
&\lesssim&\varepsilon^{1/3}e^{-\sqrt{2}|d_\beta(y,0)|}+\varepsilon^2.
\end{eqnarray*}

\section{Proof of Lemma \ref{lem 8.1}}\label{sec proof of Lem 8.1}
\setcounter{equation}{0}

In $\mathcal{N}_\alpha^2(r)$,
\begin{eqnarray*}
&&\Delta_z\phi-H^\alpha(y,z)\partial_z\phi+\partial_{zz}\phi\\
&=&W^{\prime\prime}(g_\ast)\phi+O(\phi^2)+O(e^{-\sqrt{2}D_\alpha(y)})+O(\varepsilon^2)-(-1)^{\alpha}g_\alpha^\prime\left[H^\alpha(y,0)+\Delta_0 h_\alpha(y)\right]-g_\alpha^{\prime\prime}|\nabla_zh_\alpha|^2\\
 &+&\left[O(\varepsilon^2)+O\left(|\nabla^2h_\alpha(y)|^2+|\nabla h_\alpha(y)|^2\right)\right]|z|g_\alpha^\prime\\
&+&\sum_{\beta\neq\alpha}O\left(e^{-\sqrt{2}|d_\beta(y,z)|}\right)\left[|H^\beta\left(\Pi_\beta(y,z)\right)+\Delta_0^\beta h_\beta(\Pi_\beta(y,z))|+|\nabla h_\beta\left(\Pi_\beta(y,z)\right)|^2\right]\\
&+&\sum_{\beta\neq\alpha}O\left(|d_\beta(y,z)|e^{-\sqrt{2}|d_\beta(y,z)|}\right)\left[\varepsilon^2+|\nabla^2 h_\beta\left(\Pi_\beta(y,z)\right)|^2+|\nabla h_\beta\left(\Pi_\beta(y,z)\right)|^2\right].
\end{eqnarray*}

We estimate the right hand side term by term.
\begin{itemize}
\item In $\mathcal{N}_\alpha^2(r)$,
\[W^{\prime\prime}(g_\ast)+O(\phi)\geq 2-Ce^{-cL}\geq c>0.\]

\item By the exponential decay of $g^\prime$ at infinity,
\[\Big|g_\alpha^\prime\left[H^\alpha(y,0)+\Delta_0 h_\alpha(y)\right]\Big|\lesssim e^{-cL}\big|H^\alpha(y,0)+\Delta_0 h_\alpha(y)\big|,\]
\[O(\varepsilon^2)|z|g_\alpha^\prime=O(\varepsilon^2).\]

\item Still by the exponential decay of $g^\prime$ at infinity,
\[
O\left(|\nabla^2h_\alpha(y)|^2+|\nabla h_\alpha(y)|^2\right)|z|g_\alpha^\prime+\big|g_\alpha^{\prime\prime}\big||\nabla_zh_\alpha|^2\lesssim e^{-cL}O\left(|\nabla^2h_\alpha(y)|^2+|\nabla h_\alpha(y)|^2\right).\]

\item For $\beta>\alpha$, if $e^{-\sqrt{2}|d_\beta|}\geq\varepsilon^2$,
\begin{eqnarray*}
|d_\beta(y,z)|=|d_\beta(y,0)|-z+O(\varepsilon^{1/3})&\geq&|d_{\alpha+1}(y,0)|-z+O(\varepsilon^{1/3})\\
&\geq&\frac{1}{2}|d_{\alpha+1}(y,0)|-C.
\end{eqnarray*}
Therefore
\begin{eqnarray*}
&&e^{-\sqrt{2}|d_\beta|}\left[|H^\beta\circ\Pi_\beta+(\Delta_0^\beta h_\beta)\circ\Pi_\beta|+|(\nabla h_\beta)\circ\Pi_\beta|^2\right]+e^{-\sqrt{2}|d_\beta|}\left[\varepsilon^2|d_\beta|+|(\nabla^2 h_\beta)\circ\Pi_\beta|^2+\varepsilon^2\right]\\
&\lesssim&\varepsilon^2+e^{-\frac{\sqrt{2}}{2}D_\alpha}|\left[|(\nabla h_\beta )\circ\Pi_\beta|^2+|(\nabla^2 h_\beta )\circ\Pi_\beta|^2\right]+e^{-\frac{\sqrt{2}}{2}D_\alpha}|H^\beta \circ\Pi_\beta+(\Delta_0^\beta h_\beta)\circ\Pi_\beta|\\
&\lesssim&\varepsilon^2+e^{-\sqrt{2}D_\alpha}+|H^\beta\circ\Pi_\beta+(\Delta_0^\beta h_\beta)\circ\Pi_\beta|^2+|(\nabla h_\beta)\circ\Pi_\beta|^4+|(\nabla^2 h_\beta)\circ\Pi_\beta|^4.
\end{eqnarray*}

\end{itemize}

\medskip

Putting these together we obtain
\begin{eqnarray*}
|E_\alpha^2(y,z)|&\lesssim& \varepsilon^2 + e^{-\sqrt{2}D_\alpha(y)}+|\nabla_0^2h_\alpha(y)|^2+|\nabla_0h_\alpha(y)|^2+e^{-cL}\big|H^\alpha(y,0)+\Delta_0 h_\alpha(y)\big|\\
&+&\sum_{\beta\neq\alpha}\sup_{B_{\varepsilon^{1/3}}(y)}\left[|H^\beta+\Delta_0^\beta h_\beta|^2+|\nabla h_\beta|^4+|\nabla^2 h_\beta|^4\right]\\
\end{eqnarray*}
By Lemma \ref{control on h_0}, this can be transformed into
\begin{eqnarray*}
|E_\alpha^2(y,z)|&\lesssim& \varepsilon^2+e^{-\sqrt{2}D_\alpha(y)}+|\nabla_y^2\phi(y,0)|^2+|\nabla_y\phi(y,0)|^2+e^{-cL}\big|H^\alpha(y,0)+\Delta_0 h_\alpha(y)\big|\\
&+&\sum_{\beta\neq\alpha}\sup_{B_{\varepsilon^{1/3}}(y)}\left[|H^\beta+\Delta_0^\beta h_\beta|^2+|\nabla h_\beta|^4+|\nabla^2 h_\beta|^4\right].
\end{eqnarray*}

\section{Proof of Lemma \ref{lem 8.2}}\label{sec proof of Lem 8.2}
\setcounter{equation}{0}

By Lemma \ref{upper bound on interaction}, in $\mathcal{M}_\alpha^0$ we have
\begin{eqnarray*}
&&\Delta_z\phi-H^\alpha(y,z)\partial_z\phi+\partial_{zz}\phi\\
&=&W^{\prime\prime}(g_\alpha)\phi+O(\phi^2)+O(e^{-\sqrt{2}D_\alpha(y)})+O(\varepsilon^2)-(-1)^{\alpha}g_\alpha^\prime\left[H^\alpha(y,0)+\Delta_0 h_\alpha(y)\right]-g_\alpha^{\prime\prime}|\nabla_zh_\alpha|^2\\
&+&\left[O\left(\varepsilon^2|z|\right)+O\left(\varepsilon|z|\right)O\left(|\nabla_0^2h_\alpha(y)|+|\nabla_0h_\alpha(y)|\right)\right]g_\alpha^\prime \\
&+&\sum_{\beta\neq\alpha}O\left(e^{-\sqrt{2}|d_\beta|}\right)\left[|H^\beta\left(\Pi_\beta(y,z)\right)+\Delta_0^\beta h_\beta(\Pi_\beta(y,z))|+|\nabla h_\beta\left(\Pi_\beta(y,z)\right)|^2+|\nabla^2 h_\beta\left(\Pi_\beta(y,z)\right)|^2\right]\\
&+&\sum_{\beta\neq\alpha}O\left(e^{-\sqrt{2}|d_\beta|}\right)\left[\varepsilon^2|d_\beta|+\varepsilon^2|d_\beta|^2\right].
\end{eqnarray*}
Lemma \ref{lem 8.2} follows from the following estimates.

\begin{itemize}

\item By the exponential decay of $g^\prime$ at infinity,
\begin{eqnarray*}
g_\alpha^{\prime\prime}|\nabla_zh_\alpha(y)|^2&=& O\left(|\nabla h_\alpha(y)|^2\right)e^{-\sqrt{2}|z|}.
\end{eqnarray*}

\item Still by the exponential decay of $g^\prime$ at infinity,
\[
\left[O\left(\varepsilon^2|z|\right)+O\left(\varepsilon|z|\right)O\left(|\nabla_0^2h_\alpha(y)|+|\nabla_0h_\alpha(y)|\right)\right]g_\alpha^\prime
=O(\varepsilon^2)+O\left(|\nabla_0^2h_\alpha(y)|^2+|\nabla_0h_\alpha(y)|^2\right)e^{-\sqrt{2}|z|}.\]

\item Finally, because in $\mathcal{N}_\alpha^1$,
\[e^{-\sqrt{2}|d_\beta|}\lesssim e^{-\sqrt{2}D_\alpha(y)},\]
the last two terms are bounded by $O\left(e^{-\sqrt{2}D_\alpha(y)}\right)$.

\end{itemize}

\section{Proof of Lemma \ref{Holder for RHS}}\label{sec proof of Lem 9.1}
\setcounter{equation}{0}

We estimate the H\"{o}lder norm of the right hand side of \eqref{error equation} term by term.

\begin{enumerate}

\item Because
\[\mathcal{R}(\phi)=W^\prime(g_\ast+\phi)-W^\prime(g_\ast)-W^{\prime\prime}(g_\ast)\phi,\]
we get
\[\|\mathcal{R}(\phi)\|_{C^{\theta}(B_1(y,z))}\lesssim\|\phi\|_{C^{\theta}(B_1(y,z))}^2.\]

\item By Lemma \ref{Holder bound on interaction}, in $\mathcal{M}_\alpha$,
\[\|W^\prime(g_\ast)-\sum_{\beta=1}^Q(-1)^{\beta-1}W^\prime(g_\beta)\|_{C^{\theta}(B_1(y,z))}\lesssim \sup_{B_1(y)}e^{-\sqrt{2}D_\alpha}+\varepsilon^2.\]

\item Take the decomposition
\begin{eqnarray*}
    g_\alpha^\prime\left[H^\alpha(y,z)+\Delta_zh_\alpha(y)\right]&=&g_\alpha^\prime\left[H^\alpha(y,0)+\Delta_0h_\alpha(y)\right]\\
    &+&g_\alpha^\prime\left[H^\alpha(y,z)-H^\alpha(y,0)\right]+g_\alpha^\prime\left[\Delta_0h_\alpha(y)-\Delta_zh_\alpha(y)\right].
    \end{eqnarray*}
First we have
\begin{eqnarray*}
&&\|g_\alpha^\prime\left[H^\alpha(y,z)-H^\alpha(y,0)\right]\|_{C^{\theta}(B_1(y,z))}\\
&\lesssim&\|g_\alpha^\prime\left[H^\alpha(y,z)-H^\alpha(y,0)\right]\|_{Lip(B_1(y,z))}\\
&\lesssim& \sup_{B_1(y,z)} e^{-\sqrt{2}|z|}\left[H^\alpha(y,z)-H^\alpha(y,0)\right]|+ \sup_{B_1(y,z)} |z||e^{-\sqrt{2}|z|}||A^\alpha(y,0)||\nabla A^\alpha(y,0)|\\
&\lesssim&\varepsilon^2.
\end{eqnarray*}

Next because
\[\Delta_0h_\alpha(y)-\Delta_zh_\alpha(y)=\sum_{i,j=1}^{n-1}\left(g^{ij}(y,z)-g^{ij}(y,0)\right)\frac{\partial^2h_\alpha}{\partial y_i\partial y_j}+\sum_{i=1}^{n-1}\left(b^i(y,z)-b^i(y,0)\right)\frac{\partial h_\alpha}{\partial y_i},\]
we get
\begin{eqnarray*}
&&\|g_\alpha^\prime\left[\Delta_0h_\alpha(y)-\Delta_zh_\alpha(y)\right]\|_{C^{\theta}(B_1(y,z))}\\
&\lesssim& \sup_{B_1(y,z)} e^{-\sqrt{2}|z|}\left[|g^{ij}(y,z)-g^{ij}(y,0)||\nabla^2h_\alpha(y)|+|b^i(y,z)-b^i(y,0)||\nabla h_\alpha(y)|\right]\\
&+& \sup_{B_1(y,z)} e^{-\sqrt{2}|z|}\left(|\nabla^2h_\alpha(y)|+|\nabla h_\alpha(y)|\right)\\
&&\times\left(\|g^{ij}(y,z)-g^{ij}(y,0)\|_{C^{\theta}(B_1(y,z))}+\|b^i(y,z)-b^i(y,0)\|_{C^{\theta}(B_1(y,z))}\right)\\
&+&\|h_\alpha\|_{C^{2,\theta}(B_1(y))}\sup_{B_1(y,z)} e^{-\sqrt{2}|z|}\left(|g^{ij}(y,z)-g^{ij}(y,0)|+|b^i(y,z)-b^i(y,0)|\right)\\
&\lesssim&\varepsilon\|h_\alpha\|_{C^{2,\theta}(B_1(y))}\\
&\lesssim&\varepsilon^2+\|h_\alpha\|_{C^{2,\theta}(B_1(y))}^2  \quad \quad \mbox{(by Cauchy inequality)}\\
&\lesssim&\varepsilon^2+\|\phi\|_{C^{2,\theta}(B_1(y,0))}^2+\sup_{B_1(y)}e^{-2\sqrt{2}D_\alpha}. \quad \quad \mbox{(by Lemma \ref{control on h_0})}
\end{eqnarray*}

\item In $\mathcal{M}_\alpha$, by Lemma \ref{control on h_0},
\begin{eqnarray*}
\|g_\alpha^{\prime\prime}|\nabla_zh_\alpha|^2\|_{C^{\theta}(B_1(y,z))}&\lesssim &\|\nabla h_\alpha\|_{L^\infty(B_1(y))}^2+\|\nabla h_\alpha\|_{L^\infty(B_1(y))}\|\nabla^2 h_\alpha\|_{L^\infty(B_1(y))}\\
&\lesssim& \|\phi\|_{C^{2,\theta}(B_1(y,0))}^2+\sup_{B_1(y)}e^{-2\sqrt{2}D_\alpha}.
\end{eqnarray*}

\item For $\beta\neq\alpha$, if $g_\beta^\prime\neq 0$, $|d_{\beta}(y,z)|<6|\log\varepsilon|$. Hence by Lemma \ref{comparison of distances}, $\mbox{dist}_\beta(\Pi_\beta(y,z),\Pi_\beta(y))\leq\varepsilon^{1/3}$.
Then similar to (3),
\begin{eqnarray*}
\|g_\beta^\prime\mathcal{R}_{\beta,1}\|_{C^{\theta}(B_1(y,z))}
&\lesssim& \varepsilon^2+e^{-\sqrt{2}|d_\beta(y,z)|}\left(\|\phi\|_{C^{2,\theta}(B_2^\beta(y,0))}^2+\sup_{B_2(y)}e^{-2\sqrt{2}D_\beta}\right)\\
&+&e^{-\sqrt{2}|d_\beta(y,z)|}\|H_\beta+\Delta_0^\beta h_\beta\|_{C^\theta(B_2^\beta(y,0))}.
\end{eqnarray*}

\item For $\beta\neq\alpha$, if $g_\beta^{\prime\prime}\neq 0$, $|d_{\beta}(y,z)|<6|\log\varepsilon|$. Hence by Lemma \ref{comparison of distances}, $\mbox{dist}_\beta(\Pi_\beta(y,z),\Pi_\beta(y))\leq\varepsilon^{1/3}$.
Then
\begin{eqnarray*}
\|g_\beta^{\prime\prime}\mathcal{R}_{\beta,2}\|_{C^{\theta}(B_1(y,z))}&\lesssim &e^{-\sqrt{2}|d_\beta(y,z)|}\left(\|\nabla h_\beta\|_{L^\infty(B_2(y))}^2+\|\nabla h_\beta\|_{L^\infty(B_2(y))}\|\nabla^2 h_\beta\|_{L^\infty(B_2(y))}\right)\\
&\lesssim& e^{-\sqrt{2}|d_\beta(y,z)|}\left(\|\phi\|_{C^{2,\theta}(B_2^\beta(y,0))}^2+\sup_{B_2(y)}e^{-2\sqrt{2}D_\beta}\right).
\end{eqnarray*}

\item For any $\beta\in\{1,\cdots, Q\}$,
\[\|\xi_\beta\|_{C^{\theta}(B_1(y,z))}\lesssim \|\xi_\beta\|_{Lip(B_1(y,z))}\lesssim \varepsilon^3.\]

\end{enumerate}

\medskip

Putting these together we finish the proof of Lemma \ref{Holder for RHS}.

\section{Derivation of \eqref{Schauder 2}}\label{sec proof of 9.3}
\setcounter{equation}{0}

First by \eqref{H eqn} and \eqref{orthogonal condition 2.1}-\eqref{orthogonal condition 2.3}, we have
\begin{eqnarray*}
&&\int_{-\delta R}^{\delta R}\left(\Delta_z\phi-\Delta_0\phi\right)g_\alpha^\prime +(-1)^{\alpha-1}\left(\int_{-\delta R}^{\delta R}\phi g_\alpha^{\prime\prime}\right)\Delta_0h_\alpha
+2(-1)^{\alpha-1}\int_{-\delta R}^{\delta R}g_\alpha^{\prime\prime}g^{ij}(y,0)\frac{\partial\phi}{\partial y_i}\frac{\partial h_\alpha}{\partial y_j}\\
&-&\left(\int_{-\delta R}^{\delta R}\phi g_\alpha^{\prime\prime\prime}\right)|\nabla_0 h_\alpha(y)|^2+\int_{-\delta R}^{\delta R}H^\alpha(y,z)g_\alpha^\prime\phi_z+\int_{-\delta R}^{\delta R}\xi_\alpha^\prime\phi\\
&=&\int_{-\delta R}^{\delta R}\left[W^{\prime\prime}(g_\ast)-W^{\prime\prime}(g_\alpha)\right]g_\alpha^\prime\phi+\int_{-\delta R}^{\delta R}\mathcal{R}(\phi)g_\alpha^\prime\\
&+&\int_{-\delta R}^{\delta R}\left[W^\prime(g_\ast)-\sum_{\beta=1}^Q W^\prime(g_\beta)\right]g_\alpha^\prime-(-1)^{\alpha}\left(\int_{-\delta R}^{\delta R}|g_\alpha^\prime|^2\right)\left[H^\alpha(y,0)+\Delta_0h_\alpha(y)\right]\\
&-&(-1)^{\alpha }\int_{-\delta R}^{\delta R}|g_\alpha^\prime|^2\left[H^\alpha(y,z)-H^\alpha(y,0)\right]-(-1)^{\alpha }\int_{-\delta R}^{\delta R}|g_\alpha^\prime|^2\left[\Delta_0h_\alpha(y)-\Delta_zh_\alpha(y)\right]\\
&+&\frac{1}{2}\left(\int_{-\delta R}^{\delta R}|g_\alpha^\prime|^2\frac{\partial}{\partial z}g^{ij}(y,z)\right)\frac{\partial h_\alpha}{\partial y_i}\frac{\partial h_\alpha}{\partial y_j}\\
&-&\sum_{\beta\neq\alpha}(-1)^{\beta }\int_{-\delta R}^{\delta R}g_\alpha^\prime g_\beta^\prime\mathcal{R}_{\beta,1}-\sum_{\beta\neq\alpha}\int_{-\delta R}^{\delta R}g_\alpha^\prime g_\beta^{\prime\prime}\mathcal{R}_{\beta,2}-\sum_{\beta\neq\alpha}\int_{-\delta R}^{\delta R}g_\alpha^\prime \xi_\beta.
\end{eqnarray*}

We estimate the H\"{o}lder norm of these terms one by one.

\begin{enumerate}
\item By \eqref{error in z 5},
 \begin{eqnarray*}
\Big\|\int_{-\delta R}^{\delta R}\left(\Delta_z\phi-\Delta_0\phi\right)g_\alpha^\prime\Big\|_{C^{\theta}(B_1(y))}&\lesssim&\varepsilon\sup_{|z|<6|\log\varepsilon|}e^{-(\sqrt{2}-\sigma)|z|}\|\phi\|_{C^{2,\theta}(B_1(y,z))}\\
&\lesssim&\varepsilon^2+\sup_{|z|<6|\log\varepsilon|}e^{-2(\sqrt{2}-\sigma)|z|}\|\phi\|_{C^{2,\theta}(B_1(y,z))}^2.
\end{eqnarray*}

\item
By the exponential decay of $g^\prime$ and Lemma \ref{control on h_0},
\begin{eqnarray*}
\Big\|\left(\int_{-\delta R}^{\delta R}\phi g_\alpha^{\prime\prime}\right)\Delta_0h_\alpha\Big\|_{C^{\theta}(B_1(y))}&\lesssim&\|h_\alpha\|_{C^{2,\theta}(B_1(y))}\sup_{|z|<6|\log\varepsilon|}e^{-(\sqrt{2}-\sigma)|z|}\|\phi\|_{C^{\theta}(B_1(y,z))}\\
&\lesssim&\|h_\alpha\|_{C^{2,\theta}(B_1(y))}^2+\sup_{|z|<6|\log\varepsilon|}e^{-2(\sqrt{2}-\sigma)|z|}\|\phi\|_{C^{\theta}(B_1(y,z))}^2\\
&\lesssim&\sup_{B_1(y)}e^{-2\sqrt{2}D_\alpha}+\sup_{|z|<6|\log\varepsilon|}e^{-2(\sqrt{2}-\sigma)|z|}\|\phi\|_{C^{\theta}(B_1(y,z))}^2.
\end{eqnarray*}

\item
By the exponential decay of $g^\prime$ and Lemma \ref{control on h_0},
 \begin{eqnarray*}
\Big\|\int_{-\delta R}^{\delta R}g_\alpha^{\prime\prime}g^{ij}(y,0)\frac{\partial\phi}{\partial y_i}\frac{\partial h_\alpha}{\partial y_j}\Big\|_{C^{\theta}(B_1(y))}
&\lesssim&\|h_\alpha\|_{C^{1,\theta}(B_1(y))}\sup_{|z|<6|\log\varepsilon|}e^{-(\sqrt{2}-\sigma)|z|}\|\phi\|_{C^{1,\theta}(B_1(y,z))}\\
&\lesssim&\|h_\alpha\|_{C^{1,\theta}(B_1(y))}^2+\sup_{|z|<6|\log\varepsilon|}e^{-2(\sqrt{2}-\sigma)|z|}\|\phi\|_{C^{1,\theta}(B_1(y,z))}^2\\
&\lesssim&\sup_{B_1(y)}e^{-2\sqrt{2}D_\alpha}+\sup_{|z|<6|\log\varepsilon|}e^{-2(\sqrt{2}-\sigma)|z|}\|\phi\|_{C^{\theta}(B_1(y,z))}^2.
\end{eqnarray*}

\item By the exponential decay of $g^\prime$ and Lemma \ref{control on h_0},
\begin{eqnarray*}
\Big\|\left(\int_{-\delta R}^{\delta R}\phi g_\alpha^{\prime\prime\prime}\right)|\nabla_0 h_\alpha(y)|^2\Big\|_{C^{\theta}(B_1(y))}&\lesssim&\|h_\alpha\|_{C^{1,\theta}(B_1(y))}^2\sup_{|z|<6|\log\varepsilon|}e^{-(\sqrt{2}-\sigma)|z|}\|\phi\|_{C^{\theta}(B_1(y,z))}\\
&\lesssim&\|h_\alpha\|_{C^{2,\theta}(B_1(y))}^2\\
&\lesssim&\sup_{B_1(y)}e^{-2\sqrt{2}D_\alpha}+\|\phi\|_{C^{\theta}(B_1(y,0))}^2.
\end{eqnarray*}

\item By \eqref{bound on 3rd derivatives} and the exponential decay of $g^\prime$ ,
\begin{eqnarray*}
\Big\|\int_{-\delta R}^{\delta R}H^\alpha(y,z)g_\alpha^\prime\phi_z\Big\|_{C^{\theta}(B_1(y))}&\lesssim&\varepsilon\sup_{|z|<6|\log\varepsilon|}e^{-(\sqrt{2}-\sigma)|z|}\|\phi\|_{C^{1,\theta}(B_1(y,z))}\\
&\lesssim&\varepsilon^2+\sup_{|z|<6|\log\varepsilon|}e^{-2(\sqrt{2}-\sigma)|z|}\|\phi\|_{C^{1,\theta}(B_1(y,z))}^2.
\end{eqnarray*}

\item
The $C^{\theta}(B_1(y))$ norm of $\int_{-\delta R}^{\delta R}\left[W^{\prime\prime}(g_\ast)-W^{\prime\prime}(g_\alpha)\right]g_\alpha^\prime\phi$ is bounded by
\begin{eqnarray*}
&&\|\phi\|_{C^{\theta}(B_1(y)\times(-6|\log\varepsilon|,6|\log\varepsilon|)}\sup_{B_1(y)}\left(\int_{-\delta R}^{\delta R}\left(|g_{\alpha-1}^2-1|+|g_{\alpha+1}^2-1|\right)g_\alpha^\prime\right)\\
&\lesssim&\|\phi\|_{C^{\theta}(B_1(y)\times(-6|\log\varepsilon|,6|\log\varepsilon|)}\sup_{B_1(y)}\left(D_\alpha e^{-\sqrt{2}D_\alpha}\right)\\
&\lesssim&\|\phi\|_{C^{\theta}(B_1(y)\times(-6|\log\varepsilon|,6|\log\varepsilon|)}^2+\sup_{B_1(y)}e^{-\frac{3\sqrt{2}}{2}D_\alpha}.
\end{eqnarray*}

\item The $C^{\theta}(B_1(y))$ norm of $\int_{-\delta R}^{\delta R}\mathcal{R}(\phi)g_\alpha^\prime$ is bounded by
\[\sup_{|z|<6|\log\varepsilon|}e^{-2(\sqrt{2}-\sigma)|z|}\|\phi\|_{C^{\theta}(B_1(y,z))}^2.\]

\item  By Lemma \ref{Holder bound on interaction}, the $C^{\theta}(B_1(y))$ norm of $\int_{-\delta R}^{\delta R}\left[W^\prime(g_\ast)-\sum_{\beta=1}^Q(-1)^{\beta-1}W^\prime(g_\beta)\right]g_\alpha^\prime$ is bounded by
$\sup_{B_1(y)}e^{-\sqrt{2}D_\alpha}$.

\item By the definition of $\bar{g}$ (see Subsection \ref{sec optimal approximation}), the $C^{\theta}(B_1(y))$ norm of $\int_{-\delta R}^{\delta R}|g_\alpha^\prime|^2-\sigma_0$ is bounded by $O(\varepsilon^2)$.

\item By \eqref{bound on 3rd derivatives} and \eqref{error in z 4}, the $C^{\theta}(B_1(y))$ norm of $\int_{-\delta R}^{\delta R}|g_\alpha^\prime|^2\left[H^\alpha(y,z)-H^\alpha(y,0)\right]$ is bounded by $O(\varepsilon^2)$.

\item
By \eqref{error in z 5} and Lemma \ref{control on h_0}, the $C^{\theta}(B_1(y))$ norm of $\int_{-\delta R}^{\delta R}|g_\alpha^\prime|^2\left[\Delta_0h_\alpha(y)-\Delta_zh_\alpha(y)\right]$ is bounded by
\[
\varepsilon\|h_\alpha\|_{C^{2,\theta}(B_1(y))}
\lesssim \varepsilon^2+\|h_\alpha\|_{C^{2,\theta}(B_1(y))}^2\\
\lesssim \sup_{B_1(y)}e^{-2\sqrt{2}D_\alpha}+\|\phi\|_{C^{\theta}(B_1(y,0))}^2.\]

\item By \eqref{metirc tensor} and  \eqref{metirc tensor 0}, the $C^{\theta}(B_1(y))$ norm of $\left(\int_{-\delta R}^{\delta R}|g_\alpha^\prime|^2\frac{\partial}{\partial z}g^{ij}(y,z)\right)\frac{\partial h_\alpha}{\partial y_i}\frac{\partial h_\alpha}{\partial y_j}$ is bounded by
\[
\varepsilon\|h_\alpha\|_{C^{1,\theta}(B_1(y))}^2\lesssim\sup_{B_1(y)}e^{-2\sqrt{2}D_\alpha}+\|\phi\|_{C^{\theta}(B_1(y,0))}^2.\]

\item
As in Case (11) and (12), for $\beta\neq\alpha$, the $C^{\theta}(B_1(y))$ norm of $\int_{-\delta R}^{\delta R}g_\alpha^\prime g_\beta^\prime\mathcal{R}_{\beta,1}$ is bounded by
\begin{eqnarray*}
&&\varepsilon^2+\left(\sup_{B_2(y)}\int_{-\delta R}^{\delta R}e^{-\sqrt{2}\left(|z|+|d_\beta(y,z)|\right)}\right)\left(\|\phi\|_{C^{2,\theta}(B_2^\beta(y,0))}^2+
\|H^\beta+\Delta_0^\beta h_\beta\|_{C^{\theta}(B_2^\beta(y,0))}\right)\\
&\lesssim&\varepsilon^2+\left(\sup_{B_2(y)}D_\alpha e^{-\sqrt{2}D_\alpha}\right)\left(\|\phi\|_{C^{2,\theta}(B_2^\beta(y,0))}^2+
\|H^\beta+\Delta_0^\beta h_\beta\|_{C^{\theta}(B_2^\beta(y,0))}\right)\\
&\lesssim&\varepsilon^2+\sup_{B_2(y)}e^{-\frac{3\sqrt{2}}{2}D_\alpha}+\|\phi\|_{C^{2,\theta}(B_2^\beta(y,0))}^4+
\|H^\beta+\Delta_0^\beta h_\beta\|_{C^{\theta}(B_2^\beta(y,0))}^2.
\end{eqnarray*}

\item
For $\beta\neq\alpha$, the $C^{\theta}(B_1(y))$ norm of $\int_{-\delta R}^{\delta R}g_\alpha^\prime g_\beta^{\prime\prime}\mathcal{R}_{\beta,2}$ is bounded by
\begin{eqnarray*}
&&\varepsilon^2+\left(\sup_{B_2(y)}\int_{-\delta R}^{\delta R}e^{-\sqrt{2}\left(|z|+|d_\beta(y,z)|\right)}\right)\left(\|\phi\|_{C^{2,\theta}(B_2^\beta(y,0))}^2+\sup_{B_2(y)}e^{-2\sqrt{2}D_\beta}\right)\\
&\lesssim&\varepsilon^2+\left(\sup_{B_2(y)}D_\alpha e^{-\sqrt{2}D_\alpha}\right)\left(\|\phi\|_{C^{2,\theta}(B_2^\beta(y,0))}^2+\sup_{B_2(y)}e^{-2\sqrt{2}D_\beta}\right)\\
&\lesssim&\varepsilon^2+\sup_{B_2(y)}e^{-\frac{3\sqrt{2}}{2}D_\alpha}+\|\phi\|_{C^{2,\theta}(B_2^\beta(y,0))}^4+\sup_{B_2(y)}e^{-4\sqrt{2}D_\beta}.
\end{eqnarray*}

\item By the definition of $\xi$, the $C^{\theta}(B_1(y))$ norm of $\sum_{\beta\neq\alpha}(-1)^{\beta-1}\int_{-\delta R}^{\delta R}g_\alpha^\prime \xi_\beta$ is bounded by
 $O(\varepsilon^2)$.

\end{enumerate}

\medskip

In conclusion, we get
\begin{eqnarray*}
\|H^\alpha+\Delta_0h_\alpha\|_{C^{\theta}(B_1(y))}&\lesssim&\varepsilon^2+\sup_{B_2(y)}e^{-\sqrt{2}D_\alpha}+\|\phi\|_{C^{2,\theta}(B_2(y)\times(-6|\log\varepsilon|,6|\log\varepsilon|)}^2\\
&+&\sum_{\beta\neq\alpha}\left[\|\phi\|_{C^{2,\theta}(B_2^\beta(y,0))}^4+
\|H^\beta+\Delta_0^\beta h_\beta\|_{C^{\theta}(B_2^\beta(y,0))}^2+\sup_{B_2^\beta(y,0)}e^{-4\sqrt{2}D_\beta}\right].
\end{eqnarray*}
Hence
\begin{eqnarray}
\|H^\alpha+\Delta_0h_\alpha\|_{C^{\theta}(B_r)}&\lesssim&\varepsilon^2+ \sup_{B_{r+L}}e^{-\sqrt{2}D_\alpha}+ \|\phi\|_{C^{2,\theta}(\mathcal{D}_{r+L})}^2
+ \sum_{\beta\neq\alpha}\|H^\beta+\Delta_0^\beta h_\beta\|_{C^{\theta}(B_{r+L})}^2+ \sum_{\beta\neq\alpha}\sup_{B_2(y)}e^{-4\sqrt{2}D_\beta}\nonumber\\
&\leq&C\varepsilon^2+C\sum_{\beta}\sup_{B_{r+L}}e^{-\sqrt{2}D_\beta}+\sigma\|\phi\|_{C^{2,\theta}(\mathcal{D}_{r+L})}
+\sigma\sum_{\beta}\|H^\beta+\Delta_0^\beta h_\beta\|_{C^{\theta}(B_{r+L})}.
\end{eqnarray}
\eqref{Schauder 2} follows after some simplification.

\section{Proof of Lemma \ref{bound on remainder term}}\label{sec proof of Lem 10.1}
\setcounter{equation}{0}
In this appendix we prove Lemma \ref{bound on remainder term} by estimating each term in $E_i$.
\begin{enumerate}
\item First we prove
\begin{lem}\label{commutator 1}
\[\frac{\partial}{\partial y_i}\Delta_z\varphi-\Delta_z\varphi_i=O(\varepsilon)\left(|\nabla^2\varphi(y)|+|\nabla\varphi(y)|\right).\]
\end{lem}
\begin{proof}
Because $|\nabla f_\alpha|\leq C$, $|\nabla^2f_\alpha|\lesssim\varepsilon$ and
\[g_{ij}(y,0)=\delta_{ij}+f_{\alpha,i}(y)f_{\alpha,j}(y),\]
we have
\[\big|\nabla_y g_{ij}(y,0)\big|\lesssim\varepsilon.\]
By Lemma \ref{bound on 3rd derivatives}, we see
\[
|\nabla_y g^{ij}(y,z)|+|\nabla_y^2 g^{ij}(y,z)|=O(\varepsilon).\]
Then
\begin{eqnarray*}
\frac{\partial}{\partial y_i}\Delta_z\varphi&=&\Delta_z\varphi_i+\sum_{k,l=1}^{n-1}\left(\frac{\partial}{\partial y_i}g^{kl}(y,z)\right)\frac{\partial^2\varphi}{\partial y_k\partial y_l}(y)+\sum_{k=1}^{n-1}\left(\frac{\partial}{\partial y_i}b^k(y,z)\right)\frac{\partial\varphi}{\partial y_k}(y)\\
&=&\Delta_z\varphi_i+O(\varepsilon)\left(|\nabla^2\varphi(y)|+|\nabla\varphi(y)|\right). \qedhere
\end{eqnarray*}
\end{proof}

Using the above lemma, we get
\[
\|\Delta_z\phi_i-\partial_{y_i}\Delta_z\phi\|_{L^\infty(\mathcal{M}_\alpha^2(r))}\lesssim \varepsilon^2+\|\phi\|_{C^{2,\theta}(\mathcal{M}_\alpha^2(r))}^2.
\]

\item
Because
\[
\frac{\partial H^\alpha}{\partial y_i}(y,z)=\frac{\partial}{\partial y_i}\sum_{k,l=1}g^{kl}(y,z)\left[\left(I-zA^\alpha(y,0)\right)^{-1}A^\alpha(y,0)\right]_{kl}=O(\varepsilon),\]
the second term
\[\Big\|\frac{\partial H^\alpha}{\partial y_i}(y,z)\phi_z\Big\|_{L^\infty(\mathcal{M}_\alpha^2(r))}=O(\varepsilon)\|\nabla\phi\|_{L^\infty(\mathcal{M}_\alpha^2(r))}\lesssim \varepsilon^2+\|\phi\|_{C^{2,\theta}(\mathcal{M}_\alpha^2(r))}^2.\]

\item Similarly,
\[\|H^\alpha(y,z)\partial_z\phi_i\|_{L^\infty(\mathcal{M}_\alpha^2(r))}=O(\varepsilon)\|\nabla^2\phi\|_{L^\infty(\mathcal{M}_\alpha^2(r))}\lesssim \varepsilon^2+\|\phi\|_{C^{2,\theta}(\mathcal{M}_\alpha^2(r))}^2.\]

\item In $\mathcal{M}^2_\alpha(r)$, we have
\begin{eqnarray*}
\Big|W^{\prime\prime}(g_\ast+\phi)-W^{\prime\prime}(g_\alpha)\Big|&\lesssim&|\phi|+\sum_{\beta\neq\alpha}g_\beta^\prime\\
&\lesssim&\|\phi\|_{L^\infty(\mathcal{M}_\alpha^2(r))}+\sup_{B_r}e^{-\frac{\sqrt{2}}{2}D_\alpha}+\varepsilon.
\end{eqnarray*}
By the Cauchy inequality we obtain
\[
\Big|\left[W^{\prime\prime}(g_\ast+\phi)-W^{\prime\prime}(g_\alpha)\right]\phi_i\Big|\lesssim\|\phi\|_{C^{2,\theta}(\mathcal{M}_\alpha^2(r))}^2
+\|\phi\|_{C^{2,\theta}(\mathcal{M}_\alpha^2(r))}\sup_{B_r}e^{-\frac{\sqrt{2}}{2}D_\alpha}+\varepsilon^2.
\]

\item Similarly
\[
\Big|\left[W^{\prime\prime}(g_\ast+\phi)-W^{\prime\prime}(g_\alpha)\right]g_\alpha^\prime h_{\alpha,i}(y)\Big|\lesssim\|\phi\|_{C^{2,\theta}(\mathcal{M}_\alpha^2(r))}^2+\|\phi\|_{C^{2,\theta}(\mathcal{M}_\alpha^2(r))}\sup_{B_r}e^{-\frac{\sqrt{2}}{2}D_\alpha}+\varepsilon^2.
\]

\item For $\beta\neq\alpha$, if $g_\beta^\prime\neq0$, by Lemma \ref{comparison of distances},
\[\Big|\frac{\partial d_\beta}{\partial y_i}\Big|\lesssim\varepsilon^{1/5}.\]
Therefore
\begin{eqnarray*}
&&\Big|\left[W^{\prime\prime}(g_\ast+\phi)-W^{\prime\prime}(g_\beta)\right] g_\beta^\prime\left[\frac{\partial d_\beta}{\partial y_i}-\sum_{j=1}^{n-1} h_{\beta,j}\left(\Pi_\beta(y,z)\right)\frac{\partial\Pi_\beta^j}{\partial y_i}(y,z)\right]\Big|\\
&\lesssim&\left(|\phi|+\sum_{\beta\neq\alpha}g_\beta^\prime\right)g_\beta^\prime \left(\varepsilon^{\frac{1}{5}}+\sup_{B_{\varepsilon^{1/3}(y)}}|\nabla h_\beta|\right)\\
&\lesssim&\left(\|\phi\|_{L^\infty(\mathcal{M}_\alpha^2(r))}+\sup_{B_r}e^{-\frac{\sqrt{2}}{2}D_\alpha}\right) \left(\varepsilon^{\frac{1}{5}}+\sup_{B_{\varepsilon^{1/3}(y)}}|\nabla h_\beta|\right)\sup_{B_r}e^{-\frac{\sqrt{2}}{2}D_\alpha}\\
&\lesssim&\left(\varepsilon^{\frac{1}{5}}+\sup_{B_{\varepsilon^{1/3}(y)}}|\nabla h_\beta|\right)\left(\|\phi\|_{L^\infty(\mathcal{M}_\alpha^2(r))}^2+\sup_{B_r}e^{-\sqrt{2} D_\alpha}\right) \\
&\lesssim&\|\phi\|_{C^{2,\theta}(\mathcal{M}_\alpha^2(r))}^2+\sup_{B_r}e^{-2\sqrt{2}D_\alpha}+\varepsilon^{1/5}\sup_{B_r}e^{-\sqrt{2}D_\alpha}.
\end{eqnarray*}

\item Because
\[
H^\alpha(y,z)+\Delta_z h_\alpha(y)=H^\alpha(y,0)+\Delta_0 h_\alpha(y)+O(\varepsilon^2|z|)+O(\varepsilon|z|)\left(|\nabla^2h_\alpha(y)|+|\nabla h_\alpha|\right),
\]
in $\mathcal{M}_\alpha^2(r)$, we have
\begin{eqnarray*}
&&\Big| g_\alpha^{\prime\prime}h_{\alpha,i}(y)\left[H^\alpha(y,z)+\Delta_z h_\alpha(y)\right]\Big|\\
&\lesssim& \big|h_{\alpha,i}(y)\big|\big|H^\alpha(y,0)+\Delta_0 h_\alpha(y)\big|+O(\varepsilon)\big|h_{\alpha,i}(y)\big|\\
&\lesssim&\|\phi\|_{C^{2,\theta}(\mathcal{M}_\alpha^2(r))}^2+\sup_{B_r}\big|H^\alpha(y,0)+\Delta_0 h_\alpha(y)\big|^2+\varepsilon^2+\sup_{B_r}e^{-2\sqrt{2}D_\alpha}.
\end{eqnarray*}

\item First by \eqref{A(z)} we have
\[\frac{\partial H^\alpha}{\partial y_i}(y,z)=\frac{\partial H^\alpha}{\partial y_i}(y,0)+O(\varepsilon^2).
\]
Next by Lemma \ref{commutator 1},
\[
\frac{\partial}{\partial y_i}\left(\Delta_z h_\alpha(y)\right)=\Delta_z h_{\alpha,i}(y)+O(\varepsilon)\left(|\nabla^2h_\alpha(y)|+|\nabla h_\alpha(y)|\right).
\]
Combining these two estimates we obtain
\[g_\alpha^\prime \left[\frac{\partial H^\alpha}{\partial y_i}(y,z)+\frac{\partial H^\alpha}{\partial y_i}(y,0)-\frac{\partial}{\partial y_i}\left(\Delta_z h_\alpha(y)\right)-\Delta_0h_{\alpha,i}(y)\right]=O\left(\varepsilon^2+\|\phi\|_{C^{2,\theta}(\mathcal{M}_\alpha^2(r))}^2+\sup_{B_r}e^{-2\sqrt{2}D_\alpha}\right).\]

\item By a direct expansion, we get
\[(-1)^\alpha g_\alpha^{\prime\prime\prime}|\nabla_zh_\alpha|^2h_{\alpha,i}(y)+ g_\alpha^{\prime\prime}\frac{\partial}{\partial y_i}|\nabla_zh_\alpha|^2 =O\left(\|\phi\|_{C^{2,\theta}(\mathcal{M}_\alpha^2(r))}^2+\sup_{B_r}e^{-2\sqrt{2}D_\alpha}\right).\]

\item
We have
\begin{eqnarray}\label{G 1}
&&\frac{\partial}{\partial y_i} \left[g_\beta^\prime\mathcal{R}_{\beta,1}\left(\Pi_\beta(y,z),d_\beta(y,z)\right)\right] \\
&=&(-1)^\beta g_\beta^{\prime\prime}\mathcal{R}_{\beta,1}\sum_{j=1}^{n-1}\frac{\partial h_\beta}{\partial y_j}\left(\Pi_\beta(y,z)\right)\frac{\partial\Pi_\beta^j}{\partial y_i}(y,z)+g_\beta^\prime\sum_{j=1}^{n-1}\frac{\partial\mathcal{R}_{\beta,1}}{\partial y_j}\left(\Pi_\beta(y,z)\right)\frac{\partial\Pi_\beta^j}{\partial y_i}(y,z). \nonumber
\end{eqnarray}
The first term in the right hand side is bounded by
\begin{eqnarray*}
&&e^{-\sqrt{2}|d_\beta(y,z)|}|\nabla h_\beta(\Pi_\beta(y,z))| \big|\mathcal{R}_{\beta,1}(\Pi_\beta(y,z))\big|\\
&\lesssim& e^{-\sqrt{2}|d_\beta(y,z)|}|\nabla h_\beta(\Pi_\beta(y,z))| \\
&&\times\left[\big|H^\beta(\Pi_\beta(y,z))+\Delta_0^\beta h_\beta(\Pi_\beta(y,z))\big|+\varepsilon^2|d_\beta|+\varepsilon |d_\beta|\left(|\nabla^2 h_\beta(\Pi_\beta(y,z))|+|\nabla h_\beta(\Pi_\beta(y,z))|\right)\right]
\\
&\lesssim&\varepsilon^2+\sup_{B_{r+1}}e^{-2\sqrt{2}D_\alpha}+\|\phi\|_{C^{2,\theta}(\mathcal{M}_\beta^0(r+1))}^4+\sup_{B_{r+1}^\beta}\big|H^\beta+\Delta_0^\beta h_\beta\big|^2.
\end{eqnarray*}

Next, similar to Case (8), we have
\[\Big|\frac{\partial\mathcal{R}_{\beta,1}}{\partial y_j}\left(\Pi_\beta(y,z)\right)\Big|\lesssim\Big|\nabla H_\beta+\Delta_0^\beta \nabla h_\beta\Big|+\varepsilon^2+\sup_{B^\beta_{r+1}}e^{-2\sqrt{2}D_\beta} + \|\phi\|_{C^2(\mathcal{M}_\beta^0(r+1))}^2.\]
Therefore the second term in the right hand side of \eqref{G 1} is bounded by
\begin{eqnarray*}
&&\left(\sup_{B_{r+1}}e^{-\frac{\sqrt{2}}{2}D_\alpha}\right)\left[\Big|\nabla H_\beta+\Delta_0^\beta \nabla h_\beta\Big|+\varepsilon^2+\sup_{B^\beta_{r+1}}e^{-2\sqrt{2}D_\beta} + \|\phi\|_{C^2(\mathcal{M}_\beta^0(r+1))}^2\right].
\end{eqnarray*}

\item
We have
\begin{eqnarray*}
&&\frac{\partial}{\partial y_i} \left[g_\beta^{\prime\prime}\mathcal{R}_{\beta,2}\left(\Pi_\beta(y,z),d_\beta(y,z)\right)\right]\\
&=&-g_\beta^{\prime\prime\prime}\mathcal{R}_{\beta,2}\sum_{j=1}^{n-1}\frac{\partial h_\beta}{\partial y_j}\left(\Pi_\beta(y,z)\right)\frac{\partial\Pi_\beta^j}{\partial y_i}(y,z)+g_\beta^{\prime\prime}\sum_{j=1}^{n-1}\frac{\partial\mathcal{R}_{\beta,2}}{\partial y_j}\left(\Pi_\beta(y,z)\right)\frac{\partial\Pi_\beta^j}{\partial y_i}(y,z).
\end{eqnarray*}
Therefore this term is controlled by
\[\left(\sup_{B_{r+1}}e^{-\frac{\sqrt{2}}{2}D_\alpha}\right) \sum_{\beta\neq\alpha}\|\phi\|_{C^2(\mathcal{M}_\beta^0(r+1))}^2.\]

\item Because $\xi_\beta^\prime=O(\varepsilon^3)$,
\[\sum_\beta (-1)^\beta\xi_\beta^\prime \sum_{j=1}^{n-1}h_{\beta,j}\left(\Pi_\beta(y,z)\right)\frac{\partial\Pi_\beta^j}{\partial y_i}(y,z)=O(\varepsilon^2).\]

\end{enumerate}

\medskip

Putting these together we finish the proof of Lemma \ref{bound on remainder term}.



\begin{thebibliography}{50}
\small



\bibitem {AAC} G. Alberti, L. Ambrosio and  X. Cabr\'{e}, On a long-standing
conjecture of E. De Giorgi: symmetry in 3D for general nonlinearities and a
local minimality property. Special issue dedicated to Antonio Avantaggiati on
the occasion of his 70th birthday. {\em Acta Appl. Math.} 65 (2001), no. 1-3, 9--33.


\bibitem{A-C}
 A. Ambrosio and X. Cabr\'{e}, Entire solutions of seminlinear elliptic equations in $\mathbb{R}^3$ and a conjecture of De Giorgi,
 {\em J. American Math. Soc.} 13 (2000), 725-739.

\bibitem{AM} F. Alessio and P.  Montecchiari,  Multiplicity of layered solutions for Allen-Cahn systems with symmetric double well potential. {\em J. Differential Equations } 257 (2014), no. 12, 4572�4599.

\bibitem{ADW1} O.  Agudelo, M. del Pino and J. Wei, Solutions with multiple catenoidal ends to the Allen-Cahn equation in $R^3$. {\em J. Math. Pures Appl.} (9) 103 (2015), no. 1, 142�218.

 \bibitem{ADW2} O. Agudelo, M. del Pino and J. Wei, Higher dimensional catenoid, Liouville equation and Allen-Cahn equation Liouville's equation. {\em  International Math. Research Note (IMRN)}, to appear.



 \bibitem{BCN}
H. Berestycki, L.A. Caffarelli, L. Nirenberg, Monotonicity for
elliptic equations in unbounded Lipschitz domains. {\em Comm. Pure Appl.
Math. } 50(11) (1997), 1089-1111.

\bibitem{Bers}
L. Bers,  Local behavior of solutions of general linear elliptic equations. {\em Comm. Pure Appl. Math.} 8 (1955), 473-496.

\bibitem{Ca} X. Cabr\'{e}, Uniqueness and stability of saddle-shaped solutions
to the Allen-Cahn equation. {\em J. Math. Pures Appl.} (9) 98 (2012), no. 3, 239--256.


\bibitem{CC1}
X. Cabr\'{e} and S. Chanillo, Stable solutions of semilinear elliptic problems in convex domains. {\em Selecta Math. (N.S.)} 4 (1998), no. 1, 1-10.

\bibitem{CC}
L. A. Caffarelli and A. C\'{o}rdoba, Phase transitions: uniform
regularity of the intermediate layers, {\em J. Reine Angew. Math.} 593
(2006), 209-235.


\bibitem{Cao-S-Z}
H.-D. Cao, Y. Shen and S. Zhu, The structure of stable minimal
hypersurfaces in $\R^{n+1}$, {\em Math. Res. Lett.} 4 (1997), no. 5,
637-644.


\bibitem{Choe}
J. Choe, Index, vision number and stability of complete minimal
surfaces,{\em  Arch. Rational Mech. Anal. } 109 (1990), no. 3, 195-212.

\bibitem{Choi-Schoen}
H. Choi and  R. Schoen, The space of minimal embeddings of a surface into a three-dimensional manifold of positive Ricci curvature.
{\em Invent. Math.}  81 (1985), no. 3, 387-394.

\bibitem{CM1} T. Colding and W. P. Minicozzi II, The space of embedded minimal surfaces of fixed genus in a 3-manifold.
I. Estimates off the axis for disks. {\em Ann. of Math. }(2)(2004), 160(1), 27-68.

\bibitem{CM2} T. Colding and W. P. Minicozzi II, The space of embedded minimal surfaces of fixed genus in a 3-manifold.
II. Multi-valued graphs in disks. {\em Ann. of Math.} (2) (2004), 160(1), 69-92.

\bibitem{CM3} T. Colding and W. P. Minicozzi II, The space of embedded minimal surfaces of fixed genus in a 3-manifold.
III. Planar domains. {\em  Ann. of Math.} (2) (2004), 160(2), 523-572.

\bibitem{CM4} T. Colding and W. P. Minicozzi II,  The space of embedded minimal surfaces of fixed genus in a 3-manifold.
IV. Locally simply connected. {\em Ann. of Math.} (2) (2004), 160(2), 573-615.

\bibitem{CM5} T. Colding and W. P. Minicozzi II,
The Calabi-Yau conjectures for embedded surfaces. {\em Ann. of Math.} (2) 167 (2008), no. 1, 211-243.


\bibitem{CM6} T. Colding and W. P. Minicozzi II,
The space of embedded minimal surfaces of fixed genus in a 3-manifold V; fixed genus. {\em Ann. of Math.} (2) 181 (2015), no. 1, 1-153.

\bibitem{Colding}
T. Colding and W. P. Minicozzi II, A course in minimal surfaces. Graduate Studies in Mathematics, 121. American Mathematical Society, Providence, RI, 2011.



\bibitem{DG} E. de Giorgi, {\em Convergence problems for functionals and operators}, Proc. Int. Meeting on
Recent Methods in Nonlinear Analysis (Rome, 1978), 131-188.

\bibitem{Dancer}
E. N. Dancer, Stable and finite Morse index solutions on $\R^n$ or on
bounded domains with small diffusion II, {\em Indiana University
Mathematics Journal} 53 (2004), no. 1, 97-108.





\bibitem{DKP}
M. del Pino, M. Kowalczyk and F. Pacard, Moduli space theory for the
Allen-Cahn equation in the plane, {\em Trans. Amer. Math. Soc. } 365
(2013), no. 2, 721-766.




\bibitem{DKW}  M. del Pino, M. Kowalczyk and J. Wei,  On De Giorgi Conjecture in Dimensions $N \geq 9$,  {\em Annals of Mathematics} 174 (2011), no.3, 1485-1569.


\bibitem{DKW 2}
M. del Pino, M. Kowalczyk and J. Wei, The Toda system and clustering
interfaces in the Allen-Cahn equation, {\em Arch. Rational Mech. Anal.}
190 (2008), 141-187.



\bibitem{DKW 3}
M. del Pino, M. Kowalczyk, F. Pacard and J. Wei, Multiple end solutions to the
Allen-Cahn equation in $\R^2$, {\em J. Funct. Anal.  } 258 (2010) 458-503.



\bibitem{DKWY}
M. del Pino, M. Kowalczyk, J. Wei and J. Yang, Interface foliation near minimal submanifolds in Riemannian manifolds with positive Ricci curvature. {\em  Geom. Funct. Anal.} 20(2010), no.4, 918-957.

\bibitem{dev}
B. Devyver, On the finiteness of the Morse index for Schr\"{o}dinger
operators, {\em Manuscripta Math. } 139 (2012), no.1-2 249-271.


\bibitem{Carmo-Peng}
M. do Carmo and C.-K. Peng, Stable complete minimal surfaces in
$\mathbb{R}^2$ are planes, {\em Bull. Amer. Math. Soc. } 1 (1979), 903-906.

\bibitem{Du-Gui-Wang}
Z. Du, C. Gui and K. Wang, Four ends solutions of a free boundary problem, in preparation.


\bibitem{F0} A. Farina, private commuincations.

\bibitem{FSV} A. Farina, B. Sciunzi and  Valdinoci,  Bernstein and De Giorgi type problems: New results
via a geometric approach. {\em Ann. Scuola. Norm. Sup. Pisa Cl. Sci.} (5)(2008), Vol. 7, 741�791

\bibitem{FV1} A. Farina and E. Valdinoci, Rigidity results for elliptic PDEs with uniform limits: an abstract framework with
applications. {\em Indiana Univ. Math. J.} 60(1)(2011), 121-141.

\bibitem{FV2}  A. Farina and E. Valdinoci, 1D symmetry for solutions of semilinear and quasilinear elliptic equations. {\em  Trans.
Amer. Math. Soc. } 363(2)(2011), 579-609.


\bibitem{FV3} A.  Farina and E. Valdinoci, 1D symmetry for semilinear PDEs from the limit interface of the solution. {\em Comm. Partial Differential Equations} 41 (2016), no. 4, 665�682.



\bibitem{Fischer}
D. Fischer-Colbrie, On complete minimal surface with finite Morse
index in three manifolds, {\em Inven. Math.} 82 (1985), 121-132.


\bibitem{F-Schoen}
D. Fischer-Colbrie and R. Schoen, The structure of complete stable
minimal surfaces in 3-manifolds of non-negative scalar curvature,
{\em Comm. Pure Appl. Math.} 33 (1980), 199-211.



\bibitem{GG}
N. Ghoussoub and C. Gui, On a conjecture of De Giorgi and some
related problems, {\em Math. Ann. } 311 (1998), 481-491.


\bibitem{GG2} N. Ghoussoub and C. Gui, On De Giorgi's conjecture in dimensions 4 and 5. {\em Ann. of Math.} (2) 157 (2003), no. 1, 313-334.

\bibitem{GT}
D. Gilbarg and N. Trudinger, Elliptic Partial Differential Equations of Second Order, Springer 2001.


\bibitem{Gui}
C. Gui, Hamiltonian identities for elliptic partial differential
equations, {\em J. Funct. Anal. } 254 (2008), no. 4, 904-933.


\bibitem{Gui 1}
C. Gui, Symmetry of some entire solutions to the Allen-Cahn equation
in two dimensions, {\em  J. Differential Equations}  252 (2012), no. 11,
5853-5874.


\bibitem{Gui-Liu-Wei}
C. Gui, Y. Liu and J. Wei,  Two-end solutions to the Allen-Cahn equation in $\mathbb{R}^3$, arXiv:1502.05963.


\bibitem{Gulliver}
R. Gulliver, Index and total curvature of complete minimal surfaces. Geometric measure theory and the calculus of variations (Arcata, Calif., 1984), 207-211, Proc. Sympos. Pure Math., 44, Amer. Math. Soc., Providence, RI, 1986.


\bibitem{Gulliver 2}
R. Gulliver and H. Blaine Lawson  Jr. The structure of stable minimal hypersurfaces near a singularity. Geometric measure theory and the calculus of variations (Arcata, Calif., 1984), 213-237,
Proc. Sympos. Pure Math., 44, Amer. Math. Soc., Providence, RI, 1986.


\bibitem{GNY}
 A. Grigor'an, Y. Netrusov and S.-T. Yau, Eigenvalues of elliptic operators and geometric
applications, Surveys in differential geometry. Vol. IX, Surv. Differ. Geom., IX, Int. Press,
Somerville, MA, 2004, pp. 147-217.


\bibitem{H-T}
J. Hutchinson and Y. Tonegawa, Convergence of phase interfaces in
the van der Waals-Cahn-Hilliard theory, {\em Calc. Var. PDEs } 10 (2000),
no. 1, 49-84.



\bibitem{JK} D. Jerison and N. Kamburov,  Structure of one-phase free boundaries in the plane,  arXiv:1412.4106.



\bibitem {JM}
D. Jerison, R. Monneau, Towards a counter-example to a
conjecture of De Giorgi in high dimensions.{\em  Ann. Mat. Pura Appl.} (4) 183
(2004), no. 4, 439--467.



\bibitem{KLP 1}
M. Kowalczyk, Y. Liu and F. Pacard, The space of 4-ended solutions
to the Allen-Cahn equation in the plane, {\em Ann. Inst. H. Poincar\'{e}
Anal. Non Lin\'{e}aire } 29 (2012), no. 5, 761-781.

\bibitem{KLP 2}
M. Kowalczyk, Y. Liu and F. Pacard, The classification of four-end
solutions to the Allen-Cahn equation on the plane, {\em  Analysis $\&$ PDE}
6 (2013), no. 7, 1675-1718.

\bibitem{KLP 3}
M. Kowalczyk, Y. Liu and F. Pacard, Towards classification of
multiple-end solutions to the allen-cahn equation in $\R^2$,
{\em Networks and Heterogeneous Media} 7 (2013), no. 4, 837-855.





\bibitem{KLW} M. Kowalczyk, Y. Liu and J. Wei,  Singly periodic solutions of the Allen-Cahn equation and the Toda lattice. {\em Comm. Partial Differential Equations} 40 (2015), no. 2, 329�356.

\bibitem{KLPW} M. Kowalczyk, Y. Liu, F. Pacard and J. Wei, End-to-end construction for the Allen-Cahn equation in the plane. {\em Calc. Var. Partial Differential Equations } 52 (2015), no. 1-2, 281-302.


\bibitem{Li-Wang}
P. Li and J. Wang, Minimal hypersurfaces with finite index, {\em Math.
Res. Lett. } 9 (2002), no. 1, 95-103.

\bibitem{Li-Wang 2}
P. Li and J. Wang, Stable minimal hypersurfaces in a nonnegatively
curved manifold, {\em J. Reine Angew. Math.} 566 (2004), 215-230.


\bibitem{LWW} Y. Liu, K. Wang, J. Wei, Global minimizers of the
Allen-Cahn equation in dimension $n\geq8,$ to appear in {\em Journal de
Mathematiques Pures et Appliquees}.



\bibitem{Meeks}
W. H. Meeks III and J. P\'{e}rez, A survey on classical minimal
surface theory, University Lecture Series 60, American Mathematical
Society, Providence, RI, 2012.

\bibitem{Modica}
L. Modica, A gradient bound and a Liouville theorem for nonlinear
Poisson equations, {\em  Comm. Pure Appl. Math.} 38 (1985), no. 5, 679-684.

\bibitem{Modica 2}
L. Modica, The gradient theory of phase transitions and the minimal
interface criterion, {\em Arch. Ration. Mech. Anal.} 98 (1987), no. 3,
123-142.

\bibitem{PW} F. Pacard, J. Wei, Stable solutions of the Allen-Cahn
equation in dimension 8 and minimal cones.{\em  J. Funct. Anal.} 264 (2013), no. 5, 1131--1167.

\bibitem{Polacik-Q-S}
P. Pol\'{a}cik, P. Quittner and P. Souplet, Singularity and decay
estimates in superlinear problems via Liouville-type theorems, I:
Elliptic equations and systems, {\em Duke Math. J.} 139 (2007), no. 3,
555-579.


\bibitem{Ros}
A. Ros, One-sided complete stable minimal surfaces. {\em J. Differential Geom.} 74 (2006), no. 1, 69-92.


\bibitem{ros-r-s}
A. Ros, D. Ruiz and P. Sicbaldi, A rigidity result for
overdetermined elliptic problems in the plane, to appear in {\em Comm. Pure Appl. Math.}

\bibitem{Savin1} O. Savin, Regularity of flat level sets in phase transitions. {\em  Ann. of Math.} (2) 169 (2009),
no.1, 41-78.

\bibitem{Savin-V}
O. Savin and E. Valdinoci, Some monotonicity results for minimizers
in the calculus of variations, {\em J. Funct. Anal.} 264 (2013), no. 10, 2469-2496.


\bibitem{SSV} B. Sciunzi, O. Savin  and E. Valdinoci, Flat level set regularity of p-Laplace phase transitions.
{\em Mem. Amer. Math. Soc. }182 (2006), no. 858, vi+144pp.


\bibitem{Schoen}
R. Schoen, Estimates for stable minmal surfaces in three dimensional
manifolds, {\em Ann. of Math. Stud.}, vol. 103, Princeton Univ. Press,
Princeton, NJ, 1983.

\bibitem{Simons}
J. Simons, Minimal varieties in Riemannian manifolds. Ann. of Math. (2) 88 (1968), 62-105.


\bibitem{S-Z}
P. Sternberg and K. Zumbrun, Connectivity of phase boundaries in
strictly convex domains, {\em Arch. Rational Mech. Anal. } 141 (1998),
375-400.


\bibitem{Tonegawa}
Y. Tonegawa, On stable critical points for a singular perturbation
problem, {\em Communications in Analysis and Geometry} 13 (2005), no. 3,
439-460.


\bibitem{Tonegawa 3}
Y. Tonegawa, Analysis on the mean curvature flow and the reaction-diffusion approximation, lectures given at Summer School at Fourier Institute on Geometric measure theory and calculus of variations, https://if-summer2015.sciencesconf.org/resource/page/id/8.



\bibitem{Tonegawa 2}
Y. Tonegawa and N.  Wickramasekera, Stable phase interfaces in the van der Waals-Cahn-Hilliard theory. {\em J. Reine Angew. Math.}  668 (2012), 191-210


\bibitem{Wang}
K. Wang, A new proof of Savin's theorem on Allen-Cahn equations, to appear in {\em J. Eur. Math. Soc..}


\bibitem{Wang 2}
K. Wang, The structure of finite Morse index solutions of two phase transition models in $\R^2$, {\em arXiv preprint
arXiv:1506.00491}.


\bibitem{Wang 3}
K. Wang, Some remarks on the structure of finite Morse index solutions to the Allen-Cahn equation in $\R^2$, {\em arXiv preprint
 arXiv:1506.00499}.


\bibitem{Wick}
N. Wickramasekera, A general regularity theory for stable
codimension 1 integral varifolds,  {\em Ann. of Math. } 179 (2014), no. 3,
843-1007.


\end{thebibliography}
\end{document}